\documentclass[reqno,letter, 11pt]{article}
\usepackage{color}
\usepackage[latin1]{inputenc}
\usepackage[english]{babel}
\usepackage{amsmath,amssymb,amsfonts,amsthm,enumerate}
\usepackage{mathrsfs}
\usepackage[T1]{fontenc}
\usepackage{epsf, epsfig}

\usepackage{hyperref}

\usepackage{epstopdf}  %YHK
\usepackage{cite} %YHK

\usepackage{graphicx}

\definecolor{refkeybis}{gray}{.65}% per avere le labels stampate chiare
\definecolor{labelkeybis}{gray}{.65}% basta modificare .50:= intensita grigio
{\makeatletter
\def\SK@refcolor{\color{refkeybis}}%
\def\SK@labelcolor{\color{labelkeybis}}}

\numberwithin{equation}{section} % per cambiare la numerazione
\newtheorem{theorem}{Theorem}[section]
\newtheorem{lemma}[theorem]{Lemma}
\newtheorem{definition}[theorem]{Definition}
\newtheorem{remark}[theorem]{Remark}
\newtheorem{proposition}[theorem]{Proposition}
\newtheorem{corollary}[theorem]{Corollary}

\newcommand{\q}{q}
\newcommand{\Q}{Q}
\newcommand{\N}{\mathbf{N}}
\newcommand{\R}{\mathbf{R}}

\newcommand{\Haus}[1]{{\mathscr H}^{#1}} % Misura di Hausdorff
\newcommand{\Leb}[1]{{\mathscr L}^{#1}} % Misura di Hausdorff

\newcommand{\KK}{\mathscr{K}}

\newcommand{\x}{\times}

\renewcommand{\>}{\rangle}
\renewcommand{\a}{\alpha}

\newcommand{\cl}{\overline}
\renewcommand{\d}{\delta}

\newcommand{\e}{\varepsilon}

\renewcommand{\l}{\lambda}

\renewcommand{\i}{\infty}
\newcommand{\p}{\partial}

\newcommand{\supp}{\operatorname{spt}}
\newcommand{\intr}{\operatorname{int}}
\newcommand{\dist}{\operatorname{dist}}
\newcommand{\diam}{\operatorname{diam}}
\newcommand{\co}{\operatorname{co}}

\newcommand\Bzero{{\bf (B0)}}
\newcommand\Bone{{\bf (B1)}}
\newcommand\Btwo{{\bf (B2)}}
\newcommand\Btwos{{\bf (B2)$_{\bf s}$}}
\newcommand\Bthree{{\bf (B3)}}
\newcommand\Bthrees{{\bf (B3)}$_{\bf s}$}
\newcommand\Bfour{{\bf (B4)}}
\newcommand\Athree{{\bf (A3)}}

\newcommand\Athreew{{\bf (A3)}$_{\bf w}$}
\newcommand\cross{\hbox{\rm cross}}
\newcommand\dom{\hbox{\rm dom}\thinspace}
\newcommand\zero{{\mathbf 0}}

\newcommand\U{U}
\newcommand\V{V}
\newcommand\uu{u}
\newcommand\tu{{\tilde u}}
\newcommand\vv{v}
\newcommand\bo{{b^0}}
\newcommand\tb{{\tilde b}}
\newcommand\cs{{c^*}}
\newcommand\tc{{\tilde c}}

\newcommand\bq{{\bar q}}
\newcommand\tq{{\tilde q}}
\newcommand\qo{q^0}
\newcommand\bs{{\bar s}}
\newcommand\bx{{\bar x}}
\newcommand\xo{{x_0}}
\newcommand\by{{\bar y}}
\newcommand\tx{{\tilde x}}
\newcommand\ty{{\tilde y}}

\newcommand\normal{{\hat n}}
\newcommand\DASM{{\bf DASM}}

\newcommand\red{\color{black}}
\newcommand\blue{\color{black}}
\newcommand\green{\color{black}}
\newcommand\RJM{\color{black}}
\newcommand\YHK{\color{black}}
\newcommand\RM{\color{black}}
\newcommand\rdot{{\RJM}}
\renewcommand{\nabla}{D}

\title{H\"older
continuity and injectivity of optimal maps\thanks{{\RJM An earlier
draft of this paper was circulated under the title {\em Continuity
and injectivity of optimal maps for non-negatively cross-curved
costs} at {\tt arxiv.org/abs/0911.3952}}
The authors are %pleased to thank %the Institute for Advanced Study in Princeton,
grateful to
the Institute for Pure and Applied Mathematics at UCLA,
the Institut Fourier at Grenoble, and the Fields Institute in Toronto,
for their generous hospitality during various stages of this work.
A.F. is supported {\blue in part} by NSF grant DMS-0969962.
Y-H.K. is supported partly by %United States National Science Foundation Grant
NSF grant DMS-0635607
through the membership at Institute for Advanced Study at Princeton NJ, and also in part by NSERC grant 371642-09.
{\blue R.J.M.}  is supported in part by
NSERC grants 217006-03 and -08 and %United States National Science Foundation
NSF grant DMS-0354729.
Any opinions, findings
and conclusions or recommendations expressed in this material are those of authors and do not
 reflect the views of either the Natural Sciences and Engineering
Research Council of Canada (NSERC) or the United States National Science Foundation (NSF).
\copyright 2011 by the authors.
}}
\date{\today}
\author{Alessio Figalli\thanks{Department of Mathematics, University of Texas at Austin, Austin TX USA {\tt figalli@math.utexas.edu}},
Young-Heon Kim\thanks{ Department of Mathematics, University of
British Columbia, Vancouver BC Canada {\tt yhkim@math.ubc.ca} and
School of Mathematics, Institute for Advanced Study at Princeton
NJ USA} \ and Robert J. McCann\thanks{Department of Mathematics,
University of Toronto, Toronto Ontario Canada M5S 2E4 {\tt
mccann@math.toronto.edu}}}

\begin{document}

\maketitle

\begin{abstract}
Consider transportation of one distribution of mass onto another,
chosen to optimize the total expected cost, where cost per unit mass
transported from $x$ to $y$ is given by a smooth function $c(x,y)$.
If the source density $f^+(x)$ is bounded away from zero and
infinity in an open region $\U' \subset \R^n$,  and the target
density $f^-(y)$ is bounded away from zero and infinity on its
support $\cl \V \subset \R^n$, which is strongly $c$-convex with
respect to $\U'$, and the transportation cost $c$ satisfies the
\Athreew\ condition of Trudinger and Wang \cite{TrudingerWang07p},
we deduce local H\"older continuity and injectivity of the optimal
map inside $\U'$ (so that the associated potential $u$ belongs to
$C^{1,\alpha}_{loc}(\U')$). Here the exponent $\alpha>0$ depends only on the dimension
and the bounds on the densities, but not on $c$.
%
%{\red The \Athreew\ condition on the cost and $c$-convexity of the
%target are sharp conditions, since they are necessary for continuity
%of optimal maps, as shown %respectively %by Loeper
%in \cite{Loeper07p} and %Ma, Trudinger and Wang
%\cite{MaTrudingerWang05}, respectively. }
%
%\footnote{\red RM: I suggest deleting the sentence in red from the
%abstract and starting%a new paragraph (and have already suppressed the names from this
%sentence)}
%
{\blue Our} result provides a crucial step in the low/interior
regularity setting: in a {\RJM sequel}~\cite{FigalliKimMcCann09p},
we use it to establish regularity of optimal maps with respect to
the Riemannian distance squared on arbitrary products of spheres.
{\RJM Three \blue key tools are introduced in {\RJM the present}
paper. Namely, we first find a transformation that under \Athreew \
make{\RJM s} $c$-convex functions level-set convex ({\RJM as} was
also {\RJM obtained} independently from us by Liu \cite{Liu09}). We
then derive new Alexandrov type estimates for the level-set convex
$c$-convex functions, and a topological lemma showing optimal maps
do not mix interior with boundary. This topological lemma, which
does not require \Athreew, is needed by} Figalli and Loeper \cite{FigalliLoeper08p} to
conclude {\RJM continuity of optimal maps in {\blue two dimensions}. }
%under
%\Athreew\ and assumptions on the densities {\RJM which are weaker
%than those assumed below (but} specific to dimension
%two)~\cite{FigalliLoeper08p}.
%
%%\footnote{RM: deleted \RJM ``some weaker''}
%
In higher dimensions, if the densities $f^\pm$ are H\"older
continuous,  our result permits
%second differentiability of the potentials
%continuous differentiability of the map %(i.e. $u \in C^{2,\alpha}_{loc}(\U')$)
continuous differentiability of the map inside $U'$ (in fact, $C^{2,\alpha}_{loc}$ regularity of the associated
potential) to be deduced from the work of Liu, Trudinger and
Wang~\cite{LiuTrudingerWang09p}.
%$C^{1,\alpha}_{loc}(\U')$ interior regularity
%of the associated potential,
%for some $\alpha$ depending on the
%dimension and on bounds for the ratio $f^+/f^-$, but independent
%of~$c$.
\end{abstract}

\tableofcontents

\section{Introduction}

Given probability densities $0 \le f^\pm \in L^1(\R^n)$ with respect to Lebesgue measure
$\Leb{n}$ on $\R^n$,
and a cost function $c:\R^n \times \R^n \longmapsto [0,+\infty]$,
Monge's transportation problem is to find a map $G:\R^n \longmapsto \R^n$
pushing $d\mu^+ = f^+ d\Leb{n}$ forward to $d\mu^-= f^-d\Leb{n}$
which minimizes the expected transportation cost \cite{Monge81}
%integral of the transportation cost $c(x,F(x))$ against $\mu$: i.e.
\begin{equation}\label{Monge}
\inf_{G_\#\mu^+ = \mu^-} \int_{\R^n} c(x,G(x)) d\mu^+(x),
\end{equation}
where $G_\# \mu^+ = \mu^-$ means $\mu^-[Y] = \mu^+[G^{-1}(Y)]$ for
each Borel $Y \subset \R^n$.

In this context it is interesting to know when a map attaining this infimum exists;
sufficient conditions for this were found by
Gangbo \cite{Gangbo95} and by Levin \cite{Levin99}, extending work
of a number of authors described in \cite{GangboMcCann96} \cite{Villani09}.
%Brenier, Rachev and Ruschendorf, Abdellaoui and Heinich, Cuesta-Albertos, Matran and Tuero-Diaz
%Knott and Smith, Cullen and Purser; Sudakov (Dowson and Landau, Olkin and Pukelsheim, Givens and Short)
%Caffarelli, Gangbo, McCann
One may also ask when $G$ will be smooth, in which case it must satisfy the prescribed
Jacobian equation $|\det DG(x)|=f^+(x)/f^-(G(x))$,
which turns out to reduce to a degenerate elliptic
partial differential equation of Monge-Amp\`ere type %(\ref{Monge-Ampere type equation})
for a scalar potential $\uu$ satisfying $D\uu(\tx)=-D_x c(\tx,G(\tx))$.
Sufficient conditions for this were discovered by
Ma, Trudinger and Wang \cite{MaTrudingerWang05} and
Trudinger and Wang \cite{TrudingerWang07p} \cite{TrudingerWang08p},
after results for the special case $c(x,y)= |x-y|^2/2$ had been
worked out by Brenier \cite{Brenier91}, Delan\"oe \cite{Delanoe91},
Caffarelli \cite{Caffarelli90p} \cite{Caffarelli90} \cite{Caffarelli91} \cite{Caffarelli92} %\cite{Caffarelli92b}
\cite{Caffarelli96b}, and Urbas \cite{Urbas97},  and for the
cost $c(x,y)=-\log |x-y|$ and measures supported on the unit sphere by
Wang~\cite{Wang96}. % \cite{Wang04}.

If the ratio $f^+(x)/f^-(y)$ --- although bounded away from zero
and infinity --- is not continuous, the map $G$ will not generally
be differentiable, though one may still hope for it to be
continuous. This question is not merely of {\RJM technical interest},
since discontinuities in $f^\pm$ arise unavoidably in applications
such as partial transport problems \cite{CaffarelliMcCann99}
\cite{BarrettPrigozhin09} \cite{FigalliARMA} \cite{FigalliNote}.
Such results were established for the classical cost
$c(x,y)=|x-y|^2/2$ by Caffarelli \cite{Caffarelli90}
\cite{Caffarelli91} \cite{Caffarelli92}, for its restriction to
the product of the boundaries of two strongly convex sets by
Gangbo and McCann \cite{GangboMcCann00}, and for more general
costs satisfying the strong regularity hypothesis \Athree\ of
Ma, Trudinger and Wang \cite{MaTrudingerWang05}
--- which excludes the cost $c(x,y)=|x-y|^2/2$ --- by Loeper
\cite{Loeper07p}; see also \cite{KimMcCannAppendices} \cite{Liu09}
\cite{TrudingerWang08p}. Under the weaker
%\footnote{RM: Deleted {\RJM and degenerate}, which I think
%means something specific to YHK but not to the general readership
%AF: I agree with Robert}
hypothesis \Athreew\ of Trudinger and Wang
\cite{TrudingerWang07p}, which includes the cost
$c(x,y)=|x-y|^2/2$ (and whose necessity for regularity was shown
by Loeper \cite{Loeper07p}, see also \cite{FRV-reg}), such a result remained absent from the
literature; the aim of this paper is to fill this gap, see
Theorem~\ref{T:Hoelder} below.

{ A number of interesting cost functions do
satisfy hypothesis \Athreew, and have
applications in economics \cite{FigalliKimMcCann-econ09p} and
statistics \cite{Sei09p}. Examples include the Euclidean distance {\blue squared}
between two convex graphs over two sufficiently convex sets in
$\R^n$ \cite{MaTrudingerWang05}, the simple harmonic oscillator action \cite{LeeMcCann},
and the Riemannian distance {\blue squared} on the following spaces:
{\RJM the round sphere \cite{Loepersphere}
and perturbations thereof} \cite{DelanoeGe} \cite{FigalliRifford08p}\cite{FRV-Sn},
multiple products of round spheres (and their Riemannian submersion quotients, including products of
complex projective spaces) \cite{KimMcCann08p}, {\blue and products of perturbed $2$-dimensional spheres \cite{DelanoeGe10}}.
%{\blue Note that for {\RJM
{\RJM As remarked in \cite{KimMcCann07p}, for graphs which fail to be strongly convex and
all Riemannian product geometries,
the stronger condition \Athree\ necessarily fails.}
%\footnote{RM: deleted {\RJM paper}}
%{\YHK
%In fact, all of the above examples (except the perturbations of sphere of dimension three and higher) are known to satisfy non-negative cross-curvature condition
%(i.e. \Bfour, see Section \ref{S:main-result}), but
%there are important examples not satisfying such condition such as $c(x, y) = - \log |x-y|$ that satisfies \Athree, and $c(x, y) = | x - y|^{-2}$ that satisfies \Athreew\  (but not \Athree)  as communicated to us by Trudinger.
%\footnote{YHK: I added this example. Alessio, please check it.
%AF: I'd rather remove these examples. Indeed in the case of (A3) there are better regularity results available,
%so I'm not sure if mentioning them really plays in our favor...
%{\YHK YHK. No, non A3s nor NNCC but A3w example such as $|x-y|^{-2}$ should be good for us to have..
%I also have edited the text a little bit.}.
%AF: I'm not sure that $|x-y|^{-2}$ is a good example: apart from the fact that I do not see an interest in this cost,
%if I remember correctly this cost does satisfies A3 on bounded sets (but the A3 condition degenerates at infinity)}
In a sequel,  we apply the techniques developed here to deduce
regularity of optimal maps {\blue for the multiple products of round spheres}
\cite{FigalliKimMcCann09p}. Moreover, Theorem~\ref{T:Hoelder}
allows one to apply the higher interior regularity results established
by Liu, Trudinger and Wang \cite{LiuTrudingerWang09p},
ensuring in particular that the transport map is $C^\infty$-smooth
 if $f^+$ and $f^-$ are.
\\}
%\footnote{All of the above examples actually satisfy NNCC, i.e. \Bfour!  YHK: so I added the examples.
%AF: to be precise, it is not known if perturbations of the $n$-dimensional sphere satisfy (NNCC)! So I'd juste leave as
%it was, removing YH examples.}

Most of the regularity results quoted above derive from one of two
approaches. The continuity method, used by Delano\"e, Urbas, Ma,
Trudinger and Wang,  is a time-honored technique for solving
nonlinear equations.  Here one perturbs a manifestly soluble
problem (such as $|\det DG_0(x)|=f^+(x)/f_0(G_0(x))$ with $f_0=f^+$, so that $G_0(x)=x$) to the problem of interest
($|\det DG_1(x)|=f^+(x)/f_1(G_1(x))$, $f_1=f^-$)
along a family $\{f_t\}_t$ designed to ensure the set of $t
\in[0,1]$ for which it is soluble is both open and closed.
Openness follows from linearization and non-degenerate ellipticity
using an implicit function theorem. For the non-degenerate ellipticity and closedness,
it is required to establish estimates on the size of derivatives of the solutions
(assuming such solutions exist) which depend only on information
known a priori about the data $(c,f_t)$.  In this way one obtains
smoothness of the solution $y=G_1(x)$ from the same argument which
shows $G_1$ to exist.

The alternative approach relies on first knowing existence and
uniqueness of a Borel map which solves the problem in great
generality,  and then deducing continuity or smoothness by close
examination of this map after imposing additional conditions on
the data $(c,f^\pm)$. Although precursors can be traced back to
Alexandrov \cite{Aleksandrov42b}, in the present context this
method was largely developed and refined by Caffarelli
\cite{Caffarelli90} \cite{Caffarelli91} \cite{Caffarelli92}, who
used convexity of $\uu$ crucially to localize the map $G(x) =
D\uu(x)$ and renormalize its behaviour near a point $(\tx,G(\tx))$
of interest in the borderline case $c(x,y) = - \<x, y\>$. For
non-borderline \Athree\ costs, simpler estimates suffice to
deduce continuity of $G$, as in \cite{GangboMcCann00}
\cite{CafGutHua} \cite{Loeper07p} \cite{TrudingerWang08p}; in this
case Loeper was actually able to deduce an explicit bound $\alpha
= (4n-1)^{-1}$ on the H\"older exponent of $G$ when $n>1$. {\RJM This} bound
was recently improved to its sharp value $\alpha = (2n-1)^{-1}$ by
Liu \cite{Liu09}, using {\red a key observation discovered independently from us (see Section~\ref{S:transform} and  Theorem~\ref{thm:apparently convex})};
%a technique related to the one we develop below
both Loeper and Liu
also obtained explicit exponents $\alpha=\alpha(n,p)$ for
 $f^+ \in L^p$ with $p>n$ \cite{Loeper07p} or $p> (n+1)/2$ \cite{Liu09} and $1/f^- \in L^\infty$.
 For the classical case $c(x,y) = - \<x, y\>$,
 explicit bounds were found by Forzani and Maldonado \cite{ForzaniMaldonado04},
% but they are much worse, and they {\RJM are not even finite}
% unless
{\YHK depending on}
 $\log \frac{f^+(x)}{f^-(y)} \in L^\infty$ \cite{Caffarelli92} \cite{Wangcounterex}.\\

%Below we combine the approach of Caffarelli \cite{Caffarelli90} \cite{Caffarelli91} \cite{Caffarelli92}
%with the strategy of Forzani and Maldonado \cite{ForzaniMaldonado04} to costs satisfying the \Athreew\ condition,
%a class which includes the classical quadratic cost. %; we obtain the same value of $\alpha>0$ for all such costs.
%\footnote{RM: Deleted {\RJM main}}
{\RJM Our proof introduces at least three significant new tools. Its starting point
is that} {\blue condition \Athreew \ allows one to add a null Lagrangian term to the cost function
and exploit diffeomorphism (i.e.\ gauge) invariance to choose coordinates, which depend on the
point of interest, to  transform the $c$-convex functions to level-set convex functions
(see Theorem~\ref{thm:apparently convex}); this observation was made also by Liu \cite{Liu09},
independently from us.  {\RJM Next we establish Alexandrov} type estimates
(Theorems~\ref{thm:lower Alex} and \ref{thm:estimate}) for $c$-convex functions
{\RJM whose level sets are convex,
 extending classical estimates} for convex functions.
{\RJM These % passage from \Bfour\ to \Athreew\ costs
rely
 on quantitative new aspects of the geometry of convex sets
 that we derive elsewhere, and which may have independent interest \cite{FKM-convex}.}
The {\RJM resulting Alexandrov type estimates} enable us to exploit
Caffarelli's approach \cite{Caffarelli90} \cite{Caffarelli91}
\cite{Caffarelli92} more systematically than Liu \cite{Liu09} was
able to do, to prove continuity and injectivity (see
Theorems~\ref{T:strict convex} and  \ref{T:continuity}). Once such
results are established, {\RJM the same estimates permit} us to
exploit {\RJM Forzani and Maldonado's \cite{ForzaniMaldonado04}}
approach to extend the engulfing property of Gutierrez and Huang
\cite{guthua} to \Athreew \ $c$-convex functions (see
Theorem~\ref{T:engulfing}), improving {\RJM mere} continuity of
optimal maps to H\"older continuity (Theorem~\ref{T:C1alpha} and its
Corollary~\ref{C:universal Holder exponent}).}
%Our idea is to add a null Lagrangian term
%to the cost and exploit diffeomorphism (i.e.\ gauge) invariance
%%\cite{KimMcCann07p} and the geometry of the problem
%to choose coordinates which depend on the point of interest that
%provide the level-set convexity of $\uu(x)$; under our hypothesis, we are able to exploit Caffarelli's approach more systematically
%than Liu was able to do \cite{Liu09} to prove continuity and injectivity,
%and then we exploit the approach of Forzani and Maldonado to improve the result to H\"older regularity.
%However, we still need to
%overcome serious difficulties, such as getting an Alexandrov
%estimate for $c$-subdifferentials (see Section~\ref{S:Alexandrov})
%and dealing with the fact that the domain of the cost function
{\blue  Along the way, we also have to overcome another serious difficulty, namely the fact that the domain of the cost function}
(where it is smooth and satisfies appropriate cross-curvature
conditions) may not be the whole of $\R^n$. (This situation
arises, for example, when optimal transportation occurs between
domains in Riemannian manifolds for the distance squared cost or
similar type.)
%The former require also new geometric estimates for convex sets \cite{FKM-convex}, which we believe to be of independent interest.
%The latter is accomplished using
{\blue This is handled by using} Theorem~\ref{thm:bdry-inter}, where
it is first established that optimal transport does not send
interior points to boundary points, and vice versa, under the strong
$c$-convexity hypothesis \Btwos\ described in the next section.
{\RJM Theorem~\ref{thm:bdry-inter} does not require \Athreew,
however.}
% (For this result to hold, the cost needs not to satisfy the condition \Athreew.)
Let us point out that, in two dimensions, there is an alternate
approach to establishing continuity of optimal maps;  it was carried out
by Figalli and
Loeper \cite{FigalliLoeper08p} {\blue following Alexandrov's strategy \cite{Aleksandrov42b},
but their result relies on {\RJM our} Theorem~\ref{thm:bdry-inter}.}
%, first proved below. %topological
%result first established in Theorem \ref{thm:bdry-inter} below,
%which asserts optimal transport maps interiors to interiors and
%boundaries to boundaries under the strong $c$-convexity hypothesis
%\Btwos\ described in the next section.

%--- however, there are \Athreew\ but not nonnegatively cross-curved
%costs (e.g. $|x-y|^{-2}$) as communicated to us by Trudinger.
%\footnote{please check this cost} In this paper we consider
%transportation on local domains such as bounded domains, contained
%in $\R^n$ or local coordinate charts of manifolds.  In a sequel to
%this paper \cite{FigalliKimMcCann09p}, we use the method and
%results obtained here to show regularity of transportation of
%densities supported on global domains such as multiple products of
%round spheres with nonnegatively cross-curved costs.

%\footnote{What continuity result do Trudinger-Wang get in \cite{TrudingerWang08p}?}
%Motivation:
%new results for rough data and a source which may vanish in regions;
%enables Caffarelli's techniques to be adapted to fairly general costs;
%required to address partial transport for non \Athrees\ costs;
%mention LIU, Trudinger, Wang
%\subsection{Remark about method}
%We follow the strategy of Caffarelli \cite{caff-loc,caffC1a,caff}.
%The non-negative cross-curvature condition makes it possible to generalize his method,
%especially due to the change of variable of $c$-convex functions to convex functions
%described in Section~\ref{S:transform}.

\section{Main result}\label{S:main-result}

Let us begin by formulating the relevant hypothesis on the cost
function $c(x,y)$ in a slightly different format than Ma,
Trudinger and Wang \cite{MaTrudingerWang05} \cite{TrudingerWang07p}.
We denote their condition \Athreew\ as \Bthree\ below,  and their stronger condition
 \Athree\ as \Bthrees.
Throughout the paper, $D_y$ will denote the derivative with respect to the variable $y$,
and iterated subscripts as in $D^2_{xy}$ denote iterated derivatives.
For each $(\tx,\ty)
\in \cl \U \times \cl \V$
%define $\U_\ty$ and $\V_\tx$ as in \Btwo\
%$\U_{\ty} := -D_y c(\U, \ty)$ and $\V_{\tx} := -D_x c(\tx, \V)$ and
{\YHK we define the following conditions}:\\

\noindent
\Bzero\  $\U \subset \R^n$ and $\V \subset \R^n$ are open and bounded and
$c \in C^4\big(\cl \U \times \cl \V\big)$; \\
%For each $(x_0,y_0) \in X_0 \times Y_0$ we assume that \\
\Bone\   (bi-twist)
$\left.\begin{array}{c}
        x \in \cl \U \longmapsto -D_y c(x,\ty) %\in \cl \U_{\ty}:= -D_y c(\U, \ty)
\cr     y \in \cl \V \longmapsto -D_x c(\tx,y) %\in \cl \V_{\tx}:= -D_x c(\tx, \V)
        \end{array}\right\}$ are diffeomorphisms onto their ranges; \\
%Note \Bone\   is strong form of the Spence-Mirrlees single crossing-condition
%required in \cite{Gangbo95} \cite{Levin99},
%which asserts only the injectivity (i.e.~invertibility) of these maps.
\Btwo\  (bi-convex)
%\footnote{RJM: We use the smoothness part of \Btwo\ only in Corollary \ref{C:c-ma less ma} for
%approximation (in an inessential way, I believe); perhaps this assumption can be eliminated by
%restoring a revised version of the original proof of Corollary \ref{C:c-ma less ma}?}
$\left.\begin{array}{c}
        \U_{\ty} := -D_y c(\U, \ty) \cr
        \V_{\tx} := -D_x c(\tx, \V)
        \end{array}\right\}$ are convex subsets of $\R^n$; \\
%We also assume\\
{\RJM \Bthree\ (=\Athreew)}\
for every curve
$t \in[-1,1] \longmapsto \big(D_y c(x(t),y(0)),D_x c(x(0),y(t))\big) \in \R^{2n}$
which is an affinely parameterized line segment,
\begin{equation}\label{MTW}
\cross_{(x(0),y(0))} [x'(0),y'(0)] :=
-\frac{\partial^4}{\partial s^2 \partial t^2}\bigg|_{(s,t)=(0,0)} c(x(s),y(t)) \ge 0
\end{equation}
provided
\begin{equation}\label{nullity}
\frac{\partial^2}{\partial s \partial t}\bigg|_{(s,t)=(0,0)} c(x(s),y(t)) = 0.
\end{equation}
\\

From time to time we may strengthen these hypotheses by writing either:\\

\noindent
{\YHK \Btwos\ if the convex domains $\U_\ty$ and $\V_\tx$ in \Btwo\ are
strongly convex;}\\
{\YHK \Bthrees\ (=\Athree)\ if, in the condition \Bthree, the inequality \eqref{MTW} is strict;\\
\Bfour\ if, in the condition \Bthree,  \eqref{MTW} holds even in the absence of} the extra assumption \eqref{nullity}.
\\

Here a convex set $Q
\subset \R^n$ is said to be {\em strongly} convex if there exists
a radius $R<+\infty$ (depending only on $Q${\RJM ),}
 such that each
boundary point $\tx \in
\partial Q$ can be touched from outside by a sphere of radius $R$
enclosing $Q$; i.e. $Q \subset B_R\left(\tx - R \normal_{Q}(\tx)\right)$
where $\normal_Q(\tx)$ is an outer unit normal to a hyperplane
supporting $Q$ at $\tx$.  When $Q$ is smooth, this means all
principal curvatures of its boundary are bounded below by $1/R$.

Hereafter $\cl \U$ denotes the closure of $\U$, $\intr U$ denotes
its interior, $\diam U$ its diameter, and for any measure $\mu^+
\ge 0$ on $\cl \U$, we use the term {\em support} and the notation
$\supp \mu^+ \subset \cl\U$ to refer to the smallest closed set
carrying the full mass of $\mu^+$.

Condition \Athreew\ {\blue ($=$\Bthree) was used by Trudinger and Wang to  show} smoothness
of optimal maps in the Monge transportation problem (\ref{Monge}) when the densities are smooth.
Necessity of Trudinger and Wang's condition for continuity was shown by
Loeper \cite{Loeper07p}, who
noted its covariance (as did \cite{KimMcCann07p} \cite{Trudinger06})
and some relations to curvature.  Their condition relaxes the hypothesis \Athree\
($=$\Bthrees) proposed
earlier with Ma \cite{MaTrudingerWang05}.
{\blue In \cite{KimMcCann07p}, Kim and McCann showed that
the expressions \eqref{MTW} and \eqref{nullity}}
correspond to pseudo-Riemannian sectional curvature conditions induced by the cost $c$
on $\U \times \V$,  highlighting their invariance under reparametrization
of either $\U$ or $\V$ by diffeomorphism; see \cite[Lemma 4.5]{KimMcCann07p}; {\blue  see also \cite{KimMcCannWarren09} as well as \cite{HarveyLawson-Split}, for further investigation of the pseudo-Riemannian aspects of optimal maps.}
The convexity of $\U_\ty$ required in \Btwo\  is called {\em
$c$-convexity of $U$ with respect to $\ty$} by Ma, Trudinger and
Wang (or strong $c$-convexity if \Btwos\ holds); they call curves
$x(s) \in \U$, for which $s \in [0,1] \longmapsto -D_y
c(x(s),\ty)$ is a line segment, {\em $c$-segments with respect to
$\ty$}. Similarly, $\V$ is said to be strongly $\cs$-convex with
respect to $\tx$
--- or with respect to $\cl\U$ when it holds for all $\tx \in \cl\U$ ---
and the curve $y(t)$ from \eqref{MTW} is said to be a
$\cs$-segment with respect to $\tx$.
% if {\color{red}$t \in [0,1] \longmapsto -D_x c(\tx,y(t))$} is a line segment.
%   though the star on $\cs(y,x):=c(x,y)$ may be dropped when confusion cannot arise.
Such curves correspond to geodesics $(x(t),\ty)$ and $(\tx, y(t))$
in the geometry of Kim and McCann. Here and throughout,  {\em line segments}
are always presumed to be affinely parameterized.

We are now in a position to summarize our main result:

\begin{theorem}[\blue Interior H\"older continuity and injectivity of optimal maps]
\label{T:Hoelder}
Let $c \in C^4 \big(\cl \U \times \cl \V\big)$ satisfy \Bzero--\Bthree\ and \Btwos.
Fix probability densities $f^+ \in L^1\big(\U\big)$ and $f^- \in L^1\big(\V\big)$
%\footnote{YH: Isn't $L^1$ bound the same with/without closure of the domain?}
%\footnote{RM: Only if $\Leb{n}(\partial \U)=0$,  which is true for $\U$,
%but not necessarily true for $\U'$; we have the option to remove the bars on $\U$ and $\V$,
%but need no charge on the boundary to apply Lemma \ref{L:cMA properties}}
with $(f^+/f^-) \in L^\infty\big(\U \x \V\big)$ and set $d\mu^\pm
:=f^\pm d\Leb{n}$. {\blue (Note that $\mathop{\rm spt} \mu^+$ may not be $c$-convex.)} If the ratio $(f^-/f^+) \in
L^\infty(\U' \x \V)$ for some open set $\U' \subset \U$ {\blue ($U'$ is not necessarily $c$-convex)},  then
the minimum (\ref{Monge}) is attained by a
map $G:\cl \U \longmapsto \cl \V$ whose restriction to $\U'$ is locally
H\"older continuous and one-to-one. {\blue Moreover, the H\"older exponent
depends only on $n$ and} {\RJM $\| \log (f^+/f^-)\|_{L^\infty(\U' \times \V)}$}.
%satisfying $G \in C^{\alpha}_{loc}(\U',\V)$, with a H\"older exponent $\alpha>0$
%depending only on $n$ and $\|\log (f^+/f^-)\|_{L^\infty(\U' \times \V)}$.
\end{theorem}

%\footnote{RJM: Any hope for a result assuming only $\log f^- \in L^\infty_{loc}(\V)$
%instead of $1/f^- \in L^\infty(\cl \V)$?}

\begin{proof}
As recalled below in Section \ref{S:background}
(or see e.g.\ \cite{Villani09}) %e.g.\ \cite{Gangbo95} \cite{Levin99} \cite{MaTrudingerWang05}),
it is well-known by Kantorovich duality that the optimal joint
measure $\gamma \in \Gamma(\mu^+,\mu^-)$ from \eqref{Kantorovich}
vanishes outside the $c$-subdifferential \eqref{c-subdifferential}
of a potential $u=u^{c^* c}$ satisfying the $c$-convexity
hypothesis \eqref{c-transform}, and that the map $G:\cl\U
\longmapsto \cl\V$ which we seek is uniquely recovered from this
potential using the diffeomorphism \Bone\ to solve \eqref{implicit
G}. Thus the H\"older continuity claimed in Theorem \ref{T:Hoelder} is
equivalent to $\uu \in C^{1,\alpha}_{loc}(\U')$.

{Since $\mu^\pm$ do not charge the boundaries of $\U$ (or of
$\V$),} Lemma \ref{L:cMA properties}(e) shows the
$c$-Monge-Amp\`ere measure defined in \eqref{c-Monge-Ampere
measure} has density satisfying $|\p^c u| \le
\|f^+/f^-\|_{L^\infty(\U \times \V)}$ on $\cl \U$ and
$\|f^-/f^+\|^{-1}_{L^\infty(\U' \times \V)} \le |\p^c
u| \le \|f^+/f^-\|_{L^\infty(\U' \times \V)}$ on $\U'$. Thus $u
\in C^{1,\a}_{loc}(\U')$ according to Theorem \ref{T:C1alpha}. Injectivity
of $G$ follows from Theorem \ref{T:strict convex}, and the fact
that the graph of $G$ is contained in the set $\p^c u \subset \cl
\U \times \cl \V$ of \eqref{c-subdifferential}.
{\blue The dependency of the H\"older exponent $\alpha$ only on $n$ and $\| \log (f^+/f^-)\|_{L^\infty(\U' \times \V)}$ follows by Corollary~\ref{C:universal Holder exponent}.}
\end{proof}
{ Note that in case $f^+ \in C_c(U)$ is continuous and compactly supported,
 choosing $U' = U'_\e= \{ f^+ > \e \}$ for all $\e > 0$,  yields local H\"older
 continuity and injectivity of the optimal map $y = G(x)$ throughout $U'_0$.}
%{\red For the classical costs $c(x,y) = \frac{1}{2}|x-y|^2$ and $c(x,y) = - x \cdot y$
%(which differ by an irrelevant term $\frac{1}{2}|x|^2 + \frac{1}{2}|y|^2$ called a null
%Lagrangian), our result has two disadvantages relative to
%Caffarelli's \cite{Caffarelli91} \cite{Caffarelli92}: apart from the strong
%convexity hypothesis \Btwos\ which he did not need,
%the lack of affine invariance of our estimate \eqref{vardist} prevents us from
%deducing a H\"older exponent for continuity of the map in our case.
%.}
% the H\"older exponent $\alpha = \alpha(\epsilon)$
%that we find naturally degenerates as we approach a point where $f^+$ vanishes.}

Theorem \ref{T:Hoelder} allows to extend the
higher interior regularity results established by Liu, Trudinger and
Wang in \cite{LiuTrudingerWang09p}, originally given for
\Athree\ costs, to {\RJM the
weaker and degenerate case \Athreew,
see \cite[Remark 4.1]{LiuTrudingerWang09p}.
 Note that these interior regularity results can be
applied to manifolds, after getting suitable stay-away-from-the-cut-locus
results: this is accomplished for multiple products of round spheres in
\cite{FigalliKimMcCann09p}, to yield the first regularity result that we
know for optimal maps on Riemannian manifolds which are not flat, yet
have some vanishing sectional curvatures.
%  as it yields a conclusion
%for all continuous sources $f^+ \ge 0$ on $\U$ and targets $f^- >0$ on $\V$.

Let us also point out that different
strengthenings of the \Athreew\  condition have been considered in
\cite{MaTrudingerWang05} \cite{KimMcCann07p} \cite{KimMcCann08p} \cite{LoeperVillani08p} \cite{FigalliRifford08p} \cite{FRV-Sn}.
In particular, one stronger condition is the so-called \textit{non-negative cross-curvature}, which
here is denoted here by
\Bfour.
Although not strictly needed for this paper, under the \Bfour\ condition
the cost exponential coordinates introduced in Section \ref{S:notation} allow to deduce stronger
conclusions with almost no extra effort,
and these results play a crucial role in the proof of the regularity of optimal maps on multiple products of spheres \cite{FigalliKimMcCann09p}.
For this reason, we prefer to include here some of the conclusions that one can deduce when \Athreew\
is replaced by \Bfour.}

\section{Background, notation, and preliminaries}
\label{S:background}

Kantorovich discerned \cite{Kantorovich42} \cite{Kantorovich48} that  Monge's problem
(\ref{Monge}) could be attacked by studying the linear programming problem
\begin{equation}\label{Kantorovich}
\min_{\gamma \in \Gamma(\mu^+,\mu^-)} \int_{\cl \U \times \cl \V}
c(x,y)\, d\gamma(x,y).
\end{equation}
Here $\Gamma(\mu^+,\mu^-)$ consists of the joint probability measures on
$\cl \U \times \cl \V \subset \R^n \times \R^n$ having $\mu^\pm$ for
marginals.  According to the duality theorem from linear programming,
the optimizing measures $\gamma$ vanish outside the zero set of
$\uu(x) + \vv(y) + c(x,y) \ge 0$ for some pair of functions
$(\uu,\vv) = (\vv^{c},\uu^\cs)$ satisfying
\begin{equation}\label{c-transform}
\vv^{c}(x) := \sup_{y \in \cl \V} -c(x,y) - \vv(y), \qquad
\uu^\cs(y) := \sup_{x \in \cl \U} -c(x,y) - \uu(x);
\end{equation}
these arise as optimizers of the dual program. This zero set is
called the $c$-subdifferential of $\uu$, and denoted by
\begin{equation}\label{c-subdifferential}
\partial^c \uu = \left\{(x,y) \in \cl \U \times \cl \V \mid \uu(x) + \uu^\cs(y) + c(x,y) = 0\right\};
\end{equation}
we also write $\partial^c \uu(x) := \{y \mid (x,y) \in \partial^c
\uu\}$, and $\partial^{\cs}\uu^\cs(y) := \{ x \mid (x,y) \in
\partial^c \uu\}$, and $\partial^c \uu(X) := \cup_{x \in
X} \partial^c \uu(x)$ for $X \subset \R^n$. Formula
(\ref{c-transform}) defines a generalized Legendre-Fenchel
transform called the {\em $c$-transform};  any function satisfying
$\uu = \uu^{\cs c} :=(\uu^\cs)^{c}$ is said to be {\em
$c$-convex}, which reduces to ordinary convexity in the case of
the cost $c(x,y) = - \<x, y\>$. In that case $\partial^c \uu$
reduces to the ordinary subdifferential $\partial \uu$ of the
convex function $\uu$,  but more generally we define
\begin{equation}\label{subdifferential}
\partial \uu
:= \{(x,p) \in \cl \U \times \R^n \mid \uu(\tx) \ge \uu(x) +
\langle p, \tx-x\rangle + o(|\tx-x|) {\rm\ as\ } \tx \to x\},
\end{equation}
$\partial\uu(x) := \{p \mid (x,p) \in \partial \uu\}$, and
$\partial \uu(X) := \cup_{x \in X} \partial \uu(x)$. Assuming $c
\in C^2\big(\cl \U \times \cl \V\big)$ (which is the case if
\Bzero\ holds),  any $c$-convex function $\uu=\uu^{c^*c}$ will be
semi-convex,  meaning its Hessian admits a bound from below $D^2
\uu \ge - \|c\|_{C^2}$ in the distributional sense;  equivalently,
$\uu(x) + \|c\|_{C^2}|x|^2/2$ is convex on each ball in $\U$
\cite{GangboMcCann96}.  In particular,  $\uu$ will be
twice-differentiable $\Leb{n}$-a.e. on $\U$ in the sense of
Alexandrov.

As in \cite{Gangbo95} \cite{Levin99} \cite{MaTrudingerWang05},
hypothesis \Bone\  shows the map $G:\dom D\uu \longmapsto \cl \V$ is uniquely
defined on the set $\dom D\uu \subset \cl\U$ of differentiability for $\uu$ by
\begin{equation}\label{implicit G}
D_x c(\tx,G(\tx)) = - D\uu(\tx).
\end{equation}
The graph of $G$,  so-defined,  lies in $\partial^c \uu$.
The task at hand is to show (local) H\"older continuity and injectivity of $G$ ---
the former being equivalent to $\uu \in C^{1,\alpha}_{loc}(\U)$ ---
by studying the relation $\partial^c \uu \subset \cl \U \times \cl \V$.

To this end,  we define a Borel measure $|\partial^c \uu|$ on $\R^n$ associated to
$\uu$ by
\begin{equation}\label{c-Monge-Ampere measure}
|\partial^c \uu|(X) := \Leb{n}(\partial^c \uu(X))
\end{equation}
for each {\blue Borel set} $X \subset \R^n$;
%:  the fact that this object is a Borel measure of total mass
%$\Leb{n}(\cl \V)$ follows from \Bzero--\Bone;
it will be called the  {\em $c$-Monge-Amp\`ere measure} of $\uu$.
{(Similarly, we define $|\partial u|$.)}
We use the notation $|\partial^c \uu| \ge  \lambda$ on $\U'$ as a shorthand
to indicate $|\partial^c \uu|(X) \ge \lambda \Leb{n}(X)$ for each $X \subset \U'$;
similarly, $|\partial^c \uu| \le \Lambda$ indicates
$|\partial^c \uu|(X) \le \Lambda \Leb{n}(X)$.
%When both hold we write $|\partial^c \uu| \in [\lambda,\Lambda]$ on $\U'$.
%\footnote{Inconsistency! Does this imply boundary conditions, or not?}
As the next lemma %{\blue (for which we do not claim originality)}
%\footnote{YH: I put this remark on originality just to be safe... RM: but its in a section titled `Background...'
%AF: I agree with Robert, I'd prefer to remove the remark, it looks to me a bit ``eccesive". YHK:  I don't mind either way. If you like you can change it.}
shows, uniform bounds above and below on the marginal densities
%$f \in L^1(\U)$ and $g \in L^1(\V)$
of a probability measure $\gamma$ vanishing outside $\partial^c \uu$
imply similar bounds on $|\partial ^c \uu|$.

\begin{lemma}[Properties of $c$-Monge-Amp\`ere measures]\label{L:cMA properties}
Let $c$ satisfy \Bzero-\Bone, while $u$ and $u_k$ denote
$c$-convex functions for each $k \in \N$. Fix $\tx \in \cl X$ and
constants $\lambda,\Lambda > 0$.
\\(a) Then
$\p^c u(\cl \U)\subset \cl \V$ and $|\partial^c u|$ is a
Borel measure of total mass $\Leb{n}\big(\cl V\big)$ on $\cl U$.
\\(b) If $u_k \to u_\infty$ uniformly,  then $u_\infty$ is $c$-convex and
$|\p^c u_k| \rightharpoonup |\p^c u_\infty|$ weakly-$*$ in
the duality against continuous functions on $\cl \U \times \cl \V$.
\\(c) If $u_k(\tx)=0$ for all $k$, then the functions $u_k$ converge uniformly if and only if
the measures $|\p^c u_k|$ converge weakly-$*$. % in $(C(\cl U \times \cl V),\|\cdot\|_\infty)^*$.
%\footnote{YH: Aren't (b) and (c) redundant? RM: No, (b) contains only one of the two implications in (c).
%I suppose one could drop (b),  if that is what you are suggesting,  though it is convenient to
%separate them for clarity in the proof.}
\\
(d) If $|\p^c u| \le \Lambda$ on $\cl \U$, then
$|\partial^\cs u^\cs| \ge 1/\Lambda$ on $\cl \V$. \\
(e) If a probability measure $\gamma \ge 0$
vanishes outside $\p^c \uu \subset \cl\U \times \cl\V$,
and has marginal densities $f^\pm$, % with $\int_{\bar \V} f^- d\Leb n =1$,
then $f^+ \ge \lambda$ on $\U' \subset \cl \U$ and
$f^- \le \Lambda$ on $\cl \V$ imply $|\p^c \uu| \ge \lambda/\Lambda$ on $\U'$,
whereas $f^+ \le \Lambda$ on $\U'$ and $f^- \ge \lambda$ on $\cl \V$ imply
$|\p^c \uu| \le \Lambda / \lambda$ on $\U'$.
\end{lemma}

\begin{proof}
(a) The fact $\p^c u(\cl \U) \subset \cl\V$ is an immediate consequence of definition
(\ref{c-subdifferential}).
Since $c \in C^1(\cl \U \times \cl \V)$,
the $c$-transform $\vv=\uu^\cs:\cl V \longmapsto \R$ defined by (\ref{c-transform})
can be extended to a Lipschitz function on a neighbourhood of $\cl \V$,  hence
Rademacher's theorem asserts $\dom D \vv$ is a set of full Lebesgue measure in $\cl \V$.
Use \Bone\ to define the unique solution $F:\dom D\vv \longmapsto \cl \U$ to
$$
D_y c(F(\ty),\ty) = - D \vv(\ty).
$$
As in \cite{Gangbo95} \cite{Levin99}, the vanishing of %first order condition for
$\uu(x) + \vv(y) + c(x,y) \ge 0$ implies $\partial^{\cs} \vv(\ty) = \{F(\ty)\}$,
at least for all points $\ty \in \dom D\vv$ where $\cl \V$ has Lebesgue density
greater than one half. For Borel $X \subset \R^n$,
this shows $\partial^c \uu(X)$ differs from the Borel set $F^{-1}(X)\cap \cl\V$ by
a $\Leb{n}$ negligible subset of $\cl \V$,  whence $|\partial^c \uu| = F_{\#}\big(\Leb{n} \lfloor_{\cl \V}\bigr)$
so claim (a) of the lemma is established.

(b) Let $\|u_k - u_\infty\|_{L^\infty(\cl U)} \to 0$. It is not
hard to deduce $c$-convexity of $u_\infty$, as in e.g.\ %Proposition 6.3 of
\cite{FigalliKimMcCann-econ09p}. Define $v_k = u_k^\cs$ and $F_k$
on $\dom D v_k \subset \cl\V$ as above, so that $|\p^c u_k| =
F_{k\#} \big(\Leb{n} \lfloor_{\cl \V}\bigr)$.  Moreover, $v_k \to
v_{\infty}$ in $L^{\infty}(V)$, where $v_{\infty}$ is the
$c^{*}$-dual to $u_{\infty}$. The uniform semiconvexity of $v_k$
(i.e. convexity of $v_k(y) + \frac{1}{2}\|c\|_{C^2}|y|^2$) ensures
pointwise convergence of $D v_k \to Dv_\infty$ $\Leb{n}$-a.e.\ on
$\cl \V$. From $D_y c(F_k(\ty),\ty) = - D \vv_k(\ty)$ we deduce
$F_k \to F_\infty$ $\Leb{n}$-a.e.\ on $\cl \V$. This is enough to
conclude $|\p^c u_k| \rightharpoonup |\p^c u_k|$,  by testing the
convergence against continuous functions and applying Lebesgue's
dominated convergence theorem.

(c) To prove the converse,  suppose $u_k$ is a sequence of
$c$-convex functions which vanish at $\tx$ and $|\p^c u_k|
\rightharpoonup \mu_\infty$ weakly-$*$.  Since the $u_k$ have
Lipschitz constants dominated by $\|c\|_{C^1}$ and $\cl U$ is
compact, any subsequence of the $u_k$ admits a convergent further
subsequence by the Ascoli-Arzel\`a Theorem. A priori,  the limit
$u_\infty$ might depend on the subsequences, but (b) guarantees
$|\p^c u_\infty| = \mu_\infty$,  after which \cite[Proposition
4.1]{Loeper07p} identifies $u_\infty$ uniquely in terms of
$\mu^+=\mu_\infty$ and $\mu^- = \Leb{n} \lfloor_{\cl \V}$, up to
an additive constant; this arbitrary additive constant is fixed by
the condition $u_\infty(\tx)=0$.  Thus the whole sequence $u_k$
converges uniformly.

(e) Now assume a finite measure $\gamma \ge 0$ vanishes outside
$\p^c \uu$ and has marginal densities $f^\pm$.
%with $f^-$ having the same total mass as $\gamma$.
Then the second marginal $d\mu^- := f^- d\Leb{n}$ of $\gamma$ is
absolutely continuous with respect to Lebesgue and $\gamma$
vanishes outside the graph of $F:\overline \V \longmapsto \U$,
whence $\gamma = (F \times id)_\# \mu^-$ by e.g.\ \cite[Lemma
2.1]{AhmadKimMcCann09p}. (Here $id$ denotes the identity map,
restricted to the domain $\dom D\vv$ of definition of $F$.)
Recalling that $|\partial^c \uu| = F_{\#}\big(\Leb{n} \lfloor_{\cl
\V}\bigr)$ (see the proof of (a) above), for any Borel $X \subset
\U'$ we have
$$%\begin{eqnarray*}
\lambda|\p^c u|(X)
= \lambda\Leb{n}(F^{-1}(X))
\le \int_{F^{-1}(X)} f^-(y) d\Leb n (y)
= \int_X f^+(x) d\Leb n (x)
\le \Lambda \Leb{n}(X)
$$%\end{eqnarray*}
whenever $\lambda \le f^-$ and $f^+ \le \Lambda$.  We can also
reverse the last four inequalities and interchange $\lambda$ with
$\Lambda$ to establish claim (e) of the lemma.

(d) The last point remaining follows from (e) by taking $\gamma =
(F \times id)_\#\Leb n$. Indeed
an upper bound $\lambda$ on $|\p^c u|=F_\# \Leb n$
throughout $\cl U$ and lower bound $1$ on $\Leb{n}$ translate into
a lower bound $1/\lambda$ on $|\p^\cs u^\cs|$,  since the
reflection $\gamma^*$ defined by $\gamma^*(Y \times X) := \gamma
(X \times Y)$ for each $X \times Y \subset U \times V$ vanishes
outside $\p^\cs u^\cs$ and has second marginal absolutely
continuous with respect to Lebesgue by the hypothesis $|\partial^c
u| \le \lambda$.
\end{proof}

\begin{remark}[Monge-Amp\`ere type equation]{\rm
Differentiating (\ref{implicit G}) formally with respect to $\tx$
and recalling  $|\det DG(\tx)| = f^+(\tx) /f^-(G(\tx))$  yields
the Monge-Amp\`ere type equation
\begin{equation}\label{Monge-Ampere type equation}
\frac{\det [D^2_{xx} \uu(\tx) + D^2_{xx} c(\tx,G(\tx))]}{ |\det D^2_{xy} c(\tx,G(\tx))|}
%=  |\det DG(\tx)|
= \frac{f^+(\tx)}{f^-(G(\tx))}
\end{equation}
on $\U$, where $G(\tx)$ is given as a function of $\tx$ and $Du(\tx)$
by (\ref{implicit G}).
Degenerate ellipticity follows from the fact
that $y = G(x)$ produces equality in $\uu(x) + \uu^\cs(y) + c(x,y) \ge 0$.
A condition under which $c$-convex weak-$*$ solutions are known to
exist %(\ref{mass balance})
is given by
$$%\begin{eqnarray}\label{mass balance}
\int_{\cl\U} f^+(x) d\Leb{n}(x) = \int_{\cl \V} f^-(y) d\Leb{n}(y). \\
$$
The boundary condition $\p^c \uu(\cl U) \subset \cl V$
which then guarantees $Du$ to be uniquely determined $f^+$-a.e.\ is built into
our definition of $c$-convexity {\blue of $\uu$}. In fact, \cite[Proposition 4.1]{Loeper07p}
shows $u$ to be uniquely determined up to additive constant if either $f^+>0$ or
$f^->0$ $\Leb{n}$-a.e. on its connected domain, $\U$ or $\V$.}
%\label{second BC}
%\end{eqnarray}
\end{remark}

A key result we shall exploit several times is a maximum principle
first deduced from Trudinger and Wang's work
\cite{TrudingerWang07p}  by Loeper; see
\cite[Theorem 3.2]{Loeper07p}. A simple and direct proof, and also an extension can be found in
\cite[Theorem 4.10]{KimMcCann07p}, where the principle was also
called `double-mountain above sliding-mountain' (\DASM).  Other
proofs and extensions appear in \cite{TrudingerWang08p} \cite{TrudingerWang08q}
\cite{Villani09} \cite{LoeperVillani08p} \cite{FigalliRifford08p}:

\begin{theorem}[Loeper's maximum principle `\DASM']\label{T:DASM}
Assume \Bzero--\Bthree\ and fix $x,\tx \in \cl \U$.
If $t\in[0,1] \longmapsto -D_x c(\tx,y(t))$ is a line segment then
$f(t) := -c(x,y(t)) + c(\tx,y(t)) \le \max\{f(0),f(1)\}$ for all $t \in [0,1]$.
%$(t,x) \in [0,1] \times \cl\U$.
\end{theorem}

It is through this theorem and the next that hypothesis \Bthree\ and the
non-negative cross-curvature hypothesis \Bfour\
enter crucially. Among the many corollaries Loeper deduced from this result,
we shall need two. Proved in \cite[Theorem 3.1 and Proposition 4.4]{Loeper07p}
(alternately \cite[Theorem 3.1]{KimMcCann07p} and \cite[A.10]{KimMcCannAppendices}),
they include the $c$-convexity of the so-called {\em contact set}
(meaning the $c^*$-subdifferential at a point), and a local to global principle.

\begin{corollary}\label{C:local-global}
Assume \Bzero--\Bthree\ and fix $(\tx,\ty) \in \cl \U
\x \cl\V$. If $u$ is $c$-convex then $\p^c u(\tx)$ is $c^*$-convex
with respect to $\tx \in \U$, i.e.\ $-D_x c(\tx, \p^c u(\tx))$
forms a convex subset of $T^*_\tx \U$. Furthermore,  any local
minimum of the map $x\in\U \longmapsto u(x) + c(x,\ty)$ is a
global minimum.
\end{corollary}

As shown in \cite[Corollary 2.11]{KimMcCann08p},
the strengthening \Bfour\ of hypothesis \Bthree\ improves
the conclusion of Loeper's maximum principle. %Theorem \ref{T:DASM}.
This improvement asserts that the altitude  {\blue $f(t)$}   at each point
of the evolving landscape then accelerates as a function of $t\in[0,1]$:

\begin{theorem}[{\bf Time-convex DASM}]\label{T:time-convex DASM}
Assume \Bzero--\Bfour\ and fix $x,\tx \in \cl \U$.
If $t\in[0,1] \longmapsto -D_xc(\tx,y(t))$ is a line segment then
{\blue the function}
$t \in [0,1] \longmapsto f(t) := -c(x,y(t)) + c(\tx,y(t))$ is convex.
\end{theorem}

\begin{remark}\label{rmk:symmetry}{\rm
Since all assumptions \Bzero--\Bfour\ on the cost
are symmetric in $x$ and $y$,
all the results above still hold when {\RJM the roles of $x$ and $y$ are exchanged.}}
\end{remark}
%Finally,  for any measure $\mu^+ \ge 0$ on $\cl \U$,  we use the term {\em support}
%and the notation $\supp \mu^+ \subset \cl\U$
%to refer to the smallest closed set carrying the full mass of $\mu^+$.

%{\bf A0, A1, A2}
%cross-curvature
%$c$-convexity
%$c$-segment
%subdifferential
%$c$-subdifferential
%$c$-Monge-Amp\`ere measure
%$c$-transform: $c^*(x,y)=c(y,x)$, $(\uu,\uu^\cs)$
%$\supp$ of a measure
%- $\mathfrak{S_c}(x,y)[\xi,\eta]$ -> $\cross_{(x,y)}[,]$
%{\bf DASM} / {\bf time-convex DASM}

\subsection*{{\blue $c$-Monge-Amp\`ere equation%, bi-Lipschitz constants, Jacobian bounds
%Cost-exponential coordinates, null Lagrangians, and affine renormalization
}}\label{S:notation}
%{\blue We set up additional} notation for the rest of the paper.
%Let $c:  \R^n \times \R^n \to \R$ be a continuous cost function, and let $\U$,
%$\V$ be smooth bounded domains in $\R^n$.
%

Fix $\lambda,\Lambda >0$ and an open domain $\U^\lambda \subset
\U$, and let $\uu$ be a $c$-convex solution of the
$c$-Monge-Amp\`ere equation
\begin{equation}
\label{eq:cMA}
\left\{ \begin{array}{ll}
\l \Leb{n}%_{\lfloor \U^\lambda}
\leq
|\p^c\uu| \leq \frac{1}{\l}\Leb{n} %_{\lfloor \U}
&\text{in } \U^\lambda \subset \U,
\\ |\p^c\uu| \le \Lambda \Leb{n} &\text{in } \cl\U.
%,\\ \p^c\uu \big(\cl \U\big)=\cl \V. %\qquad {\rm and} \qquad
\end{array}
\right.
\end{equation}
{\blue Note that throughout this paper, we} {\RJM require $c$-convexity \Btwo\ not of $\U^\lambda$
but only of $U$.}
We sometimes abbreviate \eqref{eq:cMA} by writing $|\p^c \uu| \in
[\lambda, 1/\lambda]$. In the following sections  we will prove
interior {\blue H\"older} differentiability of $\uu$ on $\U^\lambda$, that is {\blue $\uu
\in C^{1, \alpha}_{loc}(\U^\lambda)$; see Theorems \ref{T:continuity} and \ref{T:C1alpha}}.
%with H\"older exponent $\alpha=\alpha(n, \lambda)$ depending only on dimension $n$ and on
%$\lambda$; see Theorem~\ref{thm:C 1 alpha} and Corollary~\ref{C:final}.

\subsection*{\YHK Convex sets}
{\RJM We close by recalling two {\blue nontrivial} results for convex sets.} {\YHK These will be essential in Section~\ref{S:Alexandrov}
 and later on.}
The first one is due to Fritz John \cite{John48}:

\begin{lemma}[John's lemma]\label{L:John}
For a compact convex set $Q \subset \R^n$ {\blue with nonempty interior}, %with center of mass at $0$,
there exists an affine transformation $ L: \R^n \to \R^n$ such
that $\cl {B_1} \subset L^{-1}(Q) \subset \cl {B_{n}}$.
\end{lemma}

The above result can be restated by saying that any compact convex set $Q$ {\blue with nonempty interior} contains an ellipsoid $E$, whose dilation $n\rdot E$ by factor $n$
with respect to its center
contains $Q$:
\begin{equation}\label{E:well-centered}
E \subset Q \subset n \rdot E.
\end{equation}

{\RJM The following is our main result from \cite{FKM-convex}; it enters crucially
in the proof of Theorem~\ref{thm:estimate}.}

\begin{theorem}[\blue Convex bodies and supporting hyperplanes]
\label{T:ratio}
Let $\tilde Q \subset \R^n$ be a well-centered convex body, meaning that \eqref{E:well-centered}
holds for some ellipsoid $E$ centered at the origin. Fix $ 0 \le s \le s_0 <1$.
To each $y \in (1-s) \rdot \partial \tilde Q$
corresponds at least one line $\ell$ through the origin and hyperplane $\Pi$
supporting $\tilde Q$ such that:
$\Pi$ is orthogonal to $\ell$ and
 \begin{equation}\label{E:key bound}
\dist(y, \Pi) \le c(n,s_0) s^{ 1/2^{n-1}} \diam(\ell \cap \tilde Q).
 \end{equation}
Here, $c(n,s_0)$ is a constant depending {\blue only} on $n$ and $s_0$, namely
$c(n,s_0) = n^{3/2} (n-\frac{1}{2})
\Big( \frac{1+(s_0)^{1/2^n}}{1-(s_0)^{1/2^n}} \Big)^{n-1}$.
\end{theorem}

\section{Choosing coordinates which ``level-set convexify'' $c$-convex functions}\label{S:transform}

{\blue Recall that $c \in C^4(\cl\U \times \cl \V)$ is a {\blue cost function} satisfying \Bone--\Bthree\ on a pair of
bounded domains $\U$ and $\V$ which are strongly $c$-convex with
respect to each other \Btwos.}
In the current section, we introduce an important transformation (mixing
dependent and independent variables)
for the cost $c(x,y)$ and potential $\uu(x)$, which plays a crucial role in the subsequent
analysis. This change of variables and its most relevant properties are encapsulated in the
following {\blue Definition~\ref{D:cost exponential} and Theorem~\ref{thm:apparently convex}.

}

\begin{definition}[Cost-exponential coordinates and apparent properties]
\label{D:cost exponential}
Given $c \in C^4\big(\cl\U \times \cl\V\big)$ strongly twisted \Bzero--\Bone, we refer to
the coordinates $(q,p) \in \cl \U_{\ty} \times \cl \V_{\tx}$ defined by
\begin{equation}\label{q of x}
q = q(x) = - D_y c(x,\ty), \qquad p = p(y) = -D_x c(\tx,y),
\end{equation}
as the \emph{cost exponential coordinates} from $\ty \in \cl \V$ and $\tx \in \cl\U$ respectively.
We denote the inverse diffeomorphisms by
$x:\cl \U_{\ty} \subset T^*_{\ty} \V\longmapsto \cl\U$ and
$y:\cl \V_{\tx} \subset T^*_{\tx} \U \longmapsto \cl\V$; they satisfy
\begin{equation}\label{x of q}
q = - D_y c(x(q),\ty), \qquad p = -D_x c(\tx,y(p)).
\end{equation}
The cost $\tilde c(q,y) = c(x(q),y) - c( x(q),\ty)$ is called the
\emph{modified cost at $\ty$}. A subset of $\cl \U$ or function
thereon is said to \emph{appear from $\ty$} to have
property $A$, % when \emph{viewed from $\ty$},
if it has property $A$ when expressed in the coordinates $q \in
\cl \U_{\ty}$.
\end{definition}

\begin{remark}{\rm
Identifying the cotangent vector $0 \oplus q$
with the tangent vector $q^* \oplus 0$ to $\U \times \V$ using
the pseudo-metric of Kim and McCann \cite{KimMcCann07p}
shows $x(q)$ %
to be the projection to $\U$ of the pseudo-Riemannian
exponential map $\exp_{(\tx,\ty)} (q^* \oplus 0)$;  similarly
$y(p)$
is the projection to $\V$ of $\exp_{(\tx,\ty)} (0 \oplus p^*)$.
Also, $x(q) =: \cs$-$\exp_{\ty}q$ and $y(p) =: c$-$\exp_{\tx} p$
in the notation of Loeper \cite{Loeper07p}.}
\end{remark}

{\blue

{\em In the sequel, whenever we use the expression  $\tilde c(q,
\cdot)$ or $\tilde u(q)$, we refer to the modified cost function
{\RJM defined above} and level-set convex potential defined {\RJM
below}.} Since properties \Bzero--\Bfour\ {\RJM (}and \Btwos{\RJM )}
were shown to be tensorial in nature (i.e.\ coordinate independent)
in \cite{KimMcCann07p} \cite{Loeper07p}, the modified cost $\tilde
c$ inherits these properties from the original cost $c$ with one
exception:  (\ref{x of q}) defines a $C^3$ diffeomorphism $q \in \cl
\U_{\ty}  \longmapsto x(q) \in \cl \U $, so the cost $\tilde c \in
C^3(\cl \U_{\ty} \times \cl \V)$ may not be $C^4$ smooth. However,
its definition reveals that we may still differentiate $\tilde c$
four times as long as no more than three of the four derivatives
fall on the variable $q$,  and it leads to the same geometrical
structure (pseudo-Riemannian curvatures, including (\ref{MTW})) as
the original cost $c$ since the metric tensor and symplectic form
defined in \cite{KimMcCann07p} involve only mixed derivatives
$D^2_{qy} \tilde c$,  and therefore remain $C^2$ functions of the
coordinates $(q,y) \in \cl \U_\ty \times \cl \V$.

%For $(k, l) \in \N \times \N$, we use the notation
%\footnote{What can $c^{-(k,l)}$ mean unless $(k,l) = (1,1)$? Do we ever use this notation subsequently?}
%\begin{align}\label{notation:derivatives}
%c_{(k, l)}&= \|D_x^k D^l_y c \|_\infty =
%\sup_{i_1,  \cdots, i_k, j_1, \cdots , j_l} \| D_{x_{i_1}} \cdots D_{x_{i_k}} D_{y_{j_1}} \cdots D_{y_{j_l}} c(x, y)\|_\infty
%%\\\nonumber c^{-(k, l)}&= \|[D_x^k D^l_y c]^{-1} \|_\infty =\sup_{i_1,  \cdots, i_k, j_1, \cdots , j_l} \| [D_{x_{i_1}} \cdots D_{x_{i_k}} D_{y_{j_1}} \cdots D_{y_{j_l}} c(x, y)]^{-1}\|_\infty.
%\end{align}
%Here the domain for the $L^\infty$-norm $\| \cdot\|_\infty$ will be specified unless it is obvious from the context.
We also use
\begin{eqnarray}\label{bi-Lipschitzbound}
\beta^\pm_c &=&\beta^\pm_c(U \times V) := \ \ \| (D^2_{xy} c)^{\pm 1}\|_{L^\infty(U \times V)}\\
\gamma^\pm_c &=& \gamma^\pm_c(U \times V) := \|\det (D^2_{xy} c)^{\pm 1}\|_{L^\infty(U \times V)}
\label{Jacobian bound}
\end{eqnarray}
to denote the bi-Lipschitz constants $\beta^\pm_c$ of the
coordinate changes (\ref{q of x}) and the Jacobian bounds
$\gamma^\pm_c$ for the same transformation. Notice $\gamma^+_c
\gamma^-_c \ge 1$ for any cost satisfying \Bone, and equality
holds whenever the cost function $c(x,y)$ is quadratic. So the
parameter $\gamma^+_c\gamma^-_c$ crudely quantifies the departure
from the quadratic case. The inequality $\beta^+_c \beta^-_c \ge
1$ is much more rigid,  equality implying $D^2_{xy} c(x,y)$ is the
identity matrix, and not merely constant.
\\}

Our first contribution is the following theorem.
If the cost function satisfies \Bthree, then
the level sets of the $\tilde c$-convex potential appear convex from $\ty$,
as was discovered independently from us by Liu \cite{Liu09}, and
exploited by Liu with Trudinger and Wang \cite{LiuTrudingerWang09p}.
Moreover, for a non-negatively cross-curved cost \Bfour, it shows that
any $\tilde c$-convex potential appears convex from $\ty \in \cl \V$.
Note that although the difference between the cost $c(x,y)$ and the modified cost
$\tilde c(q,y)$ depends on $\ty$,  they differ by a null Lagrangian $c(x,\ty)$
which --- being independent of $y \in \V$ --- does not affect the question
of which maps $G$ attain the infimum (\ref{Monge}).
Having a function with convex level sets is a useful starting point,
since  {\blue it opens a possibility to apply the approach and techniques developed by Caffarelli, and refined by Gutie\'errez, Forzani, Maldonado and others (see \cite{Gutierrez01} \cite{ForzaniMaldonado04}), to address the regularity of $c$-convex potentials.}
%to apply Caffarelli's affine renormalization of convex sets
%approach and a full range of techniques from Gutierrez \cite{Gutierrez01}
%to address the regularity of $c$-convex potentials.

%\footnote{RJM: Add a sentence or corollary about Jacobian bounds}

\begin{theorem}[Modified $c$-convex functions appear level-set convex]\label{thm:apparently convex}
Let $c \in C^4 \big(\cl \U \times \cl\V\big)$ satisfy
\Bzero--\Bthree.
If $\uu = \uu^{c^* c}$ is $c$-convex on $\cl \U$,  then
$\tu(q) = \uu(x(q)) + c(x(q),\ty)$ has convex level sets,
as a function of the cost exponential coordinates
$q \in \cl \U_{\ty}$ from $\ty \in \cl \V$. Moreover,
\begin{equation}
\label{eq:semiconvexity}
{\blue \tu + M_c|q|^2\quad \text{is convex},}
\end{equation}
where $M_c:=(\beta_c^-)^2\|D_{xx}^2c\|_{L^\infty(U\times V)} + (\beta_c^-)^3\|D_{x}c\|_{L^\infty(U\times V)}\|D_{xx}^2D_yc\|_{L^\infty(U\times V)}$.
If, in addition, $c$ is non-negatively cross-curved \Bfour\ then $\tu$ is convex on
$\cl \U_{\ty}$.
In either case $\tu$ is minimized at $q_0$ if $\ty\in \p^c \uu(x(q_0))$.
Furthermore, $\tu$ is $\tilde c$-convex with respect to the modified cost
$\tilde c(q,y) := c(x(q),y) - c(x(q),\ty)$ on
$\cl \U_{\ty} \times \cl \V$, and
$\p^{\tilde c} \tu(q) = \p^c \uu(x(q))$ for all $q \in \cl \U_\ty$.
\end{theorem}

\begin{proof}
The final sentences of the theorem are elementary:
$c$-convexity $\uu=\uu^{c^*c}$ asserts
$$%\begin{eqnarray*}
\uu(x)=\sup_{y \in \cl \V} -c(x,y) -\uu^\cs(y) \quad {\rm and} \quad
\uu^\cs(y) = \sup_{q \in \cl \U_{\ty}} -c (x(q),y) - \uu(x(q))
= \tu^{\tilde c^*}(y)
$$%\end{eqnarray*}}
from (\ref{c-transform}), hence
\begin{eqnarray*}
\tu(q) &=& \sup_{y \in \cl \V} -c(x(q),y) + c(x(q),\ty) -\uu^\cs(y)
%& \uu^\cs(y) = \sup_{q \in \cl \U_{\ty}} - c(x(q),y) + c(x(q),\ty) - (\uu(x(q))-c(x(q),\ty)
\\      &=& \sup_{y \in \cl \V} - \tilde c (q, y) - \tu^{\tilde c^*}(y),
\end{eqnarray*}
and $\p^{\tilde c} \tu(q) = \p^c \uu(x(q))$ since all three
suprema above are attained at the same $y \in \cl \V$.
Taking $y=\ty$ reduces the inequality $\tu(q) + \tu^{\tilde c^{*}}(y) + \tilde c(q,y) \ge 0$
to $\tu(q) \ge - \tu^{\tilde c^{*}}(\ty)$ ,  with equality precisely if
$\ty \in \partial^{\tilde c} \tu(q)$.  It remains to address the convexity claims.

Since the supremum $\tu(q)$ of a family of convex functions is again convex,
it suffices to establish the convexity of
$q \in \cl\U_{\ty} \longmapsto -\tilde c(q,y)$ for each $y \in \cl \V$
under hypothesis \Bfour.  For a similar reason, it suffices to establish the level-set
convexity of the same family of functions under hypothesis \Bthree.

First assume \Bthree.  Since
\begin{equation}\label{special geodesic}
D_y \tilde c (q,\ty) = D_y c (x(q),\ty):=-q
\end{equation}
we see that $\tilde c$-segments in $\cl \U_{\ty}$ with respect to $\ty$
coincide with ordinary line segments. Let $q(s) = (1-s) q_0 + s q_1$ be any line
segment in the convex set $\cl \U_{\ty}$.
Define $f(s,y) := -\tilde c(q(s),y) = - c(x(q(s)),y) + c(x(q(s)),\ty)$.
Loeper's maximum principle (Theorem \ref{T:DASM} above, see also Remark \ref{rmk:symmetry})
asserts $f(s,y) \le \max\{f(0,y),f(1,y)\}$,
which implies convexity of each set
$\{q \in \cl \U_{\ty} \mid -\tilde c(q,y) \le const\}$.
Under hypothesis \Bfour, Theorem \ref{T:time-convex DASM}
goes on to assert convexity of $s \in [0,1] \longmapsto f(s,y)$ as desired.

Finally, \eqref{eq:semiconvexity} follows from the fact that $\tu$ is a supremum of cost functions and that
$|D_{qq}\tc | \leq (\beta_c^-)^2|D_{xx}^2c| + (\beta_c^-)^3|D_{x}c||D_{xx}^2D_yc|$ (observe that $D_q x(q) = - D^2_{xy} c(x(q),\tilde y)^{-1}$).
\end{proof}
{\blue
\begin{remark}{\rm
 The above level-set convexity for the modified $\tc$-convex functions requires \Bthree \ condition, as one can easily derive using Loeper's counterexample \cite{Loeper07p}.}
\end{remark}}
The effect of this change of gauge on {\blue the $c$-Monge-Amp\`ere equation  \eqref{eq:cMA}} is summarized in a corollary:

\begin{corollary}[Transformed $\tc$-Monge-Amp\`ere inequalities]
\label{C:Jacobian transform}
Using the hypotheses and notation of Theorem \ref{thm:apparently convex},
if $|\p^c u| \in [\lambda,\Lambda] \subset [0,\infty]$ on $\U' \subset \cl \U$,
then $|\p^\tc \tu| \in [\lambda/\gamma^+_c,\Lambda\gamma^-_c]$ on $\U'_\ty %=q(\U')
= -D_yc(\U',\ty)$, where $\gamma^\pm_c = \gamma^\pm_c(\U' \times
V)$ and $\beta^\pm_c = \beta^\pm_c(\U' \times V)$ are defined in
\eqref{bi-Lipschitzbound}--\eqref{Jacobian bound}. Furthermore,
$\gamma^\pm_\tc := \gamma^\pm_\tc(\U'_\ty \times \V) \le
\gamma^+_c \gamma^-_c$ and $\beta^\pm_\tc := \beta^\pm_\tc(\U'_\ty
\times \V) \le \beta^+_c \beta^-_c$.
%and
%$$
%\frac{|\p^\tc \tu}{\Leb{n}}(q(x)) =
%$$
\end{corollary}

\begin{proof}
From the Jacobian bounds $|\det D_x q(x)| \in
[1/\gamma^-_c,\gamma^+_c]$ on $\U'$, we find $\Leb{n}(X)/
\gamma^-_c \le \Leb{n}(q(X)) \le \gamma^+_c \Leb{n}(X)$ for each
$X \subset \U'$. On the other hand,  Theorem \ref{thm:apparently
convex} asserts $\p^\tc \tu(q(X)) = \p^c u(X)$,  so the claim
$|\p^\tc \tu| \in [\lambda/\gamma^+_c,\Lambda\gamma^-_c]$ follows
from the hypothesis $|\p^c u| \in [\lambda,\Lambda]$, by
definition (\ref{c-Monge-Ampere measure}) and the fact that $q:\cl
\U \longrightarrow \cl \U_\ty$ from \eqref{q of x} is a
diffeomorphism; see \Bone. The bounds $\gamma^\pm_\tc \le
\gamma^+_c \gamma^-_c$ and $\beta^\pm_\tc \le \beta^+_c \beta^-_c$
follow from $D^2_{qy} \tc (q,y) = D^2_{xy} c(x(q),y) D_q x(q)$ and
$D_q x(q) = - D^2_{xy} c(x(q),\tilde y)^{-1}$.
\end{proof}

\subsection*{Affine renormalization}\label{SS:renormal}
%\footnote{\textbf{AF: this section now is not needed anymore, so we may either to remove it or to present it in a different way (more as a remark than something that we need for our proofs)}}
%
{\blue We record here an observation that the $\tc$-Monge-Amp\`ere measure is invariant under an
affine renormalization. This {\RJM is potentially useful in applications, though}
we {\em do not} use this fact for the results of the present paper.}
%The renormalization of a function $\tu$ by an affine
%transformation $L:\R^n \to \R^n$ will be useful in Section~\ref{S:Alexandrov} to prove our Alexandrov type estimates.
%Let us therefore record the following observations.
{\blue For an affine
transformation $L:\R^n \to \R^n$,}
define
\begin{equation}\label{renormalized solution}
\tu^*(q)=|\det L|^{-2/n} \tu(Lq).
\end{equation}
Here $\det L$ denotes the Jacobian determinant of $L$, i.e. the determinant of the linear part of $L$.

\begin{lemma}[Affine invariance of $\tc$-Monge-Amp\`ere measure]
Assuming \Bzero--\Bone,
given a $\tc$-convex function $\tu: \U_\ty \longmapsto \R$ and affine bijection $L:\R^n \longmapsto \R^n$,
define the renormalized potential $\tu^*$ by \eqref{renormalized solution} and renormalized cost
\begin{equation}\label{renormalized cost}
\tc_*(q,y)=|\det L|^{-2/n} \tilde c(Lq,L^*y)
\end{equation}
using the adjoint $L^*$ to the linear part of $L$.  Then, for any
Borel set $\Q \subset \cl \U_\ty$,
\begin{align}\label{compare p var with p var *}
|\p \tu^*| (L^{-1} \Q) &=  |\det L|^{-1} |\p \tu| (\Q),
%\qquad\forall\, \Q \subset \U_{\ty} \text{ Borel}
\\|\p^{\tilde c_*} \tu^*| (L^{-1} \Q) &=  |\det L|^{-1} |\p^\tc \tu| (\Q).
%\qquad\forall\, \Q \subset \U_{\ty} \text{ Borel}.
\label{renormalized cMA measure}
\end{align}
\end{lemma}

\begin{proof}
From \eqref{subdifferential} we see $\bar p \in \p \tu(\bq)$ if
and only if $|\det L|^{-2/n} L^* \bar p \in \p \tu^*(L^{-1} \bq)$,
thus \eqref{compare p var with p var *} follows from $\p
\tu^*(L^{-1} \Q) = |\det L|^{-2/n} L^* \big(\p \tu (\Q)\big)$.
Similarly,  since \eqref{c-transform} yields $(\tu^*)^{\tilde
c^*_*}(y) = |\det L|^{-2/n} \tu^{\tilde c^*}(L^*y)$, we see $\by
\in \p^\tc \tu(\bq)$ is equivalent to $|\det L|^{-2/n} L^* \bar y
\in \p^{\tilde c_*} \tu^*(L^{-1} \bq)$ from
\eqref{c-subdifferential} (and Theorem \ref{thm:apparently
convex}), whence  $\p^{\tilde c_*} \tu^*(L^{-1} \Q) = |\det
L|^{-2/n} L^* \big(\p^{\tc} \tu^*(\Q)\big)$ to establish
\eqref{renormalized cMA measure}.
\end{proof}

As a corollary to this lemma,  we recover the affine invariance not only of the
Monge-Amp\`ere equation satisfied by $\tu(q)$ --- but also of the $\tc$-Monge-Amp\`ere equation
it satisfies --- under coordinate changes on $\V$
(which induce linear transformations $L$ on $T^*_\ty V$ and $L^*$ on $T_\ty V$):
for $q \in \U_\ty$,
$$
\frac{d|\p \tu^*|}{d\Leb{n}} (L^{-1} q) =  \frac{d|\p \tu|}{d\Leb{n}} (q)
\quad{\rm and}\quad
\frac{d|\p^{\tc_*} \tu^*|}{d\Leb{n}} (L^{-1} q) =  \frac{d|\p^\tc \tu|}{d\Leb{n}} (q).
$$

\section{Strongly $c$-convex interiors and boundaries not mixed by $\p^c \uu$}
%{Topological properties of the (multi-valued) mapping $\p^c \uu$}
\label{S:mapping}

The subsequent sections of this paper are largely devoted to ruling out exposed points
in $\U_\ty$ of {\YHK the set where the $\tc$-convex potential from
Theorem \ref{thm:apparently convex} takes its minimum.}
%\footnote{YHK: see whether you agree....
%AF: to me it looks fine}
This current section rules out exposed points on the boundary of $\U_\ty$.
We do this by proving an important topological property of the (multi-valued) mapping
$\p^c \uu \subset \cl \U \times \cl \V$.
Namely, we show that the subdifferential $\p^c \uu$ maps interior points of
$\supp |\p^c u| \subset \cl U$ %(resp. boundary points of $\U$)
only to interior %(resp. boundary)
points of $\V$, under hypothesis \eqref{eq:cMA}, and conversely that
$\p^c \uu$ maps boundary points of $\U$ only to boundary points of
$\V$. This theorem may be of independent interest, and was required
by Figalli and Loeper to conclude their continuity result concerning
maps {\blue between {\RJM \em \YHK two dimensional} domains} which optimize
\Bthree\ costs \cite{FigalliLoeper08p}.

This section does not use the \Bthree\ assumption on the cost
function $c \in C^4(\cl \U \times \cl \V)$, but relies crucially on
the \emph{strong} $c$-convexity \Btwos\ of its domains $U$ and  $V$
(but importantly, not of $\supp |\partial^{c} u |$). No analog for
Theorem \ref{thm:bdry-inter} was needed by Caffarelli to establish
$C^{1,\alpha}$ regularity of convex potentials $\uu(x)$ whose
gradients optimize the classical cost $c(x,y) = - \<x,y\>$
\cite{Caffarelli92}, since in that case he was able to take
advantage of the fact that the cost function is smooth on the whole
of $\R^n$ to chase potentially singular behaviour to infinity.  (One
general approach to showing regularity of solutions for degenerate
elliptic partial differential equations is to exploit the
threshold-hyperbolic nature of the solution to try to follow either
its singularities or its degeneracies to the boundary,  where they
can hopefully be shown to be in contradiction with boundary
conditions;  the {\em degenerate} nature of the ellipticity
precludes the possibility of {\em purely local} regularizing
effects.)

%\footnote{Might we want closures in these hypotheses?}

\begin{theorem}[Strongly $c$-convex interiors and boundaries not mixed by $\p^c \uu$]
\label{thm:bdry-inter}
Let $c$ satisfy \Bzero--\Bone\ and %\Btwos,
$\uu=\uu^{c^*c}$ be a $c$-convex function (which implies $\p^c\uu(\cl \U)=\cl \V$),
and $\lambda>0$.
\begin{enumerate}
\item[(a)]
If $|\p^c\uu| \geq \lambda$ on $X \subset \cl U$ and $\V$
is strongly $c^*$-convex with respect to $X$,
then interior points of $X$ cannot be mapped by $\p^c\uu$ to boundary points of $\V$:
i.e.\ $(X \times \p V) \cap \p^c \uu \subset (\p X \times \p \V)$.

\item[(b)] If $|\p^c\uu| \leq \Lambda$ on $\cl \U$, and $\U$ is
strongly $c$-convex with respect to $\V$, then boundary points of
$\U$ cannot be mapped by $\p^c \uu$ into interior points of $\V$:
i.e.\ $\partial \U \times \V$ is disjoint from  $\p^c \uu$.
%\subset (\partial \U \times \partial \V)$.
\end{enumerate}
\end{theorem}

\begin{proof}
Note that when $X$ is open the conclusion of (a) implies $\p^c u$
is disjoint from $ X \times \p V$. We therefore remark that it
suffices to prove (a), since (b) follows from (a) exchanging the
role $x$ and $y$ and observing that $|\p^c\uu| \leq \Lambda$
implies $|\p^\cs\uu^\cs|\geq 1/\Lambda$ as in Lemma \ref{L:cMA
properties}(d).

%\footnote{$(\tilde x,\tilde y) \to (x_0,y_0)$ or $(\bar x,\bar y)$ or $(x^0,y^0)$ or $(x',y')$?
%Want the same notation here and in COV DEFN so it can be cited}

Let us prove (a).
Fix any point $\tilde x$ in the interior of $X$, and
$\tilde y \in \p^c\uu(\tilde x)$. Assume by contradiction that
$\tilde y \in \p\V$.
{\blue The idea of proof is summarized in Figure~\ref{figcone}. We first fix appropriate coordinates.}
At $(\tilde x,\tilde y)$ we use \Bzero--\Bone\ to
define cost-exponential coordinates $(p,q)\longmapsto (x(q),y(p))$
by
\begin{eqnarray*}
p =& - D_x c(\tilde x, y(p)) + D_x c(\tilde x, \tilde y)  &\in T^*_\tx(U)\\
q =& D^2_{xy} c(\tilde x,\tilde y)^{-1}(D_y c(x(q), \tilde y) - D_y c(\tilde x, \tilde y)) &\in T_\tx(U)
\end{eqnarray*}
and define a modified cost and potential by subtracting null
Lagrangian terms:
\begin{eqnarray*}
\tilde c(q,p) &:=& c(x(q),y(p)) - c(x(p),\tilde y) - c(\tilde x, y(p)) \\
\tilde \uu(q) &:=& \uu(x(q)) + c(x(q),\tilde y).
\end{eqnarray*}
Similarly to Corollary \ref{C:Jacobian transform},
$|\partial^{\tilde c} \tilde u| \ge \tilde \lambda
:=\lambda/(\gamma^+_c \gamma^-_c)$, where $\gamma^\pm_c$ denote
the Jacobian bounds \eqref{Jacobian bound} for the coordinate
change. Note $(\tilde x,\tilde y) = (x(\zero),y(\zero))$
corresponds to $(p,q) = (\zero,\zero)$. Since $c$-segments with
respect to $\tilde y$ correspond to line segments in $\U_{\tilde
y} := -D_y c(U,\ty)$ we see $D_p \tilde c(q,\zero)$ depends
linearly on $q$, whence $D^3_{qqp} \tilde c(q,\zero)=0$; similarly
$c^*$-segments with respect to $\tilde x$ become line segments in
the $p$ variables, $D_q \tilde c(\zero, p)$ depends linearly on
$p$, $D^3_{ppq} c(\zero, p)=0$, and the extra factor $D^2_{xy}
c(\tilde x, \tilde y)^{-1}$ in our definition of $x(q)$ makes $-
D^2_{pq}\tilde c(\zero,\zero)$ the identity matrix (whence $q= -
D_p \tilde c(\zero, q)$ and $p = -D_q \tilde c(p,\zero)$ for all
$q$ in $\U_\ty= x^{-1}(U)$ and $p$ in $\V_\tx := y^{-1}(V))$.
We denote $X_\ty := x^{-1}(X)$ %and $\tilde \U = x^{-1}(\U)$
and choose orthogonal coordinates on $U$ which make $-\hat e_n$
the outer unit normal to $\V_\tx \subset T^*_\tx \U$ at $\tilde p
=\zero$.
Note that %$\tilde U_\ty \subset T_{\tilde x} \U$ and
$V_\tx$ is strongly convex by hypothesis~(a).
%and a linear change of coordinates
%on $\U$ to make $D^2_{x_i y_j} c(\zero,\zero)$ the negative identity matrix.

In these variables, {\blue as in Figure~\ref{figcone}}, consider a small cone of height $\e$ and angle $\theta$
around the $-\hat e_n$ axis:
$$
E_{\theta,\e} := \left\{q \in \R^n \mid \Big| - \hat e_n -
\frac{q}{|q|}\Big| \le {\blue \sin \theta}, |q| \le \e \right\}
$$
Observe that, if $\theta,\e$ are small enough, then $E_{\theta,\e}
\subset  X_\ty$, and its measure is of order $\e^n\theta^{n-1}$.
Consider now a slight enlargement
$$
E'_{\theta,C_0\e}:=\Bigl\{p =(P,p_n) \in \R^n \mid p_n \le
\theta|p|+C_0\e|p|^2\Bigr\},
$$
of the polar dual cone, where $\e$ will be chosen sufficiently
small depending on the large parameter $C_0$ forced on us later.
\begin{figure}
\begin{center}
\centerline{\epsfysize=1.5truein\epsfbox{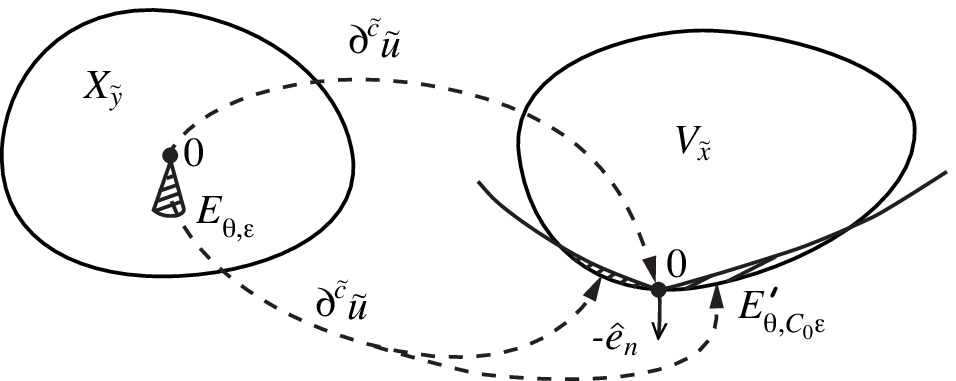}}
\caption{{\small If $\p^\tc \tu$ sends an interior point onto a boundary point,
by $\tc$-monotonicity of $\p^\tc \tu$ the small cone
$E_{\theta,\e}$ (with height $\e$ and opening $\theta$) has to be sent onto $E_{\theta,C_0\e}' \cap
\V_{\tx}$. Since $\Leb{n}(E_{\theta,\e}) \sim \theta^{n-1}$ for $\e>0$ small but fixed, while
$\Leb{n}(E_{\theta,C_0\e}' \cap \V_{\tx}) \lesssim \theta^{n+1}$}
(by the {\blue strong} convexity of $\tilde \V_{\tx}$), we get a
contradiction as $\theta \to 0$.}\label{figcone}
\end{center}
\end{figure}
%\footnote{YH: I have changed "uniform convexity" here to "strong convexity" AF: fine to me}

The strong convexity ensures $\V_\tx$ is contained in a ball
$B_R(R\hat e_n)$ of some radius $R>1$ contained in the half-space
$p_n \ge 0$ with boundary sphere passing through the origin. As
long as $C_0\e < (6 R)^{-1}$ we claim $E'_{\theta,C_0\e}$
intersects this ball --- a fortiori $\V_\tx$ --- in a set whose
volume tends to zero like $\theta^{n+1}$ as $\theta \to 0$.
Indeed,  from the inequality %$|P| < R$, $p_n < 2R$ and
%\begin{eqnarray}
%\frac{|P|^2}{2R}\le
$$
p_n
%&\le& \theta \sqrt{|P|^2 + p_n^2} + \frac{|P|^2}{6R} + \frac{p_n^2}{6R} \\
%&\le& \theta \sqrt{|P|^2 + p_n^2} + \frac{1}{6}|P| + \frac{1}{3}p_n
%&\le& \theta \sqrt{|P|^2 + p_n^2} + C_0 \e |P|^2 + C_0 \e {p_n^2} \\
\le \theta \sqrt{|P|^2 + p_n^2} + \frac{1}{6}|P| + \frac{1}{3}p_n
$$%\end{eqnarray}
satisfied by any $(P,p_n) \in E'_{\theta,C_0\e} \cap B_R(R \hat e_n)$ we deduce
$p_n^2 \le |P|^2 (1 + 9 \theta^2)/(2-9\theta^2)$,
i.e.\ $p_n < |P|$ if $\theta$ is small enough.  Combined with the
further inequalities
$$%\begin{eqnarray*}
\frac{|P|^2}{2R} \le p_n
\le \theta \sqrt{|P|^2 + p_n^2} + C_0 \e |P|^2 + C_0 \e p_n^2
$$
(the first inequality follows by the strong convexity of
$\V_\tx$, and the second from the definition of $E_{\theta, C_\e}'$),
this yields $|P| \le 6 \theta \sqrt{2}$ %/((2R)^{-1} - 2 C_0 \e)$
and $p_n \le O(\theta^2)$ as $\theta \to 0$. Thus $\Leb n
(E_{\theta, C_\e}' \cap \V_\tx) \le C \theta^{n+1}$ for a
dimension dependent constant $C$, provided $C_0\e < (6R)^{-1}$.

The contradiction now will come from the fact that, thanks to the
$\tilde c$-cyclical monotonicity of $\partial^{\tilde c} \tilde
\uu$, if we first choose $C_0$ big and then we take $\e$
sufficiently small, the image of all $q \in E_{\theta,\e}$ by
$\p^{\tilde c} \tilde \uu$ has to be contained in
$E'_{\theta,C_0\e}$ for $\theta$ small enough. Since $\p^{\tilde
c} \tilde \uu\Big(\cl{X_\ty}\Big) \subset \cl{\V_\tx}$ this will
imply
$$
\e^n\theta^{n-1} \sim \tilde \lambda \Leb{n} (E_{\theta,\e}) \leq
|\p^{\tilde c}{\tilde \uu}|(E_{\theta,\e}) \leq \Leb{n} (\V_\tx
\cap E'_{\theta,C_0\e} ) \leq C\theta^{n+1},
$$
which gives a contradiction as $\theta \to 0$, for $\e>0$ small but fixed.

Thus all we need to prove is that, if $C_0$ is big  enough, then $\p^{\tilde c} \tilde \uu(E_{\theta,\e}) \subset
E'_{\theta,C_0\e}$ for any $\e$ sufficiently small. Let
$q \in E_{\theta,\e}$ and $p \in \p^{\tilde c}\tilde \uu(q)$.
{\blue Notice that }
\begin{align}\label{eq:c-monotone-ineq} \nonumber
& \int_0^1 ds \int_0^1 dt \,D^2_{qp} \tilde c(sq,tp)[q,p]\\ \nonumber
&=\tilde c(q,p) + \tilde c(\zero,\zero)- \tilde c(q,\zero)-c(\zero,p) \\
& \le 0
\end{align}
{\blue
where the last inequality is a consequence of $\tilde
c$-monotonicity of $\partial^{\tilde c}\tilde u$, see for instance
\cite[Definitions 5.1 and 5.7]{Villani09}.
Also note that
}
\begin{align}\label{eq:Dpqtildec} \nonumber
D^2_{qp} \tilde c(sq,tp)
= &D^2_{qp} \tilde c(\zero,tp) + \int_0^s ds' D^3_{qqp} \tilde c(s'q,tp)[q] \\ \nonumber
= &D^2_{qp} \tilde c(\zero,\zero) + \int_0^t dt' D^3_{qpp} \tilde c(\zero,t'p)[p] \\ \nonumber
& + \int_0^s ds' D^3_{qqp} \tilde c(s'q,\zero)[q] + \int_0^s ds' \int_0^t dt' D^4_{qqpp} \tilde c(s'q,t'p)[q,p]
\\
=& {\blue -I_n +  \int_0^s ds' \int_0^t dt' D^4_{qqpp} \tilde c(s'q,t'p)[q,p] }
\end{align}
{\blue since $D^3_{qpp} \tilde c(\zero,t'p)$ and $D^3_{qqp} \tilde
c(s'q,\zero)$ vanish in our chosen coordinates, and $-D^2_{pq}
\tilde c(\zero,\zero) {\RJM =I_n}$ is the identity matrix.} {\blue
Then, plugging \eqref{eq:Dpqtildec}  {\RJM into}
\eqref{eq:c-monotone-ineq}    yields}
\begin{align*}
- \langle q, p\rangle
&\le - \int_0^1 ds \int_0^1dt \int_0^s ds' \int_0^t dt' D^4_{qqpp} \tilde c(s' q,t'p)[q,q,p,p] \\
&\le C_0 |q|^2 |p|^2
\end{align*}
%since $D^3_{qpp} \tilde c(\zero,t'p)$ and $D^3_{qqp} \tilde
%c(s'q,\zero)$ vanish in our chosen coordinates and $-D^2_{pq}
%\tilde c(\zero,\zero)$ is the identity matrix.
{\blue
for constant $C_0= \sup_{(q,p) \in X_{\tilde y} \times V_{\tilde x}} \| D^4_{qqpp} \tilde c \|$.
Since the term $-D^4_{qqpp} \tilde c$ is exactly the cross-curvature \eqref{MTW}, its tensorial nature implies $C_0$ depends only on $\|c\|_{C^4(U \times V)}$ and the bi-Lipschitz constants
$\beta^\pm_c$ from \eqref{bi-Lipschitzbound}.
%Due to the
%tensorial nature of the cross-curvature \eqref{MTW}, $C_0$ depends
%on $\|c\|_{C^4(U \times V)}$ and the bi-Lipschitz constants
%$\beta^\pm_c$ from \eqref{bi-Lipschitzbound}.
}

From the above inequality and the definition of $E_{\theta,\e}$ we deduce
$$
p_n = \langle p, \hat e_n + \frac{q}{|q|}\rangle - \langle p,
\frac{q}{|q|}\rangle \le \theta |p| + C_0 \e |p|^2
$$
so $p \in E'_{\theta,C_0\e}$ as desired.
\end{proof}

\section{Alexandrov type estimates {\blue for $c$-convex functions}}\label{S:Alexandrov}

%We first recall a classical lemma due to Alexandrov:
%\begin{lemma}
%Let $\Q$ be an open convex set with $B_1 \subset \Q \subset B_{n}$,
%and let $\tu$ be a convex function with $\tu=0$ on $\p \Q$ (that
%is, $\Q$ is a section of $\tu$). Then
%$$
%|\tu(q)|^n \leq C(n) \dist(q,\p \Q) |\p\tu|(\Q).
%$$
%\end{lemma}
In this section we prove the key estimates for $c$-convex potential
functions $u$ which will eventually lead to the H\"older continuity
and injectivity of optimal maps. Namely, we extend Alexandrov type
estimates commonly used in the analysis of {\RJM convex solutions to
the Monge-Amp\`ere equation} with $c(x,y)= -\<x, y\>$, to {\RJM
$c$-convex solutions of the $c$-Monge-Amp\'ere} for general {\blue
\Bthree\ cost functions.} These estimates, Theorems \ref{thm:lower
Alex} and \ref{thm:estimate}, are of independent interest{\RJM :
they concern {\em sections} of $u$, i.e. the convex sub-level sets
of the modified potential $\tu$ of Theorem \ref{thm:apparently
convex}, and compare the range of values and boundary behaviour of
$\tu$ on each section with the volume of the section.} {\YHK
We describe them briefly before we begin the details.}

{\RJM Fix $p \in \R^n$, $p_0 \in \R$, and a positive symmetric
matrix $P>0$. It is elementary to see that the range of values taken
by the parabola $\tilde u(q) =q^t P q + p \cdot q +p_0$ on any
non-empty sub-level set $Q = \{\tu \le 0\}$ is determined by $\det
P$ and the volume of $Q$:
\begin{equation}\label{parabolic identity}
|\min_{q \in Q} \tu(q)|^n =  \left(\frac{\Leb{n}(Q)}{\Leb{n}(B_1)}\right)^2
\det P.
\end{equation}
Moreover,  the parabola tends to zero linearly as the boundary of
$Q$ is approached.  Should the parabola be replaced by a convex
function satisfying $\lambda \le \det D^2 \tu \le \Lambda$
throughout a fixed fraction of $Q$,  two of the cornerstones of
Caffarelli's regularity theory are that the identity
\eqref{parabolic identity} remains true --- up to a factor
controlled by $\lambda, \Lambda$ and the fraction of $Q$ on which
these bounds hold --- and moreover that $\tu$ tends to zero at a
rate no slower than $\dist_{\p Q}(q)^{1/n}$ as $q \to
\partial Q$. The present section is devoted to showing that under
\Bthree, similar estimates hold for level-set convex solutions $\tu$
of the $c$-Monge Amp\'ere equation, on sufficiently small sub-level
sets. Although the rate $\dist_{\p Q}(q)^{1/2^{n-1}}$ we obtain for
decay of $\tu$ near $\partial Q$ is probably far from optimal, it
turns out to be sufficient for our present purpose.}

 We {\RJM begin with a Lipschitz} estimate on the  cost
function $c$, {\blue which {\RJM turns} out to be useful}.

\begin{lemma}[\RJM Modified cost gradient direction is Lipschitz]\label{L:c estimate}
Assume \Bzero--\Btwo. {\RJM Fix $\ty \in \cl\V$.}
For $\tilde c \in C^3\big(\cl \U_\ty \times \cl\V\big)$ from Definition
\ref{D:cost exponential} and each  $q,\tilde q \in \cl \U_\ty$ and fixed target $y \in \cl \V$,
\begin{align}\label{slope compare}
|-D_q \tc(q,y)+D_q\tc(\tilde q,y)|
&\le  \frac{1}{\e_c}  | q - \tilde q|\,|D_q \tc(\tilde q,y)|,
%\\
%|-D_q \tc(q, y)| & \le (1+ C |q-\tilde q|) \, |-D_q \tc(\tilde q, y)|
\end{align}
where $\e_c$ is given by
$ %\begin{align}\label{const:e 0}
\e_c^{-1} = 2 (\beta^+_c)^4 (\beta^-_c)^6 \| D^3_{xxy}c \|_{L^\infty(\U \times \V)}
$ %\big{(}(c^{-(1,1)})^2 c_{(2,1)} + (c^{-(1,1)})^3 c_{(2,1)} c_{(1,1)}\big{)} c_{(1,1)} c^{-(1,1)}
%\end{align}
in the notation \eqref{bi-Lipschitzbound}. {\RJM If $y \ne \ty$ so
neither gradient vanishes, then}
\begin{align}\label{slope compare 2}
\left|-\frac{D_q \tc(q,y)}{|D_q \tc(q,y)|}+\frac{D_q\tc(\tilde q,y)}{|D_q\tc(\tilde q,y)|}\right|
&\le  \frac{2}{\e_c}  | q - \tilde q|.
%\\
%|-D_q \tc(q, y)| & \le (1+ C |q-\tilde q|) \, |-D_q \tc(\tilde q, y)|
\end{align}
\end{lemma}

\begin{proof}
For fixed $\tq \in \cl \U_\ty$ introduce the $\tilde
c$-exponential coordinates $p(y) = - D_q \tilde c(\tq, y)$. The
bi-Lipschitz constants \eqref{bi-Lipschitzbound} of this
coordinate change are estimated by $\beta^\pm_\tc \le \beta^+_c
\beta^-_c$ as in Corollary~\ref{C:Jacobian transform}.
Thus
\begin{align*}
\dist(y,\ty)
&\leq \beta^-_\tc |-D_q \tc(\tq, y) + D_q \tc(\tq, \ty )|\\
&\leq   \beta_c^+ \beta_c^-  |D_q \tc(\tilde q, y)|.
\end{align*}
where $\tc(q,\ty) \equiv 0$ from Definition \ref{D:cost
exponential} has been used. Similarly, noting the convexity \Btwo\
of $\V_\tq :=p(\V)$,
\begin{align*}
|-D_q \tc(\tilde q,y)+D_q\tc(q,y)|
&=|-D_q\tc(\tilde q,y)+D_q\tc(q,y)+D_q\tc(\tilde q,\ty)-D_q\tc(q,\ty)|
\\&\leq \|D^2_{qq} D_p \tc\|_{L^\i(\U_\ty \times \tilde \V_\tq)}  |\tq - q| |p(y)-p(\ty)|
\\&\leq \|D^2_{qq} D_y \tc\|_{L^\i(\U_\ty \times \V)} (\beta^-_c \beta^+_c)^2|\tq - q| \dist(y,\ty)
%\\&\leq \|D^2_{qq} D_y \tc\|_{L^\i(\U_\ty \times \V)} (\beta^+_c \beta^-_c)^3 |\tilde q -q | \, |D_q \tc(\tilde q, y)|.
\end{align*}
Then \eqref{slope compare} follows since $|D^2_{qq} D_y \tc| \le( (\beta^-_c)^2 +
\beta^+_c (\beta^-_c)^3))  |D^2_{xx} D_y c| \le 2 \beta^+_c
(\beta^-_c)^3 |D^2_{xx} D_y c| $ . (The last inequality
follows from $\beta^{+}_{c}\beta^-_c \ge 1$.)

Finally, to prove \eqref{slope compare 2} we simply use the triangle inequality and \eqref{slope compare} to get
\begin{align*}
\left|-\frac{D_q \tc(q,y)}{|D_q \tc(q,y)|}+\frac{D_q\tc(\tilde q,y)}{|D_q\tc(\tilde q,y)|}\right|
\leq & \left|-\frac{D_q \tc(q,y)}{|D_q \tc(q,y)|}+\frac{D_q \tc(q,y)}{|D_q\tc(\tilde q,y)|}\right|
+ \left|-\frac{D_q \tc(q,y)}{|D_q\tc(\tilde q,y)|}+ \frac{D_q\tc(\tilde q,y)}{|D_q\tc(\tilde q,y)|}\right|\\
=& \frac{\bigl|- |D_q\tc(q,y)|+|D_q\tc(\tilde q,y)|\bigr|}{|D_q\tc(\tilde q,y)|}
+ \frac{|-D_q \tc(q,y)+D_q\tc(\tilde q,y)|}{|D_q\tc(\tilde q,y)|}\\
\leq& \frac{2}{\e_c}| q - \tilde q|.
\end{align*}
\end{proof}

\subsection{Alexandrov type lower bounds}
In this subsection we prove {\blue one of the key estimates of this paper, namely,} a bound on the infimum on a $c$-convex function inside a section in terms of the measure of the section
and the mass of the $c$-Monge-Amp\`ere measure inside that section.
{\blue The corresponding result for the affine cost function is very easy, but in  our case the unavoidable nonlinearity of the cost function requires new ideas.

%\footnote{\RJM RM: $Z \to \Q$, $z \to q$, $z_v \to q_0$, $v' \to v_0$, $q_0 \to q_1$ throughout}

\begin{theorem}[Alexandrov type lower bounds]\label{thm:lower Alex}
Assume \Bzero--\Bthree, define
$\tilde c \in C^3\big(\cl\U_\ty \times \cl\V\big)$ as in
Definition \ref{D:cost exponential}, and let $\gamma^\pm_\tc=\gamma^\pm_\tc(\Q \times \V)$ be as in (\ref{Jacobian bound}).
Let $\tu:\cl \U_\ty \longmapsto \R$ be a $\tilde c$-convex function as in Theorem \ref{thm:apparently convex},
and let $\Q := \{ \tu \le 0\}\subset \cl \U_\ty$. Note that $\Q\subset \R^n$ is convex. Let $\e_c>0$ small be given
by Lemma~\ref{L:c estimate}.
Finally, let $E$ be the ellipsoid given by John's Lemma (see \eqref{E:well-centered}),
and assume there exists a small ellipsoid $E_\delta=x_0+\delta \rdot E \subset \Q$, with $\delta \leq \min\{1,\e_c/(4\diam(E))\}$ and
$|\partial^{\tc} \tu | \ge \lambda>0$ inside $E_\delta$. We also assume that $E_\delta \subset B \subset 4 \rdot B \subset U_{\bar y}$,
where $B$, $4 \rdot B$ are a ball and its dilation.
Then there exists a constant $C(n)$, depending only on the dimension, such that
\begin{equation}
\label{eq:Alex lower}
\Leb{n}(\Q)^2  \leq C(n) \frac{\gamma^-_\tc}{\delta^{2n}\lambda }\,  |\inf_\Q \tu|^n.
\end{equation}
\end{theorem}
}
The proof of the theorem above, {\blue which is given in the last part of this subsection,}
relies on the following result, which will also play a key role in Section \ref{S:H\"older}
(see \eqref{bound:gradient}) to show the
engulfing property and obtain {\blue H\"older continuity of optimal maps.} % between strongly $c$-convex domains.

\begin{lemma}[\blue Dual norm estimates]
\label{lemma:bound dual norm}
With the same notation {\blue and assumptions} as in Theorem \ref{thm:lower Alex},
let $\KK\subset \Q$ be {\blue an open} convex set such that $\diam(\KK)\leq \e_c/4$.
{\blue We also assume that {\red there exists a ball $B$ such that}
$\KK \subset B \subset 4 \rdot B \subset U_{\bar y}$, where $4 \rdot B$ denotes the dilation of $B$
by a factor $4$ with respect to its center. (This  assumption is to locate all relevant points
inside the domain where  the assumptions on $\tilde c $ hold.)}
Up to a translation, assume that the ellipsoid associated to $\KK$
by John's Lemma (see \ref{E:well-centered}) is centered at the origin $\zero$,
and for any $\rho\in (0,1)$ let $\rho \rdot \KK$ denote the dilation of $\KK$ with respect to the origin.
Moreover, let {\blue $\|\cdot\|_\KK^*$ denote} the ``dual norm'' associated to $\KK$, that is
\begin{equation}
\label{eq:dual norm}
\|v\|_\KK^*:= \sup_{w \in \KK} w \cdot v.
\end{equation}
Then, for any $\rho<1$ {\RJM setting $C_*(n,\rho):=\frac{8n}{(1-\rho)^2}$ implies}
%there exists a constant $C_*(n,\rho)=\frac{(1-\rho)^2}{8n}$, depending only on the dimension and $\rho$, such that
\begin{equation}
\label{eq:q bound} \|-D_q\tc(q,y)\|_\KK^*\leq C_*(n,\rho) |\inf_\KK
\tu| \qquad \forall\, q \in \rho \rdot \KK,\, y \in \p^\tc
\tu(\rho\rdot \KK).
\end{equation}
\end{lemma}

\begin{figure}
\begin{center}
\centerline{\epsfysize=2.0truein\epsfbox{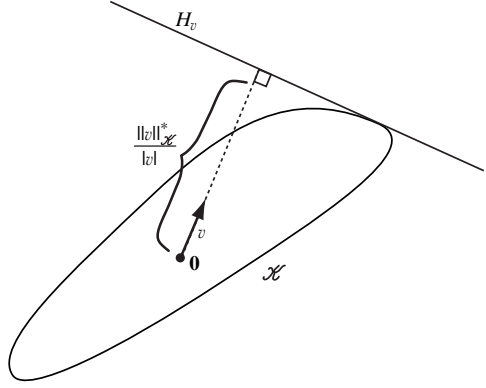}}\caption{{\small
The quantity $\frac{\|v\|_\KK^*}{|v|}$ represents the distance between the origin and
the supporting hyperplane orthogonal to $v$.}}\label{fig2bis-A}
\end{center}
\end{figure}

Before proving the above lemma, let us explain the geometric
intuition behind the result: for $v \in \R^n \setminus\{\zero\}$,
let $H_v$ be the supporting hyperplane to $\KK$ orthogonal to $v$
and contained inside the half-space $\{q \mid q\cdot v >0\}$. Then
$$
\dist(\zero,H_v)=\sup_{w \in \KK} w \cdot \frac{v}{|v|}=\frac{\|v\|_\KK^*}{|v|},
$$
see Figure \ref{fig2bis-A}, and
Lemma \ref{lemma:bound dual norm}
states that, for all $v=-D_q\tc(q,y)$, with $q \in \rho \rdot \KK$ and $y \in \p^\tc \tu(\rho \rdot \KK)$,
$$
\dist(\zero,H_v) |v| \leq C_*(n,\rho) |\inf_\KK \tu|.
$$
Hence, roughly speaking, \eqref{eq:q bound} is just telling us that
the size of gradient at a point inside $\rho \rdot \KK$ times the
width of $\KK$ in the direction orthogonal to the gradient is
controlled, up to a factor $C_*(n,\rho)$, by the infimum of $\tu$
inside $\KK$ (all this provided the diameter of $\KK$ is
sufficiently small). This would be a standard estimate if we were
working with convex functions and $c$ was the standard quadratic
cost in $\R^n$, but in our situation the proof is {\RJM both subtle
and involved}.

{\blue   To prove \eqref{eq:q bound}}
let us start observing that, since $|w| \leq \diam(\KK)$ for $w \in \KK$, the following {\blue useful}  inequality holds:
\begin{equation}
\label{eq:comparison norms}
\|v\|_\KK^* \leq \diam(\KK)\,|v| \qquad \forall \,v \in \R^n.
\end{equation}
{\blue  We will need} two preliminary results.
\begin{lemma}

\label{lem:line from 0}
{\blue With the same notation and assumptions as in Lemma \ref{lemma:bound dual norm}},
let {\blue $q \in \rho \rdot \KK$,
$y \in \V$,}
%$y \in \p^\tc \tu(\rho \KK)$,
and let $m_y$ be a function of the form
$m_y:=-\tc (\cdot,y) + C_y$ for some constant $C_y \in \R.$
Set $v :=-D_q m_y(q)$,
assume that $\KK \subset \{m_y < 0\}$, and
let $\q_0$ denote the intersection of the half-line
$\ell_v^\q:=\q+\R^+v=\{\q+tv\,|\,t >0 \}$ with $\{m_y=0\}$ (assuming it exists {\blue in $U_\ty$}).
{\blue Define
\begin{align*}
 \hat \q_0:=
\left\{
\begin{array}{ll}
\q_0 & \text{if }|\q_0-q| \leq \diam \KK,\\
\q + \frac{\diam (\KK)}{|v|} v & \text{if $|\q_0-q| \geq \diam (\KK)$  or $\q_0$ does not exist  in $U_\ty$},\\
\end{array}
\right.
\end{align*}
(Notice that, by the assumption $\KK \subset B \subset 4 \rdot B \subset U_\ty$, we have $\hat \q_0 \in U_\ty$.)}
Then
\begin{align}\label{eq:hatzv-z}
{\blue 2|\hat \q_0-\q|} \geq (1-\rho) \frac{\|v\|_\KK^*}{|v|}.
\end{align}
\end{lemma}
\begin{figure}
\begin{center}
\centerline{\epsfysize=2.0truein\epsfbox{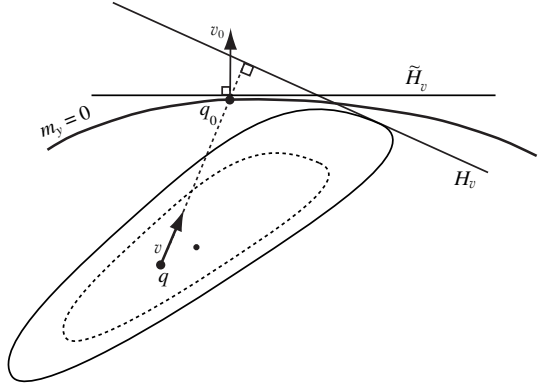}}\caption{{\small
It is easily seen that $|\q_0-\q|$ controls $\dist(\zero,H_{\tilde v})$,
and using that the gradient of $m_y$ does not vary much we deduce that the latter is close to
$\dist(\zero,H_{v})$ in terms of $|\q_0-\q|$.}}\label{fig2bis-B}
\end{center}
\end{figure}

\begin{proof}
%\footnote{YH: it would be better if we provide some figure(s) in this section....
%AF: They should be ready by tomorrow :)}
{\blue   In the case $\hat \q_0 = \q+ \diam (\KK) v$, the inequality \eqref{eq:hatzv-z} follows from \eqref{eq:comparison norms}. Thus, we can assume that $\q_0$ exists in $U_\ty$  and that  $\hat \q_0 = \q_0$. }
Let us recall that $\{m_y<0\}$ is convex (see Theorem \ref{thm:apparently convex}).
Now, let
{\blue $\tilde H_v $}
%$\tilde H_v \subset \{y\,:\,y\cdot v >0\}$
 denote the hyperplane
tangent to $\{m_y=0\}$ at $\q_0$.
Let us observe that $\tilde H_v$ is orthogonal to the vector $v_0:= - D_q m_y(\q_0)$, see Figure \ref{fig2bis-B}.
%{\blue Notice that by the assumption $\KK \subset B \subset 4 \rdot B \subset U_\ty$, the intersection point $\q_0'$ is in $U_\ty$.}
Since $\KK\subset \{m_y<0\}$, we have $\tilde H_v \cap \KK=\emptyset$, which implies
$$
\dist(\zero,\tilde H_v)\geq \sup_{w \in \KK} w \cdot \frac{v_0}{|v_0|}=\frac{\|v_0\|_\KK^*}{|v_0|}.
$$
Moreover, since $\q \in \rho \rdot \KK$ and $\tilde H_v \cap
\KK=\emptyset$, we also have
$$
\dist(\q,\tilde H_v) \geq (1-\rho) \dist(\zero,\tilde H_v).
$$
Hence, observing that
$\q_0 \in \tilde H_v$ we obtain
$$
|\q_0 - \q| \geq \dist(\q,\tilde H_v) \geq (1-\rho) \frac{\|v_0\|_\KK^*}{|v_0|}.
$$
Now, to conclude the proof, we observe that \eqref{slope compare 2} applied with $v_0=-D_q \tc (\q_0,y)$ and $v=-D_q \tc (\q,y)$, together with \eqref{eq:comparison norms}
and the assumption $\diam(\KK) \leq \e_c/4$, implies
$$
\left\| \frac{v}{|v|} - \frac{v_0}{|v_0|}\right\|_\KK^* \leq \diam(\KK)\left| \frac{v}{|v|} - \frac{v_0}{|v_0|}\right| \leq  \frac{2\diam(\KK)}{\e_c} |\q_0-\q| \leq \frac{|\q_0-\q|}{2}.
$$
Combining all together and using the triangle inequality for $\|\cdot\|_\KK^*$ (which
is a consequence of the convexity of $\KK$) we finally obtain
\begin{align*}
 \left\| \frac{v}{|v|}\right\|_\KK^*
& \leq \left\| \frac{v_0}{|v_0|}\right\|_\KK^*+
\left\| \frac{v}{|v|} - \frac{v_0}{|v_0|}\right\|_\KK^* \\
& \leq
\frac{|\q_0-\q|}{1-\rho}+ \frac{|\q_0-\q|}{2} \\
& \leq \frac{2|\q_0-\q|}{1-\rho},
%\\ & {\blue = \frac{2|\hat \q_0-\q|}{1-\rho},}
\end{align*}
as desired.
\end{proof}

\begin{lemma}
\label{lem:line from z}
{\blue With the same notation and assumptions as in  Lemmata \ref{lemma:bound dual norm} and  \ref{lem:line from 0},} fix $\q' \in \rho \KK$, and
let $\ell_v^{\q'}$ denote the half line $\ell_v^{\q'} :=\q'+\R^+v=\{ \q' + t v\ | \,t> 0 \}$.
Denote by $\q_0'$ the intersection of $\ell_v^{\q'}$ with $\{m_y=0\}$
{\blue (assuming it exists  in $U_\ty$)}.
Then,
%there exists a constant $C(n,\rho)$, depending only on the dimension and $\rho$, such that
\begin{align*}
{\blue |\q_0' - \q'| \geq \frac{1-\rho}{2n} |\hat \q_0-\q|.}
\end{align*}
\end{lemma}
\begin{proof}
%{\blue We emphasize that the following proof uses only the convexity.}

Let $\ell_{\q\q'}$ and {\blue $\ell_{\hat \q_0\q_0'}$} denote the
lines passing through $\q,\q'$ and {\blue $\hat \q_0,\q_0'$}
respectively, and denote by $\q_1:=\ell_{\q\q'}\cap \ell_{\hat
\q_0\q_0'}$ their intersection point (see Figure \ref{fig2bis-C}); {\RJM since the first four
points lie in the same plane,  we can (after slightly perturbing
$\q$ or $\q'$ if necessary) assume this intersection} exists and is
unique. {\blue Note that in the following we will only use
convexity of $\KK$, and in particular we do not require $\q_1 \in
U_\ty$.}

{\blue Let us now distinguish two cases, depending whether $|\q-\q_1| \leq |\q'-\q_1|$ or
$|\q-\q_1| \geq |\q'-\q_1|$.

 {\em $\bullet$ $|\q-\q_1| \leq |\q'-\q_1|$}:
In this case} we simply observe that, since {\blue $\hat \q_0-\q$ and $\q_0'-\q'$} are parallel, by similarity
$$
1 \le \frac{|\q'-\q_1|}{|\q-\q_1|}={\blue \frac{| \q_0'-\q'|}{|\hat \q_0-\q|}},
$$
and so the result is proved {\blue since  $1 \ge \frac{1-\rho}{2n}$.} %$C(n,\rho)=1$.

\begin{figure}
\begin{center}
\centerline{\epsfysize=2.0truein\epsfbox{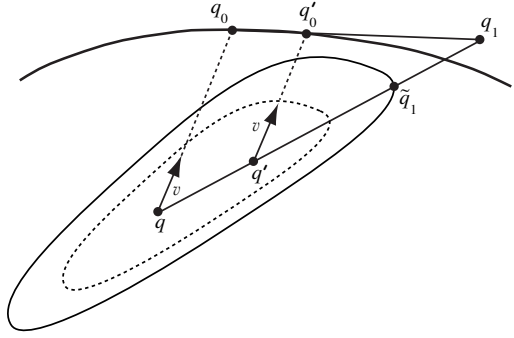}}\caption{{\small
In this figure we assume $\hat q_0=q_0$.
Using similarity, we only need to bound from below the ratio $\frac{|\q'-\q_1|}{|\q-\q_1|}$,
and the latter is greater or equal than $\frac{|\q'-\tilde \q_1|}{|\q-\tilde \q_1|}$.
An easy geometric argument allows to control this last quantity.
.}}\label{fig2bis-C}
\end{center}
\end{figure}

{\em $\bullet$ $|\q-\q_1 {\RJM | \geq |} \q'-\q_1|$}:
In this second case,
we first claim that $\q_1 \not\in \KK$. Indeed, we observe that
since $\hat \q_0-\q$ and $\q_0'-\q'$ are parallel, the point $\q_1$ cannot belong to the segment joining $\hat \q_0$
%\footnote{YHK: I corrected $\q_0$ to $\hat \q_0$ as in Alessio's email.}
and $\q_0'$. Hence, since $\hat \q_0 \in \{ m_y \le 0 \}$, $\q_0' \in \{m_y=0\}$, and $\KK \subset \{m_y <0\}$ by assumption, by
the convexity of the set $\{m_y \leq 0\}$
we necessarily have {\blue $\q_1$ outside $\{m_y < 0\}$, thus outside of $\KK$,} which proves the claim.

Thus, we can find the point $\tilde \q_1$ obtained by intersecting $\p\KK$
with the segment going from $\q'$ to $\q_1$, and
by the elementary inequality
$$
\frac{a+c}{b+c} \geq \frac{a}{b}\qquad \forall\, 0<  a \leq b,\, c \geq 0,
$$
we get
$$
\frac{|\q'-\q_1|}{|\q-\q_1|} \geq \frac{|\q'-\tilde \q_1|}{|\q-\tilde \q_1|}.
$$
Now, to estimate the right hand side from below, let $L$ be the affine transformation provided by John's Lemma such that
$B_1 \subset L^{-1}(\KK) \subset B_n$. Since the points $\q,\q',\tilde \q_1$ are aligned, we have
$$
\frac{|\q'-\tilde \q_1|}{|\q-\tilde \q_1|}=\frac{|L^{-1}\q'-L^{-1}\tilde \q_1|}{|L^{-1}\q-L^{-1}\tilde \q_1|}.
$$
Now, since $L^{-1}\q,L^{-1}\tilde \q_1 \in B_n$, we immediately get
$$
|L^{-1}\q-L^{-1}\tilde \q_1| \leq 2n.
$$
On the other hand, since $L^{-1}\q' \in \rho \rdot \bigl(L^{-1}(\KK)\bigr)$
while $L^{-1}\tilde \q_1
\in \partial \bigl(L^{-1}(\KK)\bigr)$,
%it is easily seen (for instance by a simple compactness argument) that there exists a constant $c(n,\rho)>0$,
%depending only on $\rho$ and the dimension, such that
{\blue it follows from $B_1 \subset L^{-1}(\KK)$ that
$$
|L^{-1}\q'-L^{-1}\tilde \q_1| \, \geq  \, 1-\rho.
$$
To see this last inequality, for each point $w \in \p \big(L^{-1}(\KK)\big)$ consider the
convex hull $\mathcal{C}_w \subset L^{-1}(\KK)$ of $\{w \} \cup B_1$. Then, centered at
the point $ \rho w \in \p\Big( \rho \rdot \big(L^{-1}(\KK)\big)\Big)$ there exists a ball
of radius $1-\rho$ contained in $\mathcal{C}_w$, thus in $L^{-1}(\KK)$. This shows that the
distance from any point in $\rho\big( L^{-1}(\KK)\big)$ to $\p \big(L^{-1}(\KK)\big)$ is at
least  $1-\rho$, proving the inequality.

From these inequalities and using again similarity, we get
$$
\frac{|\q_0'-\q'|}{|\hat \q_0-\q|} = \frac{|\q'-\q_1|}{|\q-\q_1|} \geq \frac{1-\rho}{2n},
$$
concluding the proof.}
%and we conclude as in the first case, now with $C(n,\rho)=\frac{2n}{c(n,\rho)}$.
\end{proof}

\begin{proof}[\bf \blue  Proof of Lemma \ref{lemma:bound dual norm}]
%Now we can prove \eqref{eq:q bound}.
Fix $v=-D_q\tc(\q,y)$, with {\blue $\q \in \rho \rdot \KK$,} $\q'
\in \rho \rdot \KK$ and $y \in \p^\tc \tu(\q')$, {\blue and define
the function $m_y (\cdot):=-\tc (\cdot,y) + \tc (\q', y) + \tu
(\q')$. Notice that since $y \in \p^\tc \tu (\q')$, the inclusion
$\Q = \{ \tu \le 0\} \subset \{ m_y \le 0\}$ holds.
%Let  $\ell_v^{\q'}$ denote  the half line  $\q'+\R^+v$.
%Set
%\begin{align*}
%%\q_0 := \ell_{v}^{\q} \cap \{m_y=0\}, \quad
%\q_0' := \ell_{v}^{\q'} \cap \{m_y=0\} \qquad \hbox{ (assuming such point  exists {\blue  in $U_\ty$})}.
%\end{align*}

Recall the point $\q_0'$ from Lemma \ref{lem:line from z} (when it exists in $U_\ty$), and the point $\hat \q_0 \in U_\ty$ from Lemma \ref{lem:line from 0}. Define,
%Let us now define the points $\hat \q_0$ and  $\hat \q_0'$ as
%$$
%\q_0'':=
%\left\{
%\begin{array}{ll}
%\q_0' & \text{if }|\q_0' - \q'| \leq \e_c/4,\\
%\q' + \frac{\e_c}{4}v & \text{if $|\q_0' - \q'| \geq \e_c/4$  or $\q_0'$ does not exist {\blue in $U_\ty$}},\\
%\end{array}
%\right.
%$$
\begin{align*}
%&\hat \q_0:=
%\left\{
%\begin{array}{ll}
%\q_0 & \text{if }|\q_0 - \q| \leq \diam \KK,\\
%\q + \diam (\KK) v & \text{if $|\q_0 - \q| \geq \diam (\KK)$  or $\q_0$ does not exist  in $U_\ty$},\\
%\end{array}
%\right.
%\\
& \hat \q_0':=
\left\{
\begin{array}{ll}
\q_0' & \text{if }|\q_0' - \q'| \leq \diam \KK,\\
\q' + \frac{\diam (\KK)}{|v|}  v & \text{if $|\q_0' - \q'| \geq \diam (\KK)$  or $\q_0'$ does not exist  in $U_\ty$},\\
\end{array}
\right.
\end{align*}
Notice that by the assumption $\KK \subset B \subset 4 \rdot B \subset U_\ty$,  the points $\hat \q_0$, $\hat \q_0'$ belong to $U_\ty$.
Let
 $[\q', \hat \q_0'] \subset U_\ty$} denote the segment going from $\q'$ to {\blue $\hat \q_0'$}. Then,
since ${\blue \hat \q_0'} \in \{m_y \leq 0\}$ and $\tu$ is negative inside $\KK\subset \Q$, we get
\begin{align}\label{eq:infKKtu}\nonumber
|\inf_\KK \tu| &\ge - \tu(\q') \geq -\tu(\q')  + \tu(\hat \q_0') \\
 &\geq -\tc( \hat \q_0', y )  + \tc(\q', y) = \int_{[\q', \hat \q_0']} \langle -D_q \tc(w, y),
 \frac{v}{|v|} \rangle \,dw
\end{align}
{\blue Since $\q,\q' \in \KK$ and  $|\hat \q_0' - \q'| \leq \diam (\KK)$,
% Since $\q,\q' \in \KK$, $\diam(\KK) \leq \e_c/4$, and $|\q_0'' - \q'| \leq \e_c/4$,
we have
\begin{align*}
|\q-w| \leq 2 \diam (\KK) \le \e_c/2 \qquad \forall\,w \in [\q',\hat \q_0'].
\end{align*}
}
Hence, recalling that $-D_q \tc (\q, y) = v$ and applying Lemma~\ref{L:c estimate}, we obtain
\begin{align*}
 \langle -D_q \tc(w, y),  \frac{v}{|v|} \rangle \ge \frac{|v|}{2}  \qquad \forall\,w {\blue \in [\q',\hat \q_0'].}
\end{align*}
%Thus, thanks to Lemmas \ref{lem:line from 0} and \ref{lem:line from z}, we finally conclude
{\blue Applying this and Lemmas \ref{lem:line from 0} and \ref{lem:line from z} to \eqref{eq:infKKtu}, we finally conclude
\begin{align*}
 |\inf_\KK \tu| &\geq \min\left\{\diam (\KK), |\q_0' - \q'|\right\} \frac{|v|}{2} \qquad \hbox{(from \eqref{eq:infKKtu})}\\   &\geq  \min\left\{\diam (\KK), \frac{1-\rho}{2n} |\hat \q_0-\q| \right\} \frac{|v|}{2} \qquad \hbox{(from Lemma  \ref{lem:line from z})}\\ & \geq  \min\left\{\diam (\KK), \frac{(1-\rho)^2}{4n}\frac{\|v\|_\KK^*}{|v|}\right\} \frac{|v|}{2} \qquad \hbox{(from Lemma  \ref{lem:line from 0})}\\
 &=  \min\left\{\frac{1}{2} \diam (\KK) |v|,  \frac{(1-\rho)^2}{8n}\|v\|_\KK^*\right\}\\
 &  \ge  \min\left\{\frac{1}{2}\|v\|_\KK^*,  \frac{(1-\rho)^2}{8n} \|v\|_\KK^*\right\}   \qquad \hbox{(from \eqref{eq:comparison norms} }\\
& = \frac{(1-\rho)^2}{8n} \|v\|_\KK^*
\end{align*}
}
%\begin{align*}
% |\inf_\KK \tu| &\geq \min\left\{\frac{\e_c}4, |\q_0' - \q'|\right\} \frac{|v|}{2} \\ & \geq  \min\left\{\frac{\e_c}4, {\blue\frac{(1-\rho)^2}{4n}}\frac{\|v\|_\KK^*}{|v|}\right\} \frac{|v|}{2}\\
% &= \frac{1}{4} \min\left\{\frac{\e_c}{2}|v|,  {\blue\frac{(1-\rho)^2}{2n}}\|v\|_\KK^*\right\}\\
% & {\blue \ge \frac{1}{4} \min\left\{2\|v\|_\KK^*,  {\blue\frac{(1-\rho)^2}{2n}}\|v\|_\KK^*\right\} \ \ \hbox{(from \eqref{eq:comparison norms} and recalling that $\diam(\KK) \leq \frac{e_c}{4}$)} }\\
% & {\blue \ge \frac{(1-\rho)^2}{8n} \|v\|_\KK^* }
%\end{align*}
{\blue as desired, completing the proof.}
%from which \eqref{eq:q bound} follows using
%\eqref{eq:comparison norms} and recalling that $\diam(\KK) \leq \e_c/4$.
\end{proof}

\begin{proof}[{\blue \bf Proof of Theorem \ref{thm:lower Alex}}]
Set $h := |\inf_\Q \tu |$. Since $ \Leb{n} (\Q) \leq n^n  \Leb{n} (E)$, it is enough to show
\begin{align}\label{eq:box ineq}
\lambda \Leb{n}(E)^2 \le C(n) \gamma^-_\tc \frac{h^n}{\delta^{2n}}.
\end{align}
So, the rest of the proof is devoted to \eqref{eq:box ineq}.
Let $E_\delta = x_0 + \delta \rdot E$ be defined as in the statement of the proposition.
With no loss of generality, up to a change of coordinates we can assume that $E_\delta$ is
of the form $\{{\RJM q \mid \sum_i a_i^2q_i}^2<1\}$.
Define
\begin{align*}
 \mathcal{C}_0 = \left\{ q_0 \in T_\zero \Q \ | \ q_0 = -D_q \tc(\zero, y), \ \ y \in \partial^{\tc} \tu \big({\textstyle \RJM \frac{1}{2}} \rdot E_\delta\big)\right\}.
\end{align*}
Notice that
\begin{align}
\label{eq:first bound}
\Leb{n}\big( \partial ^{\tc} \tu \big({\textstyle \RJM \frac{1}{2}} \rdot E_\delta\big)\big) \le \gamma^-_\tc \Leb{n} ( \mathcal{C}_0)
\end{align}
(see (\ref{Jacobian bound})).
Now, let us define the norm
$$
\|q\|_{E_\delta}^*:= \sqrt{\sum_{i=1}^n\frac{q_i^2}{a_i^2}}=\sup_{v \in E_\delta} v \cdot q.
$$
%Since $\diam(E_\delta)=\delta \diam(E) \leq \e_c/4$,
{\blue Since $\KK=E_\delta$ satisfies the assumptions for Lemma \ref{lemma:bound dual norm},
%we can apply Lemma \ref{lemma:bound dual norm} with $\KK=E_\delta$ and $\rho=1/2$
choosing $\rho=1/2$} we get
\begin{equation}
\label{eq:bound q0}
\|q_0\|_{E_\delta}^*\leq {\red 32 n} h  \hbox{ for all  $q_0 \in \mathcal{C}_0$}.
\end{equation}
Let $\Phi_h: \R^n  \to \R^n$ denote the linear map
$ \Phi_h(x) :=  {\red 32 n}  h \left( a_1x_1, \ldots, a_nx_n\right).$ Then  {\blue from  \eqref{eq:bound q0} it follows}
\begin{align}\label{eq: C inclusion}
\mathcal{C}_0 \subset \Phi_h (B_1).
\end{align}
Since
 \begin{align}\label{eq: cone bound}
 \Leb{n} (\Phi_h (B_1))) \leq C(n)
  \frac{h^n}{\Leb{n} (E_\delta)},
\end{align}
combining \eqref{eq:first bound}, \eqref{eq: C inclusion}, and \eqref{eq: cone bound}, we get
\begin{align*}
\Leb{n}\big( \partial ^{\tc} \tu \big({\textstyle \RJM \frac{1}{2}} \rdot E_\delta\big)\big) \le C(n) \gamma^-_\tc \frac{h^n}{\Leb{n}(E_\delta)}. % \approx \frac{h^n}{\Leb{n} (\Q)}
\end{align*}
As
$\Leb{n}\big( \partial ^{\tc} \tu \big({\textstyle \RJM \frac{1}{2}} \rdot E_\delta\big)\big)  \ge \lambda 2^{-n} \Leb{n} (E_\delta)$
and $\Leb{n}(E_\delta) = \delta^{n}\Leb{n} (E)$, we obtain the desired conclusion.
\end{proof}

\subsection{{\RJM Bounds for $\tc$-cones over convex sets}}\label{S:c-cone}

%\subsubsection{$\tc$-cones over convex sets}
%{$c$-cones associated to sections of a $c$-convex function}\label{SS:c-cone}

We now progress toward the Alexandrov type {\blue upper bounds in
Theorem}~{\RJM \ref{thm:estimate}.}
%\ref{L:prop-infleq}.}
 In this subsection we construct and
study the $\tc$-cone associated to the section of a $\tc$-convex
function. This $\tc$-cone {\RJM --- whose entire
$\tc$-Monge-Amp\`ere mass concentrates at a single prescribed point
---} plays an essential role in our proof of
%the Alexandrov type estimate
Lemma~\ref{L:prop-infleq}.

%\footnote{RJM: perhaps just write $h$ instead of $h^\tc$, throughout?}

\begin{definition}[$\tc$-cone]
\label{def:ccone}
Assume \Bzero--\Bthree, % and take $\gamma^\pm_c$ as in (\ref{Jacobian bound}).
and let $\tu:\cl \U_\ty \longmapsto \R$ be the $\tilde c$-convex function with
convex level sets from Theorem~\ref{thm:apparently convex}.
%, with $\tu \ge 0$ on $\p \U_\tq$.
%and let $\tc:\cl \U_\tq \times \V \longmapsto \R$ denote the modified cost
%from Definition \ref{D:cost exponential}.
Let $\Q$ denote the section $\{ \tilde u \le 0\}$, fix $\tilde q
\in \intr \Q$, and assume $\tu = 0$ on $\p \Q$ {\blue and $\cl \Q \subset U_{\ty}$}. The \emph{$\tc$-cone
$h^\tc: \U_\ty \longmapsto \R$ generated by $\tq$ and $\Q$ with
height $-\tu(\tq)>0$} is given by
\begin{align}\label{c-cone}
h^\tc(q)  :=\sup_{y \in \cl \V}
\{ -\tc(q, y) + \tc(\tq, y) + \tu(\tq) \mid
   -\tc(q, y) + \tc(\tilde q, y) +\tu(\tq)\le 0 \text{ on } \p \Q \}.
%h^c (q)  :=\sup_{(z, y)} \{ \ \bar m_z (q) \mid
%z \in \p \Q, \ \ y \in \p^c \tu (z) \}.
\end{align}
\end{definition}

Notice the $\tc$-cone $h^\tc$ depends only on the convex set $\Q \subset {\blue \U_\ty}$, $\tq \in \intr \Q$,
and the value $\tu(\tq)$, but is otherwise independent of $\tilde u$.
Recalling that $\tc(q, \ty) \equiv 0$ on $\U_{\ty}$, we record several key properties
of the $\tc$-cone:

\begin{lemma}[Basic properties of $\tc$-cones]
\label{lemma:propccone} Adopting the notation and hypotheses of
Definition \ref{def:ccone}, let $h^\tc: {\blue \U_\tq} \longmapsto
\R$ be the $\tc$-cone generated by $\tilde q$ and $\Q$ with height
$-\tu (\tilde q){\RJM >0}$. Then
\begin{enumerate}
\item[(a)] $h^\tc$ has convex level sets; {\RJM furthermore,} it is a convex function if \Bfour\ holds;
\item[(b)] $h^\tc(q) \geq h^\tc(\tilde q) = \tu(\tilde q)$ for all $q \in \Q$;
\item[(c)] $h^\tc = 0 \text{ on } \p \Q$;
\item[(d)] $\p^\tc h^\tc (\tilde q) \subset \p^\tc \tu (\Q)$.
\end{enumerate}
\end{lemma}

%\footnote{\RJM RM: changed $z \to q_0$ $\sigma \to y$, ?$_z \to$
%?$_0$, $\bar m_0 \to m_t$, $ 0 \leftrightarrow 1$ etc. in this
%proof}
\begin{proof}
Property (a) is a consequence of the level-set convexity of $q
\longmapsto -\tc(q, y)$ proved in Theorem \ref{thm:apparently
convex}, or its convexity assuming \Bfour. Moreover, since
$-\tc(q,\ty)+\tc(\tilde q,\ty)+\tu(\tilde q)=\tu(\tilde q)$ for all
$q\in \U_\ty$, (b) follows.

For each pair $q_0 \in \p \Q$ and $y_0
\in \p^\tc \tu (q_0)$, consider the supporting mountain $m_0 (q)=
-\tc(q, y_0 ) + \tc(q_0, y_0) $, i.e. $m_0(q_0)=0=\tu (q_0)$ and
$m_0 \le \tu$. Consider the $\tc$-segment $y(t)$ connecting
$y(0)=y_0$ and $y(1)= \ty$ in $\V$ with respect to $q_0$. Since
$-\tc(q, \ty) \equiv 0$, by continuity there exists some $t
 \in [0, 1[$ for which $m_t (q):= -\tc(q, y(t)) +
\tc(q_0, y(t))$ satisfies $m_t (\tilde q) = \tu(\tilde q)$. From
Loeper's maximum principle (Theorem \ref{T:DASM} above), we have
$$
m_t \le \max[ m_0 , -\tc(\cdot, \ty) {\blue +  \tc (q_0, \ty)}]=
\max[m_0 , 0],
$$
and therefore, from $m_0 \le \tu$,
$$
m_t \le 0 \text{ on } \Q.
$$
By the construction, $m_t$ is of the form
$$
{\blue -\tc(\cdot, y(t)) + \tc(\tilde q, y(t))} + \tu (\tilde q),
$$
and vanishes at $q_0$. This proves (c). Finally (d) follows from (c)
and  the fact that $h^\tc (\tilde q ) = \tu(\tilde q)$. Indeed, it
suffices to move down the supporting mountain of $h^\tc$ at $\tilde
q$ until the last moment at which it touches the graph of $\tu$ inside
$\Q$. The conclusion then follows from Loeper's local to
global principle, Corollary \ref{C:local-global} above.
%to Loeper's maximum principle.
\end{proof}

%For the $c$-subdifferential image $\partial^c h (\tilde q)$
The following estimate shows that the Monge-Amp\`ere measure, and
the relative location of the vertex within the section which
generates it, control the height of any well-localized $\tc$-cone.
Together with
Lemma~\ref{lemma:propccone}(d), this proposition plays a key role
in the proof of our Alexandrov type estimate
(Lemma~\ref{L:prop-infleq}).

%\footnote{RJM: No $\tu$ in this proposition;  only $h^\tc$!}

\begin{proposition}%[Monge-Amp\`ere measure and distance to the boundary control the height of a $\tc$-cone]
[Lower bound on the Monge-Amp\`ere measure of a small $\tc$-cone]
\label{P:partial c-cone} Assume \Bzero--\Bthree, and define
$\tilde c \in C^3\big(\cl\U_\ty \times \cl\V\big)$ as in
Definition \ref{D:cost exponential}. Let  $\Q \subset {\blue  \U_\ty}$
be a compact convex set, and $h^\tc$ the $\tc$-cone generated by
$\tq \in \intr \Q$ of height $-h^\tc(\tq)>0$ over $\Q$. Let
$\Pi^+,\Pi^-$ be two parallel hyperplanes contained in $T^*_\ty \V
\setminus \Q$ and touching $\p \Q$ from two opposite sides.
{\blue
We also assume that {\red there exists a ball $B$ such that} $\Q \subset B \subset 4 \rdot B \subset U_{\bar y}$.}
Then
there exist $\e_c>0$ small, depending only on the cost (and given
by Lemma~\ref{L:c estimate}), and a constant $C(n)>0$ depending
only on {\blue the} dimension, such that the following holds:

If $\diam(\Q) \leq {\RJM \e_c' :=} \e_c/C(n)$,
then
%\footnote{RJM: Only  $C(n)$!}
\begin{align}\label{p c-cone lower}
|h^\tc (\tilde q)|^n \leq C(n) \frac{\min\{\dist(\tilde
q,\Pi^+),\dist(\tilde q,\Pi^-)\}}{\ell_{\Pi^+}}|\p h^\tc|(\{\tilde
q\}) \Leb{n}(\Q),
\end{align}
%\begin{multline}\label{p c-cone lower}
%|\tu (\tilde q)|^n \leq C(n,\|c\|_{C^3(\Q \times \V)}, \|(D^2_{qy}c)^{-1}\|_{L^\i(\Q \times \V)})\\
%\frac{\min\{\dist(\tilde q,\Pi^+),\dist(\tilde q,\Pi^-)\}}{\ell_{\tilde q}}|\p h^\tc|(\tilde q)|\Q|
%\qquad \forall\,\tilde q \in \Q,
%\end{multline}
where $\ell_{\Pi^+}$ denotes the maximal length among all
the segments obtained by intersecting $\Q$ with a line orthogonal
to $\Pi^+$.
%which passes through $\tilde q$.
\end{proposition}

\begin{proof}
We fix $\tilde q {\blue \in \intr \Q}$.
Let $\Pi^i$, $i=1, \cdots n$, (with $\Pi^1$ equal either $\Pi^+ $ or $\Pi^-$)  be hyperplanes contained in $T^*_\ty \V \setminus \Q \simeq \R^n
\setminus \Q$, touching $\p \Q$, and such that
$\{\Pi^+,\Pi^2,\ldots,\Pi^n\}$ are all mutually orthogonal (so
that $\{\Pi^-,\Pi^2,\ldots,\Pi^n\}$ are {\RJM also} mutually
orthogonal).
Moreover we choose $\{\Pi^2,\ldots,\Pi^n\}$ in such a way that,
if $\pi^1(\Q)$ denotes the projection of $\Q$ on $\Pi^1$
and $\Haus{n-1}(\pi^1(\Q))$ denotes its $(n-1)$-dimensional Hausdorff measure, then
\begin{equation}
\label{eq:choice hyperpl}
C(n)\Haus{n-1}(\pi^1(\Q))\geq \prod_{i=2}^n \dist(\tilde
q,\Pi^i),
\end{equation}
for some universal constant $C(n)$. Indeed, as $\pi^1(\Q)$ is
convex, by Lemma~\ref{L:John} we can find an ellipsoid $E$ such
that $E \subset \pi^1(\Q) \subset (n-1) \rdot E$, and for instance we can
choose $\{\Pi^2,\ldots,\Pi^n\}$ among the hyperplanes orthogonal
to the axes of the ellipsoid (for each axis we have two possible
hyperplanes, and we can choose either of
them).
%\footnote{RJM: lower indices!}

Each hyperplane $\Pi^i $ touches $\Q$ from outside, say at $q^i \in
T^*_\ty \V$. Let $p_i \in T_\ty V$ be the outward (from $\Q$) unit
vector at $q^i$ orthogonal to $\Pi^i$. Then $s_i p_i \in \p h^\tc
(q^i)$ for some $s_i >0$, and by Corollary \ref{C:local-global}
there exists $y_i \in \p^\tc h^\tc(q^i)$ such that $$ -D_q \tc(
q^i , y_i) = s_i p_i .
$$
Define $y_i(t)$ as
$$
-D_q \tc(q^i, y_i (t) ) = t\, s_i p_i,
$$
i.e. $y_i(t)$ is the $\tc$-segment from $\ty$ to $y_i$ with respect to $q^i$.
As in the proof of Lemma~\ref{lemma:propccone}(c),
%the intermediate value theorem yields %$y_i (t_i)$,
{\blue there exists}
$0 <t_i \le 1$ such that
{\blue  the function
\begin{align*}
m_{y_i (t_i)} (\cdot) : = -\tc(\cdot , y_i (t_i) ) + \tc(\tilde q, y_i (t_i)) + h^\tc (\tilde q)
\end{align*}
satisfies
\begin{align}\label{eq:m-y-i-t-i}
m_{y_i (t_i)} \le 0 \qquad \text{ on $\Q$ with equality at $q^i$,}
\end{align}
see Figure \ref{figc-cone}.
By the definition of $h^\tc$,  \eqref{eq:m-y-i-t-i} implies
$y_i (t_i) \in \p^\tc h^\tc (\tilde  q) \cap \p^\tc h^\tc (q^i)$,
$$
-D_q \tc (\tilde q, y_i (t_i)) \in \p h^\tc (\tilde q)
\quad \text{and} \quad t_i s_i p_i =-D_q \tc (q^i, y_i (t_i)) \in \p h^\tc (q^i).
$$
%\footnote{YH: In figure \ref{figc-cone} I have changed the notation for $m_i$ and $m_i (t_i)$}
\begin{figure}
\begin{center}
\centerline{\epsfysize=1.7truein\epsfbox{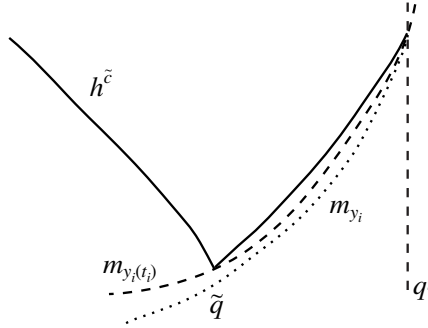}}\caption{{\small The dotted line represents the graph of $m_{y_i}:=-\tc(\cdot , y_i ) + \tc(\tilde q, y_i) + h^\tc (\tilde q)$,
while the dashed one represents the graph of $m_{y_i(t_i)}:=-\tc(\cdot
, y_i(t_i) ) + \tc(\tilde q, y_i(t_i)) + h^\tc (\tilde q)$. The
idea is that, whenever we have $m_{y_i}$ a supporting function for
$h^\tc$ at a point $q^i \in \p \Q$, we can let $y$ vary
continuously along the $\tc$-segment from $\ty$ to $y_i$ with
respect to $q^i$, to obtain a supporting function $m_{y_i(t_i)}$ which
touches $h^\tc$ also {\blue at $\tq$ as well as $q^i$}.}}\label{figc-cone}
\end{center}
\end{figure}
Note that the sub-level set $S_{y_i(t_i)} :=\{z \in U_\ty \ | \  m_{y_i (t_i)} \le 0\}$ is convex.
Draw from $\tq$ a half-line orthogonal to $\Pi^i$. Let $\tq^i$ be point where this line meets with the boundary $\p S_{y_i(t_i)}$. By the assumption
$\Q \subset B \subset 4 \rdot B \subset U_\ty$, we see $\tq^i \in U_\ty$. By convexity of $S_{y_i(t_i)}$, we have
\begin{align}\label{eq:tqtqi}
|\tq - \tq^i| \le \dist(\tq, \Pi^i) .
\end{align}
Let ${\rm diam} \, \Q \le \d_n\e_c$
for some small constant {\blue $0< \d_n <1/3$} to be fixed.
By \eqref{slope compare} and the trivial inequality $\dist(\tq, \Pi^i) \le \diam \Q$, we have{\red
\begin{equation}
\label{eq:close gradients}
\begin{split}
\big|- D_q \tc(\tq , y_i (t_i))+D_q \tc(\ell \tq+(1-\ell)\tq^i, y_i (t_i) ) \big|
& \le \frac{1}{\e_c} |\tq - q^i | \, \big|-D_q \tc(\tq , y_i (t_i)\big|\\
&\le  \delta_n\big|-D_q \tc( \tq, y_i (t_i))\big|
\end{split}
\end{equation}
for all $0 \leq \ell \leq 1$.
%Now consider the affine function $P^i$ with slope $-D_q \tc (q^i, y_i(t_i))$ and with $P^i (\Pi^i)\equiv 0$.
Therefore, we see
\begin{align*}
|h^\tc  (\tilde q)|& = |\tc (\tq,y_i (t_i)) - \tc (\tq^i,y_i (t_i))|
\\
& =\Big| \int_0^1 - D_q \tc\big(\ell \tq+(1-\ell)\tq^i, y_i (t_i)\big) \cdot (\tq - \tq^i ) d\ell \Big|
\\
& \le (1+\delta_n)  \big| -D_q \tc(\tq , y_i (t_i)) \big| \big|  \tq - \tq^i \big|  \qquad \hbox{ (by  \eqref{eq:close gradients})} \\
& \le (1+\delta_n) \big| -D_q \tc(\tq , y_i (t_i)) \big| \dist (\tq, \Pi^i)   \qquad \hbox{ by \eqref{eq:tqtqi}).  }
\end{align*}
Thanks to {\green this estimate it} follows }
\begin{align}\label{D c at tq lower}
| -D_q \tc (\tq, y_i (t_i))|
%\ge \frac{|\tc (\tilde q,y_i (t_i)) - \tc (q^i,y_i (t_i))|}{2\dist (\tilde q, q^i)}
\ge \frac{|h^\tc  (\tilde q)|}{2\dist (\tilde q, \Pi^i)}.
\end{align}
%\begin{align}\label{D c at z i lower}
%| -D_q \tc (q^i, y_i (t_i))| \ge \frac{|\tc (\tilde q,y_i (t_i)) - \tc (q^i,y_i (t_i))|}{2\dist (\tilde q, \Pi^i)}=\frac{|h^\tc  (\tilde q)|}{2\dist (\tilde q, \Pi^i)}.
%\end{align}
{\green Moreover, similarly to \eqref{eq:close gradients}, we have
\begin{align} \label{eq:Dcqitq}
&\big|   - D_q \tc(q^i , y_i (t_i))+D_q \tc(\tq, y_i (t_i) )    \big|
\le \d_n \big| -D_q \tc(\tq, y_i (t_i) )\big| .
\end{align}
Since} the vectors $\{-D_q \tc(q^i, y_i (t_i))\}_{i=1}^n$ are mutually orthogonal, \eqref{eq:Dcqitq} implies that
for $\d_n$ small enough the convex hull of $\{-D_q \tc(\tilde q, y_i (t_i))\}_{i=1}^n \subset \p h^\tc (\tilde q)$ has measure of order
$$\prod_{i=1}^n\big|-D_q \tc(\tq , y_i (t_i))\big|.$$
Thus, by the lower bound
\eqref{D c at tq lower} and the convexity of $\p h^\tc (\tilde q)$, we obtain}
\begin{align*}
\Leb{n}(\p h^\tc (\tilde q)) & \ge C(n) \frac{|h^\tc (\tilde q)|^n}{\prod_{i=1}^n \dist(\tilde q, \Pi^i)}.
\end{align*}
%Previous C(n, \|c\|_{C^3}) is now C(n)
Since $\Pi^1$ was either $\Pi^+$ or $\Pi^-$, we have proved that
$$
|h^\tc(\tilde q)|^n \leq C(n) {\RJM |\p h^\tc|(\{\tilde q\})}
\min\{\dist(\tilde q,\Pi^+),\dist(\tilde q,\Pi^-)\} \prod_{i=2}^n
\dist(\tilde q,\Pi^i).
$$
%$$
%|\tu(\tilde q)|^n \leq C(n,\|c\|_{C^3},|\det (D^2_{qy}c)|) \min\{\dist(\tilde
%q,\Pi^+),\dist(\tilde q,\Pi^-)\} \prod_{i=2}^n \dist(\tilde
%q,\Pi^i)|\p h^\tc|(\tilde q).
%$$
To conclude the proof, we apply Lemma~\ref{lemma:orthogonal sections} below
with $\Q'$ given by the segment obtained intersecting
$\Q$ with a line orthogonal to $\Pi^+$. % and passing through $\tilde q$.
Combining
that lemma with \eqref{eq:choice hyperpl}, we obtain
$$
C(n) {\RJM \Leb{n}(\Q)} \geq \ell_{\Pi^+} \prod_{i=2}^n \dist(\tilde
q,\Pi^i),
$$
and last two inequalities prove the proposition (taking $C(n) \ge 1/\delta_n$ larger if necessary).
\end{proof}

\begin{lemma}[Estimating a convex volume using one slice and an orthogonal projection]
\label{lemma:orthogonal sections}
Let $\Q$ be a convex set in $\R^n = \R^{n'} \times \R^{n''}$.  Let
$\pi', \pi''$ denote the projections to the components $\R^{n'}$,
$\R^{n''} $, respectively.
%\footnote{RJM: section = {\em cross-section}, or {\em slice}?}
Let $\Q'$ be a slice orthogonal to the second component, that is
$$
\Q' = (\pi'')^{-1}(\bar x'')\cap  \Q \qquad\text{for some }\bar x'' \in \pi''(\Q).
$$
Then there exists a constant $C(n)$, depending only on $n=n'+n''$,
such that
$$
C(n) \Leb{n}(\Q) \ge \Haus{n'}(\Q') \Haus{n''}(\pi'' (\Q)),
$$
where $\Haus{d}$ denotes the $d$-dimensional Hausdorff measure.
\end{lemma}

\begin{figure}
\begin{center}
\centerline{\epsfysize=1.7truein\epsfbox{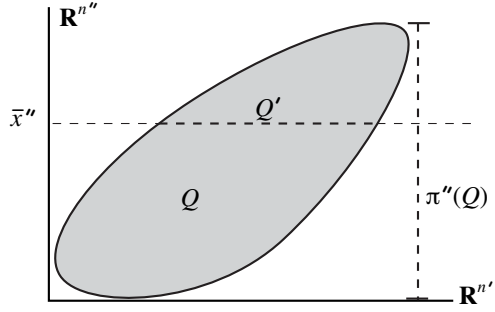}}\caption{{\small The volume of any convex set always controls the
product (measure of one slice) $\cdot$ (measure of the projection
orthogonal to the slice).}}\label{figconvex}
\end{center}
\end{figure}

\begin{proof}
Let $L:\R^{n''} \to \R^{n''}$ be an affine map with determinant $1$ given by
Lemma~\ref{L:John} such that $B_r \subset L(\pi''(\Q)) \subset B_{n''r}$ for some $r>0$.
Then, if we extend $L$ to the whole $\R^n$ as $\tilde L(x',x'')=(x',Lx'')$, we have
$
\Leb{n}(L(\Q))=\Leb{n}(\Q)$, $\Haus{n'}(\tilde L(\Q'))=\Haus{n'}(\Q')$, and
$$
\Haus{n''}(\pi'' (\tilde L(\Q)))=\Haus{n''}(L(\pi'' (\Q)))=\Haus{n''}(\pi'' (\Q)).
$$
Hence, we can assume from the beginning that $B_r \subset \pi''(\Q)
\subset B_{n''r}$. Let us now consider the point $\bar x''$, and
we fix an orthonormal basis $\{\hat e_1,\ldots, \hat e_{n''}\}$ in
$\R^{n''}$ such that $\bar x''=c \hat e_1$ for some $c\leq 0$.
Since $\{r \hat e_1,\ldots,r \hat e_{n''}\} \subset \pi''(\Q)$,
there exist points $\{x_1,\ldots,x_{n''}\} \subset \Q$ such that
$\pi''(x_i)=r \hat e_i$. Let $C'$ denote the convex hull of $\Q'$
with $x_1$, and let $V'$ denote the $(n'+1)$-dimensional strip
obtained taking the convex hull of $\R^{n'}\times\{\bar x''\}$
with $x_1$. Observe that $C'\subset V'$, so
\begin{equation}
\label{eq:C'}
\Haus{n'+1}(C')=\frac{1}{n'+1}\dist(x_1,\R^{n'}\times\{\bar x''\})\Haus{n'}(\Q') \geq \frac{r}{n'+1}\Haus{n'}(\Q').
\end{equation}
We now remark that, since $\pi''(x_i)=r \hat e_i$ and $\hat e_i
\perp V'$ for $i=2,\ldots,n''$, we have $\dist(x_i,V')=r$ for all
$i=2,\ldots,n''$. Moreover, if $y_i\in V'$ denotes the closest
point to $x_i$, then the segments joining $x_i$ to $y_i$ parallels
$\hat e_i$, hence these segments are all mutually orthogonal, and
they are all orthogonal to $V'$ too. From this fact it is easy to
see that, if we define the convex hull
$$
C:=\co(x_2,\ldots,x_{n''},C'),
$$
then, since $|x_i-y_i|=r$ for $i=2,\ldots,n''$, by \eqref{eq:C'} and the inclusion $\pi''(\Q) \subset B_{n''r}\subset \R^{n''}$ we get
$$
\Leb{n}(C)= \frac{(n'+1)!}{n!}\Haus{n'+1}(C') r^{n''-1} \geq
\frac{n'!}{n!}\Haus{n'}(\Q') r^{n''} \geq
C(n)\Haus{n'}(\Q')\Haus{n''}(\pi'' (\Q)).
$$
This concludes the proof, as $C \subset \Q$.
\end{proof}

\subsection{{\blue Alexandrov type upper bounds}}\label{S:alex}

The next Alexandrov type lemma
holds for localized sections $\Q$ of $\tc$-convex functions.

%\footnote{localized}
%\footnote{the constant is corrected!}

\begin{lemma}[Alexandrov type estimate and lower barrier]
\label{L:prop-infleq}
Assume \Bzero--\Bthree, and
let $\tu:\cl \U_\ty \longmapsto \R$ be a $\tilde c$-convex function
from Theorem \ref{thm:apparently convex}.
%, with $\tu \ge 0$ on $\p \U_\tq$.
%and let $\tc:\cl \U_\tq \times \V \longmapsto \R$ denote the modified cost
%from Definition \ref{D:cost exponential}.
Let $\Q$ denote the section $\{ \tilde u \le 0\}\subset \cl \U_\ty$, assume $\tu =0$ on $\partial \Q$,
and fix $\tilde q \in \intr \Q$.
%, and assume $\tu = 0$ on $\Q$.
Let $\Pi^+,\Pi^-$ be two parallel hyperplanes contained in $\R^n
\setminus \Q$ and touching $\p \Q$ from two opposite sides.
{\blue
We also assume that {\red there exists a ball $B$ such that}
$\Q \subset B \subset 4 \rdot B \subset U_{\bar y}$.}
Then
there exist {\RJM $\e_c'>0$ (given by Proposition
\ref{P:partial c-cone})} such that, if $\diam(\Q) \leq {\RJM \e_c'}$ then
$$
|\tu(\tilde q)|^n \leq C(n) \gamma^+_\tc(\Q \times \V)
\frac{\min\{\dist(\tilde q,\Pi^+),\dist(\tilde
q,\Pi^-)\}}{\ell_{\Pi^+}}|\p^\tc\tu|(\Q) \Leb{n}(\Q),
$$
where $\ell_{\Pi^+}$ denotes the maximal length among all
the segments obtained by intersecting $\Q$ with a line orthogonal
to $\Pi^+$,
%$\ell_{\tilde q}$ denotes the length of the segment obtained intersecting
%$\Q$ with the line orthogonal to $\Pi^+$ and passing through $\tilde q$
and {\blue $\gamma^+_\tc(\Q \times \V)$} is defined as in (\ref{Jacobian bound}).
\end{lemma}

\begin{proof}
 Fix $\tilde q \in \Q$. Observe that $\tu = 0$ on $\p \Q$ and consider the $\tc$-cone $h^\tc$
generated by $\tilde q$ and $\Q$ of height $-h^\tc(\tq)=-\tu(\tq)$ as in \eqref{c-cone}.
From Lemma~\ref{lemma:propccone}(d) we have
$$
% \geq |\p^c h^c|(\Q)
|\p^\tc h^\tc|(\{\tilde q\}) \le |\p^\tc\tu|(\Q),
$$
%In fact, one can verify that the last inequality is indeed an equality. Indeed, ...
\noindent
and from Loeper's local to global principle, Corollary \ref{C:local-global} above,
% to Loeper's maximum principle,
$$
\p h^\tc(\tq)=-D_q \tc(\tq,\p^\tc h^\tc(\tq)).
$$
%Since $\det (D^2_{qy}c(\tilde q,y)) \neq 0$, we have
Therefore,
$$
|\p h^\tc|(\{\tq\}) \le \| \det D^2_{qy} \tc\|_{C^0 (\{\tq\} \times \V)} |\p^ch^c|(\{\tq\}).
$$
%with $C$ depending only on $\frac{1}{|\det (D^2_{qy}c )|}$.
The lower bound on $|\p h^\tc|(\{\tilde q\})$ comes from \eqref{p c-cone lower}. This finishes the proof.
\end{proof}

%\subsection{Estimating solutions to the $\tc$-Monge-Amp\`ere inequality %es under the assumption
%$|\p^\tc \tu| \in [ \lambda, 1/\lambda]$}
%\footnote{YH: I have combined these two subsections RM: I uncombined them, but eliminated all
%sub-subsections AF: fine for me as it is now.}

Combining {\blue this with Theorem~\ref{T:ratio}, we get the following important estimates:

%\footnote{\RJM RM: $\e_c'(n) \to \e_c'$}
\begin{theorem}[Alexandrov type upper bound]%[Bounding local Dirichlet solutions to $\tc$-Monge-Amp\`ere inequalities]
\label{thm:estimate} Assume \Bzero--\Bthree, and let $\tu:\cl
\U_\ty \longmapsto \R$ be a $\tilde c$-convex function from
Theorem \ref{thm:apparently convex}. There exist $\e_c' >0$
small, depending only on the dimension and the cost function, and constant $C(n)$,
%$C_i(n)>0$, $i=1,2$,
depending only on the dimension, such that the
following holds:

%, with $\tu \ge 0$ on $\p \U_\tq$.
%and let $\tc:\cl \U_\tq \times \V \longmapsto \R$ denote the modified cost
%from Definition \ref{D:cost exponential}.
Letting $\Q$ denote the section $\{ \tilde u \le 0\}\subset \U_\ty$, assume
$|\p^\tc \tu| \leq 1/\l$ in $\Q$ and $\tu = 0$ on $\p \Q$.
We also assume that {\red there exists a ball $B$ such that} $\Q \subset B \subset 4 \rdot B \subset \U_{\bar y}$, and that
$\diam(\Q) \leq \e_c'$.
 For $\frac{1}{2n} < t\le 1$, let $q_t \in \Q$ be a point such that
$q_t \in t \rdot \p \Q$, where $t \rdot \partial \Q$ denotes the dilation with the factor $t$
with respect to the center of $E$, the ellipsoid given by John's Lemma (see \eqref{E:well-centered}). Then,
\begin{equation}\label{vardist}
|\tu (q_t)|^n \leq C(n) \frac{\gamma_\tc^+}{\lambda} (1-t)^{1/2^{n-1}} \Leb{n}(\Q)^2,
\end{equation}
where $\gamma^+_\tc = \gamma^+_\tc(\Q \times \V)  $
is defined as in \eqref{Jacobian bound}, which satisfies  $\gamma^+_\tc \le \gamma^+_c \gamma^-_c$ from Corollary
\ref{C:Jacobian transform}.
Moreover,
\begin{equation}\label{inf|\Q|}
\frac{|\inf_{\Q} \tu|^n}{ \Leb{n}(\Q)^2}
\leq \ C(n)
\frac{\gamma_\tc^+}{\lambda}.
\end{equation}
 \end{theorem}

\begin{remark}{\rm
 Aficionados of the Monge-Amp\`ere theory may be less surprised by
  these estimates once it is recognized that the localization in
coordinates ensures the cost is approximately affine, at least in
one of its two variables. However, no matter how well we approximate,
the non-affine nature of the cost function {\RJM remains relevant and persistent.
Controlling the departure from affine
 is vital to our analysis and requires the new ideas
developed above.}}
\end{remark}
}

\begin{proof}[\bf Proof of Theorem~\ref{thm:estimate}]
For $0<s_0 \le t\le 1$,
let $q_t \in \Q$ be a point such that
$q_t \in t \partial \Q$.
By Theorem \ref{T:ratio} applied with $s_0=1/(2n)$ we can find
$\Pi^+ \ne \Pi^-$ parallel hyperplanes contained in $T_\ty^*\V
\setminus \Q$, supporting $\Q$ from two opposite sides, and such that
$$
\frac{\min\{\dist(\tilde q,\Pi^+),\dist(\tilde
q,\Pi^-)\}}{\ell_{\Pi^+}} \leq C(n) (1-t)^{1/2^{n-1}}.
$$
%where $\ell_{\Pi^+}$ denotes the maximal length among all
%the segments obtained by intersecting $\Q$ with a line orthogonal
%to $\Pi^+$.
Then \eqref{vardist} follows from  Lemma~\ref{L:prop-infleq} and the assumption
$|\p^\tc \tu| \le 1/\l$.

To prove \eqref{inf|\Q|}, observe that, since ${\RJM \frac{1}{2n} \rdot \Q \subset \frac{1}{2}} \rdot E$,
$$
\frac{|\inf_{\Q\setminus ({\RJM \frac{1}{2}} \rdot E)} \tu|^n}{ \Leb{n}(\Q)^2}
\leq \ C(n)
\frac{\gamma_\tc^+}{\lambda}(1-\frac{1}{2n})^{1/2^{n-1}}.
$$
On the other hand, taking  $\Pi^+$ and $\Pi^-$ orthogonal to one of the longest axes of $E$ and choosing
$q \in {\RJM \frac{1}{2} }\rdot E$ in Lemma \ref{L:prop-infleq} yields
\begin{align*}
|\tu(q)|^n & \leq C(n) \frac{\gamma_\tc^+}{\lambda} n \Leb{n}(\Q)^2
, \qquad \forall q \in {\RJM\textstyle \frac{1}{2}} \rdot E.
\end{align*}
{\blue  Combining these two estimates we obtain \eqref{inf|\Q|}
to complete the proof.}
\end{proof}

\section{The contact set is either a single point or crosses the domain}
%and absence of its exposed points}
\label{S:contact}
{\blue The previous estimates (Theorems~\ref{thm:lower Alex} and \ref{thm:estimate})
{\RJM may have some} independent interest, but they also provide key ingredients
{\RJM which we use to deduce injectivity and H\"older continuity of optimal maps.}
In this and the subsequent section,} we
prove the strict $c$-convexity of the $c$-convex optimal transport potentials
$u:\cl \U \longmapsto \R$, meaning
$\p^c u(x)$ should be disjoint from $\p^c u(\tx)$ whenever $x,\tx
\in \U^\l$ are distinct. {\RJM This shows the injectivity of optimal maps, and}
is accomplished in Theorem~\ref{T:strict convex}.
In the present section we show that, if the contact set %$\p^\cs u^\cs(\ty)$
does not consist of a single point, then it  extends to the
boundary of $\U$.
%or rather $\p^{\tc^*} \tu^{\tc^*}(\ty)$ has no exposed points in the interior of $\U_\ty$.
Our method relies on the \Bthree\ assumption on the cost $c$.

%From now on we adopt the following notation: $a \sim b$ means that
%there exist two positive constants $C_1$ and $C_2$, depending  on
%$n$ and $\gamma^+_c \gamma^-_c/\l$ only, such that $C_1a\leq b\leq
%C_2a$. Analogously we will say that $a \lesssim b$ (resp. $a
%\gtrsim b$) if there exists a positive constant $C$, depending  on
%$n$ and $\gamma^+_c \gamma^-_c/\l$ only, such that $a\leq C b$
%(resp. $Ca \geq b$).

Recall that a point $x$ of a convex set $S \subset \R^n$ is {\em
exposed} if there is a hyperplane supporting $S$ exclusively at
$x$. Although the {\em contact set} $S := \p^\cs u^\cs(\ty)$ may
not be convex, it appears convex from $\ty$
 by Corollary \ref{C:local-global}, meaning
its image $q(S) \subset \U_\ty$ in the coordinates \eqref{q of x}
is convex.
% to Loeper's maximum principle.
%Abusing terminology, we use the phrase {\em $S$ has no exposed points} to mean the convex
%set $q(S)$ has no exposed points.  Since any compact convex set can be generated as
%the closed convex hull of its exposed points \cite[Theorem 18.7]{Rockafellar},
%their absence implies this convex set must extend across its domain $\cl \U_\ty$.
%This result is the following theorem.  We prove it by showing the
%solution geometry near any exposed point of $S$ inside $\U$ would be inconsistent with
%bounds established in the previous section.
The following theorem shows this convex set is either a singleton,
or contains a segment which stretches across the domain. We prove
it by showing the solution geometry near certain exposed points of
$q(S)$ inside $\U_\ty$ would be inconsistent with the bounds {\blue (Theorem~\ref{thm:lower Alex} and \ref{thm:estimate})}
established in the previous section.

\begin{theorem}[The contact set is either a single point or crosses the domain]
\label{thm:noexposed}
Assume \Bzero--\Bthree, and
let $\uu$ be a $c$-convex solution of \eqref{eq:cMA} with $\U^\l \subset \U$ open.
Fix $\tx \in \U^\l$ and %$X := \supp |\p^c u|$,
$\ty \in \p^c\uu(\tx)$, and define the contact set
$S:=\{x \in \cl \U \mid \uu(x)=\uu(\tx)-c(x,\ty)+c(\tx,\ty)\}$. % = \p^\csu^\cs(\by)$.
Assume that $S \neq \{\tx \}$, i.e. it is not a singleton. Then
$S$ intersects $\p\U$.
\end{theorem}

\begin{proof}
{\blue  To derive a contradiction, we assume that $S \neq \{\tx \}$ and $S \subset \subset U$ (i.e. $S\cap \p U = \emptyset$).}

As in Definition \ref{D:cost exponential}, we transform $(x, \uu)
\longmapsto (q, \tu)$ with respect to $\ty$, i.e. we consider the
transformation $q \in \cl \U_\ty \longmapsto x(q) \in \cl\U$,
defined on $\cl\U_\ty := -D_yc(\cl\U,\ty) \subset T^*_\ty \V$ by
the relation
$$
-D_y c(x(q),\ty)=q,
$$
and the modified cost function
$
\tilde c(q,y):=c(x(q),y)-c(x(q),\ty)
$
on $\cl \U_\ty \times \cl \V$, for which
the $\tc$-convex potential function
$
q \in \cl \U_\ty \longmapsto \tu(q):=\uu(x(q)) + c(x(q),\ty)
$
is level-set convex.
We observe that $\tc(q, \ty) \equiv 0$ for all $q$, and moreover the set $S = \p^\cs u^\cs(\ty)$
appears convex from $\ty$, meaning $S_\ty := -D_y c(S, \ty)$ is convex,
by Corollary \ref{C:local-global}.

Our proof is reminiscent of Caffarelli's for the cost $\tc(q,y) =
-\<q, y\>$ \cite[Lemma 3]{Caffarelli92}. Observe $\tq := - D_y
c(\tx,\ty)$ lies in the interior of the set $\U^\l_\ty := -D_y
c(\U^\l,\ty)$ where $|\p^\tc \tu| \in
[\l/\gamma^+_c,\gamma^-_c/\l]$, according to Corollary
\ref{C:Jacobian transform}.
 Choose the point $\qo \in {\blue S_\ty
\subset \subset  \U_\ty}$ furthest from $\tq$; it is an exposed point of
$S_\ty$. {\blue We will see below that the presence of such exposed point gives a contradiction, proving the theorem.}
%We claim either $\qo = \tq$ or $\qo \in \partial U_\ty$.
%To derive a contradiction,  suppose the preceding claim fails,
%meaning $\qo \in U_\ty \setminus \{\tq\}$.
%

{\blue  Before we proceed further, note that because of our assumption $S\subset \subset U$,
the sets in the following argument, which will be sufficiently close to $S$, are also contained in $U$;
the {\RJM same} holds for the corresponding sets in different coordinates. This is to make sure that
we can perform the analysis using the assumptions on  $c$.}

For a suitable choice of Cartesian coordinates on $\V$ we may,
without loss of generality,  take $\qo-\tq$ parallel to the positive $y^1$
axis. Denote by $\hat e_i$
the associated orthogonal basis for $T_\ty \V$, and set $\bo :=
\langle \qo, \hat e_1 \rangle$ and $\tb := \langle \tq, \hat e_1
\rangle$, so the halfspace $\{q \in T^*_\ty \V \simeq \R^n \mid q_1 := \langle q,\hat e_1 \rangle \ge
\bo\}$ intersects $S_\ty$ only at $\qo$.
%the outward normal to $S_\ty$ at $\qo$ is given by $\hat e_1$.
Use the fact that $q^0$ is an exposed point of $S_\ty$ to cut a
corner $K_0$ off the contact set $S$ by choosing $\bs>0$ small
enough that $\bar b = (1-\bs) b^0 + \bs \tb$ satisfies:
\begin{figure}
\begin{center}
\centerline{\epsfysize=2.0truein\epsfbox{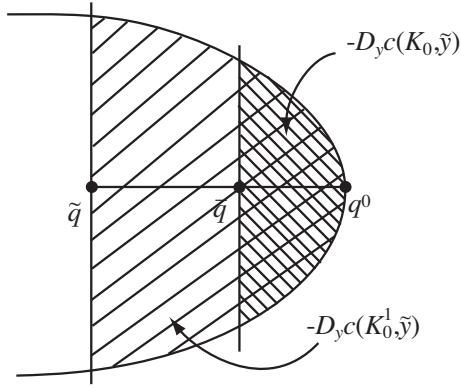}}\caption{{\small If the contact set $S_\ty$ has an exposed point $q_0$,
we can cut two portions of $S_\ty$
with two hyperplanes orthogonal to $\tq - q^0$. The diameter of
$-D_yc(K_0,\ty)$ needs to be sufficiently small to apply the
Alexandrov estimates from Theorem \ref{thm:estimate}, while
$-D_yc(K_0^1,\ty)$ has to intersect $U_\ty^\l$ in some nontrivial set to apply Theorem \ref{thm:lower Alex}
(this is not automatic, as $q^0$ may not be an interior point of $\supp |\p^\tc\uu|$).}}\label{figcontactset}
\end{center}
\end{figure}
\begin{enumerate}
\item[(i)] $-D_y c(K_0,\ty):=S_{\ty}
\cap \{q \in \cl \U_\ty \mid q_1 \geq \bar b\}$ is a compact convex set in the interior of $\U_\ty$;
\item[(ii)] $\diam(-D_y c(K_0, \ty)) \le \e_c'/2$, where $\e_c'$ is from Theorem
\ref{thm:estimate}.
\item[(iii)] {\blue  $-D_y c(K_0, \ty) \subset B \subset 5 \rdot B \subset \U_\ty$ for some ball $B$ as in the assumptions of Theorem \ref{thm:estimate}.}
\end{enumerate}
Defining $q^s := (1-s) q^0 + s \tq$, $x^s := x(q^s)$ the
corresponding $c$-segment with respect to $\ty$, and $\bq =
q^\bs$, note that $S_\ty \cap \{q_1 = \bar b\}$ contains $\bq$,
and $K_0$ contains $\bx:= x^\bs$ and $x^0$. Since the corner $K_0$
may not intersect the support of $|\p^c u|$ (especially,
when $q^0$ is not an interior point of $\supp |\p^c u|$),  we
shall need to cut a larger corner $K_0^1$ as well, defined by
$-D_y c(K^1_0,\ty):=S_{\ty} \cap \{q \in \cl \U_\ty \mid q_1 \geq
\tb\}$, which intersects $\U^\l$ at $\tx$.  By tilting the
supporting function slightly, we shall now define sections $K_\e
\subset K_\e^1$ of $u$ whose interiors include the extreme point
$x^0$ and whose boundaries pass through $\bx$ and $\tx$
respectively, but which converge to $K_0$ and $K_0^1$ respectively
as $\e \to 0$.

Indeed, set $y_\e := \ty + \e \hat e_1$ and observe
\begin{align}
m^s_\e (x) & := -c(x, y_\e) + c(x, \ty) + c(x^s, y_\e) - c(x^s, \ty)
\cr&= \e \langle -D_y c(x,\ty)+D_y c(x^s,\ty),\hat e_1 \rangle+o(\e)
\cr&= \e (\langle-D_y c(x,\ty), \hat e_1\rangle -(1-s)\bo - s\tb)+o(\e).
\label{e asymptotic}
\end{align}
Taking $s \in \{\bar s,1\}$ in this formula and $\e>0$ shows the sections defined by
\begin{align*}
K_\e &:=\{x \mid \uu(x) \leq \uu(\bx)-c(x,y_\e) +c(\bx,y_\e)\},\\
K^1_\e &:= \{x \mid \uu(x) \leq \uu(\tx)-c(x,y_\e) +c(\tx,y_\e)\},
\end{align*}
both include a neighbourhood of $\xo$ but
converge to $K_0$ and $K_0^1$ respectively as $\e \to 0$.
\begin{figure*}
\begin{center}
\centerline{\epsfysize=1.5truein\epsfbox{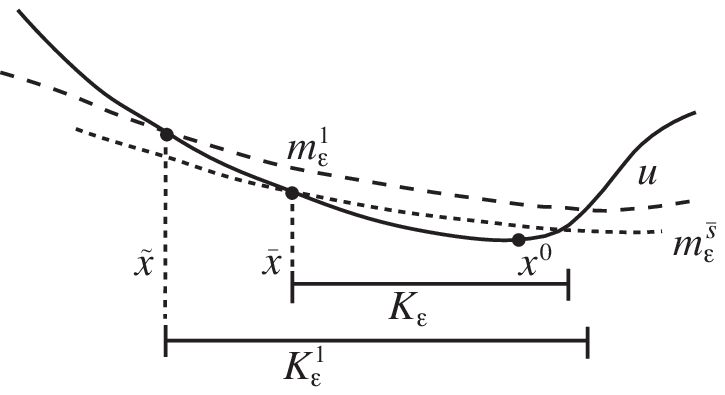}}\caption{{\small  We cut the graph of $u$ with the two functions $m_\e^{\bar s}$ and $m_\e^1$
to obtain two sets $K_\e \approx K_0$ and $K_\e^1 \approx K_0^1$
inside which we can apply our Alexandrov estimates to get a
contradiction (Theorem \ref{thm:estimate} to $K_\e$, and Theorem \ref{thm:lower Alex} to $K_\e^1$). The idea is that the value of $u - m_\e^{\bar s}$
at $\xo$ is comparable to its minimum inside $K_\e$, but this is
forbidden by our Alexandrov estimates since $x_0$ is too close to
the boundary of $K_0^\e$. However, to make the argument work we
need also to take advantage of the section $K_\e^1$, in order to
``capture'' some positive mass of the $c$-Monge-Amp\`ere
measure.}}\label{figcutsection}
\end{center}
\end{figure*}

We remark that there exist a
priori no coordinates in which {\blue all sets $K_\e$ are simultaneously} convex. However
for each fixed $\e>0$,
we can change coordinates so that both $K_\e$ and $K^1_\e$ become convex:
use $y_\e$ to make the transformations
\begin{align*}
q &:=-D_y c(x_\e(q),y_\e),
\\\tc_\e(q,y) &:=c(x_\e(q),y)-c(x_\e(q),y_\e),
\end{align*}
so that the functions
\begin{align*}
\tu_\e(q)&:=\uu(x_\e(q))+ c(x_\e(q),y_\e)-\uu(\bx)-c(\bx,y_\e), \\
\tu^1_\e(q)&:=\uu(x_\e(q))+ c(x_\e(q),y_\e)-\uu(\tx)-c(\tx,y_\e).
\end{align*}
are level-set convex on $\U_{y_\e} := D_y c(\U,y_\e)$.
Observe that, in these coordinates, $K_\e$ and $K^1_\e$ become convex:
\begin{align*}
\tilde K_\e &:=-D_y c(K_\e,y_\e)=\{q \in \cl\U_{y_\e} \mid \tu_\e(q) \leq 0\}, \\
\tilde K^1_\e &:=-D_y c(K^1_\e,y_\e)=\{q \in \cl\U_{y_\e} \mid \tu^1_\e(q) \leq 0\},
\end{align*}
and either $\tilde K_\e \subset \tilde K^1_\e$ or $\tilde K^1_\e \subset \tilde K_\e$
since $\tu_\e(q) - \tu^1_\e(q) = const$.
For $\e>0$ small, the inclusion must be the first of the two since the limits satisfy
 $\tilde K_0 \subset \tilde K^1_0$ and $\tq \in \tilde K^1_0 \setminus \tilde K_0$.

In the new coordinates,  our original point $\tx \in \U^\l$, the
exposed point $x^0$, and the $c$-convex combination $\bx$ with
respect to $\ty$, correspond to
\begin{align*}
\tq_\e :=-D_y c(\tx,y_\e), \ \ %\in -D_y c(K^1_0,y_\e),
q_\e^0 :=-D_y c(x^0,y_\e), \ \ %\in -D_y c(K_0,y_\e), \\ %\subset \tilde K_\e,
\bq_\e :=-D_y c(\bx,y_\e). %\in -D_y c(K_0,y_\e). %\subset \tilde K_\e.
\end{align*}
Thanks to {\blue (ii) and (iii), for $\e$ sufficiently small we have $\diam(\tilde K_\e) \leq \e_c'$ and $\tilde K_\e \subset B \subset 4 \rdot B \subset U_{y_\e}$ for some ball $B$
(note that this ball
can be different from the one in (iii)),}
so that all the estimates of Theorem \ref{thm:estimate} apply.

Let us observe that, since
$\lim_{\e \to 0} \qo_\e - \bq_\e= \qo -\bq$ and $q^0 \in \partial \Q$, \eqref{vardist} combines with $K_\e \subset K^1_\e$ and
$|\p^{\tc_\e} \tu_\e|(K_\e) \le  \Lambda\gamma^-_c \Leb{n}(K_\e)$
from \eqref{eq:cMA} and Corollary \ref{C:Jacobian transform}, to
yield
\begin{equation}\label{contradictory upper bound}
\frac{|\tu_\e(\qo_\e)|^n} { \Lambda\gamma^-_c \Leb{n}(\tilde K^1_\e)^2} \to 0 \qquad \text{as }\e \to 0.
\end{equation}

On the other hand,  $\bx \in S$ implies  $\tu_\e(\qo_\e) = {\blue m^\bs_\e(x^0)}$,
and $\tx \in S$ implies $\tu^1_\e(\qo_\e) = {\blue m^1_\e(x^0)}$ similarly.
Thus \eqref{e asymptotic} yields
\begin{align}\label{comparison at q0}
\frac{\tu_\e (\qo_\e)}{\tu^1_\e (\qo_\e)}
= \frac{\e(b^0-\bar b) + o(\e)  }{\e (b^0 - \tb) + o(\e)}
\to \bs \quad {\rm as} \quad \e \to 0.
\end{align}
Our contradiction with \eqref{contradictory upper bound}--\eqref{comparison at q0}
will be established by bounding the
ratio $|\tu^1(\qo_\e)|^n/\Leb{n}(K^1_\e)^2$
away from zero.

Recall that
$$
b^0=\langle -D_y c(x^0,\ty), \hat e_1 \rangle \ =\ \max \{q_1 \mid q \in -D_y c(K_0,\ty) \}) \ >\  \tb
$$
and $\uu(x)-\uu(\tx)\geq -c(x,\ty)+c(\tx,\ty)$ with equality at $x^0$.
From the convergence of $K^1_\e$ to $K^1_0$ and the asymptotic behaviour \eqref{e asymptotic}
of $m^1_\e(x)$ we get
\begin{equation}\label{ratio var and inf var}
\begin{split}
%\frac{\tu_\e^*(\qo^{\e*})}{\inf_{\tilde K_\e^*}\tu_\e^*}=
\frac{\tu^1_\e(\qo_\e)}{\inf_{\tilde K^1_\e} \tu^1_\e}
&=\frac{-\uu(x^0) - c(x^0, y_\e)+\uu(\tx)+c(\tx,y_\e)}
{\sup_{q \in \tilde K^1_\e} [-\uu(x(q))- c(x(q),y_\e)+\uu(\tx)+c(\tx,y_\e)]}\\
%&=\frac{c(x^0, \ty) - c(x^0, y_\e)-c(\tx, \ty)+c(\tx,y_\e)}
%{\sup_{\tilde K_\e} [-\uu(x(q))- c(x(q),y_\e)+\uu(\tx)+c(\tx,y_\e)]}\\
& \ge \frac{- c(x^0, y_\e) + c(\tx, y_\e) + c(x^0, \ty) - c(\tx, \ty)}
{\sup_{x \in K^1_\e}[-c(x, y_\e) + c(\tx, y_\e) + c(x, \ty)- c(\tx, \ty)]} \\
&\geq \frac{\e (\langle-D_y c(x^0,\ty), e_1\rangle - \tb)+o(\e)}
 {\e (\max \{q_1 \mid q \in -D_y c(K^1_\e,\ty)\} - \tb)+o(\e) }\\
&\geq \frac{1}{2}
\end{split}
\end{equation}
for $\e$ sufficiently small ({\blue because, by our construction, $\max \{q_1 \mid q \in -D_y c(K^1_\e,\ty)\}$ is
exactly $ \langle-D_y c(x^0,\ty), e_1\rangle$}).
This shows $\tu^1(q^0_\e)$ is comparable to the minimum value of $\tu^1_\e$. % on $\tilde K^1_\e$.
To conclude the proof we would like to apply Theorem \ref{thm:lower Alex}, but we need to show that
$|\p^c u| \geq \l$ on a stable fraction of $K^1_\e$ as
$\e \to 0$. We shall prove this as in \cite[Lemma 3]{Caffarelli92}.

Since $K^1_\e$ converges to $K^1_0$ for sufficiently small $\e$,
observe that $K^1_\e$ {\blue  (thus, $\tilde K^1_\e$)} is {\blue bounded  uniformly in $\e$}.  Therefore the affine
transformation $(L^1_\e)^{-1}$ that sends $\tilde K^1_\e$ to $B_1
\subset \tilde K^{1, *}_\e \subset \cl B_n$ as in Lemma
\ref{L:John} is an expansion, i.e. $|(L^1_\e)^{-1} q -
(L^1_\e)^{-1}q'| \geq C_0|q-q'|$, with a constant $C_0>0$
independent of $\e$. Since $\tx$ is an interior point of $\U^\l$,
$B_{2\beta^-_c\d/C_0}(\tx) \subset \U^\l$ for sufficiently small
$\delta >0$ (here $\beta^-_c$ is from \eqref{bi-Lipschitzbound}), hence $B_{2\d/C_0}(\tq_\e) \subset \U^\l_{y_\e}$,
{\blue where
$\U^\l_{y_\e}:=-D_yc(\U^\l,y_\e)$.  Let $\tq_\e^*:= (L^1_\e)^{-1} (\tq_\e)$. Then, by the expansion property of $(L^1_\e)^{-1}$, we have
$$
\U^{\l,*}_{y_\e}:=(L_\e^1)^{-1}(\U^\l_{y_\e}) \supset B_{2\d}(\tq_\e^*).
$$
Since $\tilde K^{1,*}_{\e}$ is convex, it contains
the convex hull $\mathcal{C}$ of $B_{1} \cup \{\tq^*_\e \}$. Consider $\mathcal{C} \cap B_{2\d} (\tq_\e^*)$. Since $\dist (\zero, \tq^*_\e) \le n$, there exists  a ball $B^*$ of radius $\d/n$ (centered somewhere in $\mathcal{C} \cap B_{2\d}(\tq_\e^*) $) such that
\begin{align*}
B^* \subset \mathcal{C} \cap B_{2\d}(\tq_\e^*) \subset B_{2\d}(\tq_\e^*) \cap \tilde K^{1,*}_{\e}.
\end{align*}
Therefore, the ellipsoid
\begin{align*}
 E_\delta:=L_\e^1\bigl(B^* \bigr)
\end{align*}
is contained in $\U^\l_{y_\e}$.
Notice that $E_\d$ is nothing but a dilation and translation of the ellipsoid $E= L_\e^1 (B_1)$ associated to $\tilde K^1_\e$ by John's Lemma in \eqref{E:well-centered}.}
%{\red Since
%$E_\d \subset \tilde K_\e \subset B \subset 4 \rdot B \subset U_{y_\e}$,}
{\green   By dilating further if necessary (but, with a factor independent of $\e$), one may assume that
$E_\d  \subset B \subset 4 \rdot B \subset U_{y_\e}$ for some ball $B$ (as before,
this ball can be different from that for $K_0$ or $\tilde K_\e$). Thus,}
we can apply {\blue Theorem \ref{thm:lower Alex} (with $\Q= \tilde K^1_\e$)} and obtain
\begin{align*}
\frac{|\inf_{\tilde K^1_\e} \tu^1_\e|^n}{\Leb{n}(\tilde K^1_\e)^2} \gtrsim \delta^{2n}
\end{align*}
{\blue where  the inequality $\gtrsim$ is independent of $\e$. As $ \e\to 0$ }
%Since $\delta>0$ is independent of $\e$
this contradicts
\eqref{contradictory upper bound}--\eqref{ratio var and inf var} to complete the proof.
\end{proof}

\begin{remark}
{\rm
As can be easily seen from the proof, {\RJM when $U^\l=U$} one can actually show that if
$S$ is not a singleton, then $S_\ty$ has no exposed points in the interior of $U_\ty$.
Indeed, if by contradiction there exists $\qo$ an exposed point of $S_\ty$
belonging to the interior of $U_\ty$, we can choose a point $\tq\in S_\ty$ {\blue sufficiently close to $q^0$} in the interior of $U_\ty=U_\ty^\l$
such that the segment $\qo-\tq$ is orthogonal to a hyperplane supporting $S_\ty$ at $\qo$.
Then it can be  immediately checked that the above proof
(which could even be simplified in this particular case, since one may choose $K_0^1=K_0$
and avoid the last part of the proof) shows that such a point $\qo$ cannot exist.
}
\end{remark}

\section{Continuity and injectivity of optimal maps}
\label{S:ContinuousInjection}

The first theorem below combines results of Sections
\ref{S:mapping} and \ref{S:contact} to deduce strict $c$-convexity
of the $c$-potential for an optimal map, if its target is strongly
$c$-convex.  This strict $c$-convexity --- which is equivalent to
injectivity of the map
--- will then be combined with an adaptation of Caffarelli's argument \cite[Corollary 1]{Caffarelli90}
to obtain interior continuity of the map --- or equivalently
$C^{1}$-regularity of its $c$-potential function --- for
\Bthree\ costs.

\begin{theorem}[Injectivity of optimal maps to a strongly $c$-convex target]
\label{T:strict convex}
Let $c$ satisfy \Bzero--\Bthree\ and \Btwos.
If $\uu$ is a $c$-convex solution of \eqref{eq:cMA} on $\U^\l \subset \U$ open,
then $\uu$ is strictly $c$-convex on $\U^\l$, meaning
$\p^c \uu (x)$ and $\p^c \uu(\tx)$ are disjoint whenever $x,\tx \in \U^\l$ are distinct.
\end{theorem}

\begin{proof}
Suppose by contradiction that $\ty \in \p^c \uu(x) \cap \p^c
\uu(\tx)$ for two distinct points $x,\tx \in \U^\l$, and set
$S=\p^\cs \uu^\cs(\ty)$. According to Theorem \ref{thm:noexposed},
the set $S$ intersects the boundary of $\U$ at a point $\bx \in \p
\U \cap \p^\cs u^\cs(\ty)$. Since \eqref{eq:cMA} asserts $\l \le
|\p^c u|$ on $\U^\l$ and $|\p^c u| \le \Lambda$ on $\cl \U$,
Theorem \ref{thm:bdry-inter}(a) yields $\ty \in \V$ (since $x, \tx
\in U^\l$), and hence $\bar x \in \U$ by Theorem
\ref{thm:bdry-inter}(b). This contradicts $\bar x \in \p \U$ and
proves the theorem.
\end{proof}

By adapting Caffarelli's argument \cite[Corollary 1]{Caffarelli90}, we now show
continuity of the optimal map.
Although in the next section we will actually prove a stronger result (i.e., optimal maps to strongly $c$-convex targets are locally H\"older continuous),
we prefer to prove this result for two reasons: first, the proof is much simpler than the one of H\"older continuity.
Second, although not strictly {\RJM necessary},
knowing in advance that solutions of \eqref{eq:cMA} are $C^1$ will avoid some technical issues in the proof of the $C^{1,\a}$
regularity.

\begin{theorem}[Continuity of optimal maps to strongly $c$-convex targets]
\label{T:continuity} Let $c$ satisfy \Bzero--\Bthree\ and \Btwos.
If $\uu$ is a $c$-convex solution of \eqref{eq:cMA} on $\U^\l
\subset \U$ open, then $\uu$ is continuously differentiable inside
$\U^\l$.
\end{theorem}

\begin{proof}
Recalling that $c$-convexity implies semiconvexity (see Section \ref{S:background} and also \eqref{eq:semiconvexity}), all we need to
show is that the $c$-subdifferential $\p^{c} \uu(\tx)$ of $\uu$ at
every point $\tx \in \U_\l$ is a singleton.
{\blue Notice that $\p^{c}\uu (\tx) \subset \subset V$ by Theorem~\ref{thm:bdry-inter}(a).}

Assume by contradiction that is not. As $\p^{c} \uu(\tx)$ is
compact, one can find a point $y_0$ in the set $\p^c\uu(\tx)$ such
that $-D_xc(\tx,y_0) \in \p\uu(\tx)$  is an exposed point of the
compact convex set $\p\uu(\tx)$. Similarly to Definition
\ref{D:cost exponential}, we transform $(x, \uu) \longmapsto (q,
\tu)$ with respect to $y_0$, i.e. we consider the transformation
$q \in \cl \U_{y_0} \longmapsto x(q) \in \cl\U$, defined on
$\cl\U_{y_0} = -D_yc(\cl\U,y_0) + D_yc(\tx,y_0)\subset T^*_{y_0}
\V$ by the relation
$$
-D_y c(x(q),y_0)+D_yc(\tx,y_0)=q,
$$
and the modified cost function $ \tilde
c(q,y):=c(x(q),y)-c(x(q),y_0) $ on $\cl \U_{y_0} \times \cl \V$,
for which the $\tc$-convex potential function $q \in \cl \U_{y_0}
\longmapsto \tu(q):=\uu(x(q)) -\uu(\tx)+ c(x(q),y_0) -c(\tx,y_0)$
is level-set convex. We observe that $\tc(q, y_0) \equiv 0$ for all $q$, the
point $\tx$ is sent to $\zero$, $\tu\geq \tu(\zero)=0$, and $\tu$
is strictly $\tc$-convex thanks to Theorem \ref{T:strict convex}.
\begin{figure}
\begin{center}
\centerline{\epsfysize=1.2truein\epsfbox{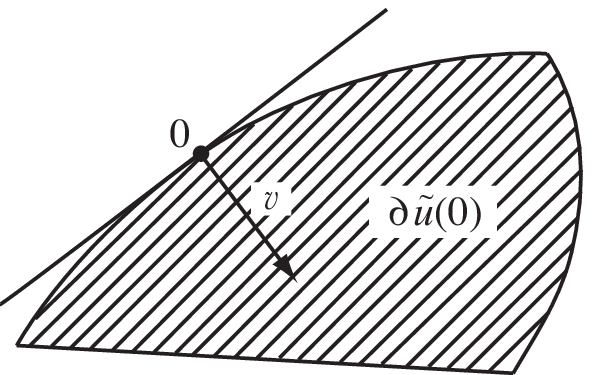}}\caption{{\small $v \in \p\tu(\zero)$ and the hyperplane orthogonal to $v$ is supporting $\p\tu(\zero)$ at
$0$.}}\label{figC1subdiff}
\end{center}
\end{figure}
Moreover, since  $-D_xc(\tx,y_0) \in \p\uu(\tx)$ was an exposed
point of $\p\uu(\tx)$, $\zero=-D_q\tc(\zero,y_0)$ is an exposed point
of $\p\tu(\zero)$. Hence, we can find a vector $v \in \p\tu(\zero)
\setminus\{\zero\}$ such that the hyperplane orthogonal to $v$ is a
supporting hyperplane for $\p\tu(\zero)$ at {\blue $\zero$, see Figure~\ref{figC1subdiff}}.
\begin{figure}
\begin{center}
\centerline{\epsfysize=1.2truein\epsfbox{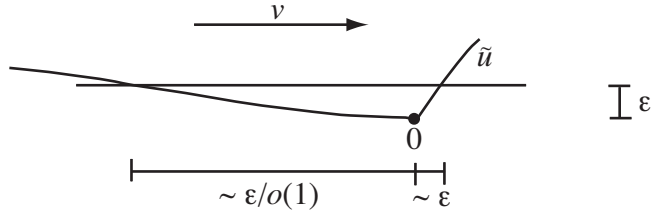}}\caption{{\small Since the hyperplane orthogonal to $v$ is supporting $\p\tu(\zero)$ at $0$, we have
$\tu(-tv)=o(t)$ for $t \geq 0$. Moreover, by the semiconvexity, $\tu$ grows at least
linearly in the direction of $v$.}}\label{figC1sect}
\end{center}
\end{figure}
By the semiconvexity of $\tu$ (see \eqref{eq:semiconvexity}), this implies that
\begin{equation}
\label{eq:good v} \tu( -t v)=o(t) \quad \text{for }t \geq 0,\qquad
\tu(q) \geq \<v, q\>  -M_c|q|^2\quad \text{for all }q \in
\U_{y_0}.
\end{equation}
Let us now consider the (convex) section $K_\e:=\{\tu \leq \e\}$.
Recalling that $\zero$ the unique minimum point with $\tu(\zero)=0$, we have that $K_\e$ shrinks to $\zero$ as $\e\to 0$, and $\tu>0$ on $U_{y_0}\setminus \{\zero\}$.
Thus by \eqref{eq:good v} it is
easily seen that for $\e$ sufficiently small the following hold:
$$
K_\e \subset \{q \mid \<q, v\> \leq 2\e\},\qquad -\a(\e)v \in K_\e,
$$
where $\a(\e)>0$ is a positive constant depending on $\e$ and such
that $\a(\e)/\e \to +\infty$ as $\e \to 0$. Since $\zero$ is the
minimum point of $\tu$, this immediately implies that one between
our Alexandrov estimates \eqref{eq:Alex lower} or \eqref{vardist} must be
violated by $\tu$ inside $K_\e$ for $\e$ sufficiently small, which
is the desired contradiction, {\blue see Figure~\ref{figC1sect}}.
\end{proof}

{\blue \section{Engulfing property and H\"older continuity of
optimal maps}\label{S:H\"older} }

Our goal is to prove the $C^{1,\alpha}_{loc}$ regularity of $\uu$ inside $U^\l$.
The proof will rely on the strict  $c$-convexity of $\uu$ inside $U^\l$
(Theorem \ref{T:strict convex}) and the Alexandrov {\blue type } estimates from Theorems \ref{thm:lower Alex} and \ref{thm:estimate}.
Indeed, these results will enable us to show the engulfing property under the \Bthree\ condition (Theorem~\ref{T:engulfing}),
thus extending the result of Gutierrez and Huang \cite[Theorem 2.2]{guthua} given for the classical Monge-Amp\`ere equation.
Then, the engulfing property allows us to apply the method of Forzani and Maldonado \cite{ForzaniMaldonado04} to obtain local $C^{1, \alpha}$ estimates (Theorem~\ref{T:C1alpha}).
{\red Finally, by a covering argument, we show that the H\"older exponent $\alpha$ depends only on $\lambda$ and $n$, and not on particular cost function (Corollary~\ref{C:universal Holder exponent}).}\\

\begin{remark}\label{R:no-need-B2-B2u}
{\rm We point out that in this section the assumptions \Btwo \ and \Btwos \  are used only to ensure
the strict $c$-convexity (Theorems \ref{T:strict convex}).
% and continuous differentiability of $\uu$ (Theorems \ref{T:strict convex} and \ref{T:continuity}).
 Since the following arguments are performed {\blue locally, i.e. after} restricting to small
 neighborhoods, having already obtained the strict $c$-convexity of $\uu$, the $c$-convexity
 {\blue \Btwo \  and \Btwos\ of the ambient domains are not further required.}
 This is a useful remark for the covering argument in Corollary~\ref{C:universal Holder exponent}.}
\end{remark}

Given points ${\RM \tx} \in U^\l$ and ${\RM \ty} \in \partial^c \uu({\RM \tx})$, and $\tau>0$,
we denote by $S({\RM \tx,\ty},\tau)$ the section
\begin{equation}
\label{eq:S} S({\RM \tx,\ty},\tau):=\{{\RM x} \in U \mid \uu({\RM x}) \leq
\uu({\RM \tx})-c({\RM x,\ty})+c({\RM \tx,\ty})+\tau\}.
\end{equation}
{\green Notice that by the  strict $c$-convexity of $\uu$ (Theorem~\ref{T:strict convex}),
$S({\RM \tx,\ty},\tau) \to \{ {\RM \tx}\}$ as $\tau \to 0$.

In the following, we} assume that all points $x, \tx$ that we choose
inside $U^\l$ are close to each other, and
``relatively far'' from the boundary of $U^\l$,
%(in a universal fashion).
{\blue i.e.
\begin{align*}
 \dist(x, \tx) \ll \min \big( \dist (x, \p U^\l), \dist (\tx, {\green \p}U^\l) \big).
\end{align*}
This assumption {\green ensures} all the relevant sets, i.e.
sections around $x$ or $\tx$, stay strictly inside $U^\l$.}

As we already did many times in the previous section, given a point
${\RM \ty} \in\partial^c \uu({\RM \tx})$ with ${\RM \tx} \in U^\l$,
we consider the transformation $(x, \uu) \longmapsto (q,
\tu)$ with respect to ${\RM \ty}$ (see Definition \ref{D:cost exponential}),
and define the sections
$$
\Q_\tau= \{ q  \in T_{\RM \ty}^* V \mid \tu(q) \leq \tau \}, \qquad \tau \geq 0.
$$
Note that $\Q_\tau$ corresponds to $S({\RM \tx,\ty},\tau)$ under the coordinate
change. By the \Bthree\ condition, each $\Q_\tau$ is a convex set
(Theorem \ref{thm:apparently convex}). We also keep using the notation
$\rho \rdot \Q_\tau$ to denote the dilation of $\Q_\tau$ by a factor
$\rho>0$ with respect to the center of the ellipsoid given by John's
Lemma (see \eqref{E:well-centered}).

{\blue   In our analysis, we will only need to consider the case $\tau \ll 1$. This is useful since,
{\green thanks to the  strict $c$-convexity of $\uu$ (Theorem~\ref{T:strict convex}),}
we can consider in the sequel only sections
%We always assume
%that $t$ is sufficiently small so that all the sections
%which we will consider in the sequel are
%\footnote{RM: deleted ``the, they are, so, )''}
contained inside
$U^\l$ {\green and} sufficiently small. In particular, for every section
$\Q=S(x,y,\tau)$ or $\Q=S({\RM \tx,\ty},K\tau)$, after the transformation in
Definition~\ref{D:cost exponential}:
\begin{enumerate}
\item[-] we can apply Theorem \ref{thm:lower Alex} with $E_\delta=E$ (so \eqref{eq:Alex lower}
holds with $\delta=1$);
\item[-] we can apply Lemma \ref{lemma:bound dual norm} with $\KK=\Q$;
\item[-] Theorem \ref{thm:estimate} hold{\green s}.
\end{enumerate}
}

The following result generalizes \cite[Theorem 2.1(ii)]{guthua} for $c(x,y) = -x\cdot y$ to \Bthree\ costs:

%\footnote{{\YHK I didn't address this except adding "given in this section"} RM: Shouldn't we delete the first red phrase, or else introduce $t_0$ such that $t<t_0$
%implies the desired inclusion; can we make the hypotheses implicit in the second red phrase explicit?
%Also, Robert may want to reconsider the $\cdot$ notation as he mentioned in our discussion.}
%\footnote{AF: What about replacing the first lines with:\\
%``Let $x_0 \in U^\l$, and assume $t$ to be sufficiently
%small (so that in particular $S_\tau = S(x_0,y_0,t) \subset U^\l$). Then there exist
%$0<\rho_0<1$, depending only.........''\\
%Considering all the discussion done before the statement, it looks pretty clear to me and should not create confusion.
%YHK. modified a bit..
%AF: I still prefer my solution: we already said few lines before that all points are in the interior, etc.
%so in some sense the sentence
%``Use the same notation and the assumptions given above in this section."
%looks a bit redundant and even ``confusing'' (since the reader may wonder if he missed some assumption) to me.}

\begin{lemma}[\blue Comparison of sections with different heights]
\label{lemma:dilation section}
{\RM Assume \Bzero-\Bthree\ and let $u$ be a strictly $c$-convex solution to \eqref{eq:cMA}
on $U^\lambda$. Take $\tx \in U^\lambda$, $\ty \in \p^c u(\tx)$ and
$\tau$} sufficiently
small so that $${\green S_\tau =} S({\RM \tx,\ty},\tau) \subset U^\l$$ {\green
and set $\Q_\tau := - D_y c(S_\tau, \ty)$}. Then, there exist $0<{\blue \rho_0}<1$, depending only
on {\blue the dimension $n$ and {\green $\gamma^+_c \gamma^-_c/\l$}
% the dimension, $\l$  and the cost function
 %(but independent of $x_0,y_0$ and $t$), such that
(in particular  independent of ${\RM \tx,\ty}$ and $t$), such that}
$$
\Q_{\tau/2} \subseteq {\blue \rho_0} {\green \rdot} \Q_\tau.
$$
\end{lemma}

\begin{proof}
This is a simple consequence of Theorems \ref{thm:lower Alex} and  \ref{thm:estimate}: indeed,
considering $\tu - \tau$, for $\rho>1/2n$
\eqref{vardist} gives
$$
|\tu(q) -\tau|^n %\lesssim
{\le \blue C(n) \frac{\gamma_\tc^+}{\l}}
%\frac{1}{\l}
(1-\rho)^{2^{-n+1}} \Leb{n}(\Q_\tau)^2 \qquad \forall\, q \in \Q_\tau \setminus\rho \rdot \Q_\tau,
$$
while by \eqref{eq:Alex lower} {\blue (with $\delta =1$)}
% and \eqref{inf|\Q|} imply
 {\blue
 $$
\Leb{n}(\Q_\tau)^2  \le  C(n) \frac{\gamma^-_\tc}{\lambda } \tau^n.
$$
%$$
%\Leb{n} \big(\partial^{\tc} \tu (\Q_\tau)\big) \Leb{n}(\Q_\tau) \sim t^n.
%$$
%Hence, we get
%\begin{align*}
% |\tu (q) -t| & \le C(n) \frac{\gamma^+_\tc \gamma^-_\tc}{\l^2} (1-\rho)^{2^{-n+1}} t^n.
%\end{align*}
%
}
{\green
Hence, we get
\begin{align*}
 |\tu (q) -\tau| & \le C(n) \frac{\gamma^+_\tc \gamma^-_\tc}{\l^2} (1-\rho)^{2^{-n+1}} \tau^n \\
 & \le C(n) \bigg[\frac{\gamma^+_c \gamma^-_c}{\l} \bigg]^2 (1-\rho)^{2^{-n+1}} \tau^n
\end{align*}
where the last inequality follows from $\gamma^+_\tc \gamma^-_\tc/\l^2 \le \big[\gamma^+_c \gamma^-_c/\l \big]^2$ as in  Corollary~\ref{C:Jacobian transform}.}
Therefore, for {\blue $1-\rho_0$} sufficiently small {\blue (depending only on  the dimension
$n$ and {\green $ \gamma^+_c \gamma^-_c/\l$})} we get
$$
|\tu(q) -\tau| \leq \tau/2 \qquad \forall\, q \in  \Q_\tau \setminus {\blue
\rho_0} \rdot \Q_\tau,
$$
so $\Q_{\tau/2} \subseteq {\blue \rho_0} \rdot \Q_\tau$ as desired.
\end{proof}
Thanks to {\blue Lemma~\ref{lemma:dilation section} and  Lemma~\ref{lemma:bound dual norm}}, we can prove the engulfing property under \Bthree, extending  the classical Monge-Amp\`ere case of \cite[Theorem 2.2]{guthua}:

%\footnote{RM: I have reworked this entire theorem and proof.  Shouldn't it be a proposition? YHK: kept as theorem.}
%\footnote{YH: I have put it as a theorem AF: fine with me.}
\begin{theorem}[Engulfing]
\label{T:engulfing}
{\RM Assume \Bzero-\Bthree\ and let $u$ be a strictly $c$-convex solution to \eqref{eq:cMA}
on $U^\lambda$.}
Let $x,\tx \in U^\l$ be close:  {\blue  i.e.}
$${\blue \dist(x, \tx) \ll \min \big( \dist (x, \p U^\l), \dist ({\RJM \tx, \p}U^\l) \big).}$$
Then, there exists a constant $K>1$, depending only on {\blue the dimension $n$ and  {\green $\gamma^+_c \gamma^-_c/\l$},}
 %$\l$ and the cost function,
such that, for all $y \in \partial^c u(x)$, $\ty \in \partial^c u(\tx)$, and $\tau>0$ small
{\blue (so that the relevant sets are in $U^\l$)},
\begin{align}\label{eq:engulfing}
x \in S(\tx,\ty,\tau)\quad \Rightarrow \quad \tx \in S(x,y,K \tau).
\end{align}
\end{theorem}

\begin{proof}
First, {\blue  fix $(\tx,\ty) \in \p^c u$. We} consider the transformation $(x, \uu) \longmapsto (q,
\tu)$ with respect to $\ty$ (see Definition \ref{D:cost exponential}).
Let $\tq {\RJM = -D_x c(\tx,\ty)} \in T_\ty^* V$ denote {\RJM the point corresponding to}
$\tx$ in {\green these} new coordinates.
{\blue To show \eqref{eq:engulfing}  we will find $K>0$  such that }%is equivalent to
\begin{align} \label{engulfRM}
q \in \Q_\tau \  \ \ \Longrightarrow  \ \ \ \tu(\tq) \le \tu (q) +\tc(q, y) - \tc(\tq, y) +   K\tau , \qquad \forall\, y \in \partial^\tc \tu (q).
\end{align}
{\blue   Fix $\tau>0$ small, {\RJM $q \in \Q_\tau$ and $y \in \partial^\tc \tu (q)$}.
{\red Assume by translation that the {\green John} ellipsoid of $\Q_{2\tau}$ is centered at the origin.}}
%\footnote{RM: Notation $\cdot$ may allow us to delete red phrase, if we adapt definition of norm.
%AF: If I understand correcly, with the actual notation, we can now remove the red phrase.}
By  Lemma \ref{lemma:dilation section} we have $\Q_\tau \subset \rho_0 \rdot \Q_{2\tau}$
for some $\rho_0<1$ {\blue (depending only on the dimension $n$ and  {\green $\gamma^+_c \gamma^-_c/\l$)}.}
Hence, by Lemma \ref{lemma:bound dual norm} applied with $\KK=\Q_{2\tau}$ we obtain
\begin{equation}
\label{bound:gradient}
\|D_q \tc(\bar q,y) \|_{\Q_{2\tau}}^* %\lesssim \tau
{\blue \le C\Big(n, {\green \frac{\gamma^+_c \gamma^-_c}{\l}}\Big)\, \tau }
\qquad \forall \, \bar q \in \Q_\tau,
\end{equation}
where {\blue $\|\cdot\|_{\Q_{2\tau}}^*$ } is the dual norm associated to $\Q_{2\tau}$ (see \eqref{eq:dual norm}).

%\footnote{RM: replaced$\tu \geq 0$ and $\tu(\tq) \leq t$; RM: omitting the middle inequality,
%the conclusion is obvious without this normalization}

%{\RJM Adding a constant so that $\Q_\tau = \{ 0 \le \tu \le t\}$
%and recalling $\tq$ minimizes $\tu$,
%{\blue for $K>0$ to be determined below} we have}
%\begin{align*}
%\tu (q)+\tc(q,y)-\tc(\tq,y)  +K t
%& \ge  \tc(q,y)-\tc(\tq,y) +K t\\
%& \geq \tu(\tq)+\tc(q,y)-\tc(\tq,y) + (K-1)t.
%\end{align*}
%Moreover,

{\RJM Recall from Theorem \ref{thm:apparently convex} that $\tq$ minimizes $\tu$,
so $\tu(\tq) \le \tu(q)$.
To obtain \eqref{engulfRM} from this, we need only estimate the difference between the two costs.}
Since both $q,\tq \in \Q_\tau \subset \Q_{2\tau}$, by the definition of $\|\cdot\|_{\Q_{2\tau}}^*$ and \eqref{bound:gradient},
(recalling $sq+(1-s)\tq \in \Q_\tau$ due to convexity of $\Q_\tau$) we find
\begin{align*}
& \bigl|\tc(q,y)-\tc(\tq,y)\bigr|\\
& =\left|\int_0^1 D_q \tc\big(sq+(1-s)\tq, y \big) \,ds \cdot (q-\tq) \right| \\
& \le \left|\int_0^1 D_q \tc\big(sq+(1-s)\tq, y \big)  \cdot q \,ds\right| + \left|\int_0^1 D_q \tc\big(sq+(1-s)\tq, y \big)  \cdot \tq \, ds \right|  \\
& \leq   {\RJM 2} \int_0^1 \| D_q \tc\big(sq+(1-s)\tq, y\big)\|_{\Q_{2\tau}}^*\,ds  \\ &
\le C\Big(n, \frac{\gamma^+_c \gamma^-_c}{\l}\Big)\, \tau \\
&=   K \tau,
\end{align*}
{\RJM thus establishing \eqref{engulfRM}.}
%Hence, if $K$ is sufficiently large {\blue (depending only on the dimension $n$ and  {\green $\gamma^+_c \gamma^-_c/\l$})}, $$\bigl|\tc(q,y)-\tc(\tq,y)\bigr| \leq(K-1)t$$ and we get
%$$
%\tu (q)+ \tc(q,y)-\tc(\tq,y)+K t \geq \tu(\tq)
%$$
%as desired.
\end{proof}

Having established the engulfing property, the $C^{1,\alpha}$ estimates of the potential functions
follows by applying a modified version of Forzani and Maldonado's method \cite{ForzaniMaldonado04}.
Here is a key {\RJM consequence of} {\blue the engulfing property}.

\begin{lemma}[\RJM Gain in $c$-monotonicity due to engulfing]
\label{lemma:eng}
Assume the engulfing property \eqref{eq:engulfing} holds. Let  ${\RM c,u}, x,\tx$ and $K$
be   as in Theorem~\ref{T:engulfing},
and let $y \in \p^c u(x)$ and $\ty \in \p^c u(\tx)$. Then,
\begin{align*}
\frac{1+K}{K} &[u(x)-u(\tx) -c(\tx,\ty)+ c(x,\ty)]\leq c(\tx,y) - c(x,y) -c(\tx,\ty)+ c(x,\ty).
\end{align*}
\end{lemma}

\begin{proof}
Given $x,\tx$, notice that  $u(\tx) - u(x) + c(\tx,y)- c(x,y)\ge 0$. Fix $\e>0$ small, {\RJM to ensure}
$$
\tau:= u(\tx) - u(x) + c(\tx,y)- c(x,y)+\e >0.
$$
Then
$\tx \in S(x,y,\tau)$, which by the {\RJM e}ngulfing property implies $x \in S(\tx,\ty,K\tau)$, that is
$$
u(x) \leq u(\tx)+c(\tx,\ty) - c(x,\ty)+K[ u(\tx) - u(x) + c(\tx,y)- c(x,y)+\e].
$$
Letting $\e \to 0$ and rearranging terms we get
$$
(K+1) u(x) \leq (K+1) u(\tx) + c(\tx,\ty) - c(x,\ty) + K[c(\tx,y)- c(x,y)],
$$
or equivalently
$$
u(x)-u(\tx)
\leq \frac{1}{1+K}[c(\tx,\ty) - c(x,\ty)]+ \frac{K}{1+K}[c(\tx,y)- c(x,y)].
$$
This  gives
$$
u(x)-u(\tx) -c(\tx,\ty)+ c(x,\ty) \leq \frac{K}{1+K} [c(\tx,y) - c(x,y) -c(\tx,\ty)+ c(x,\ty)],
$$
as desired
\end{proof}

In the above lemma, it is crucial to have a factor $(1+K)/K > 1$. Indeed, the above result implies
the desired H\"older continuity of $\uu$, {\blue {\RJM with a H\"older exponent independent of} the
particular choice of $c$ (see Corollary~\ref{C:universal Holder exponent})}:

\begin{theorem}[H\"older continuity of optimal maps to strongly $c$-convex targets]
\label{T:C1alpha} Let $c$ {\green satisfy} \Bzero--\Bthree\ and \Btwos.
If $\uu$ is a $c$-convex solution of \eqref{eq:cMA} on $\U^\l
\subset \U$ open, then $\uu \in C^{1,1/K}_{loc}(\U^\l)$, with $K$ as in Theorem~\ref{T:engulfing} {\blue which depends only on the dimension $n$ and {\green $\gamma^+_c \gamma^-_c/\l$}.}
\end{theorem}

\begin{proof}
As we already pointed out in the previous section, although not strictly needed, we will use the additional information that
$\uu \in C^1(U^\l)$ (Theorem \ref{T:continuity}) to avoid some technical issues in the following proof.
However, it is interesting to point out the argument below works with minor modifications even if $u$ is not $C^1$,
replacing the gradient by subdifferentials (recall that $u$ is semiconvex), and using that semiconvex function are Lipschitz and so differentiable a.e.
We leave the details to the interested reader.

%\footnote{\RJM RM: $x \to x_0$, $z\to z_1$, etc. through this proof}
The proof uses the idea of Forzani and Maldonado~\cite{ForzaniMaldonado04}.
{\RM The $c$-convexity of $u$ is strict on $U^\lambda$, according to
Theorem \ref{T:strict convex}.}
Given a point $x_s \in U^\l$, we denote by
$y_s$ the unique element in ${\RJM \p^c u}(x_s)$; the uniqueness of $y_s$
follows from the $C^1$ regularity of $\uu$, since $y_s$ is uniquely identified by the relation
$\nabla u(x_s)=-D_x c(x_s,y_s)$.

Let $x_0\in U^\l$. {\blue  We will show that for $x_1 \in U^\lambda$ sufficiently close to $x_0$, }
\begin{align*}
|\uu(x_0) - \uu(x_1) -\nabla \uu {\RJM (x_1)} \cdot (x_0-x_1) | \lesssim |x_0-x_1|^{1+ 1/K} ,
\end{align*}
{\red from which the local $C^{1,1/K}$ regularity of $u$ follows by standard arguments.}

Fix a direction $v$ with $|v|$ small, {\RJM set $x_s = x_0 +s v$},
and consider the function
$$
\phi(s):= u(x_s) - u(x_0) {\blue +c(x_s,y_0) - c(x_0,y_0)} {\RJM \ge 0},
$$
for $s \in [0,1]$.
{\blue The idea is to use Lemma~\ref{lemma:eng} to derive a differential inequality, which controls the growth of $\phi$.
First, observe that
}
$$
\phi'(s) = \nabla u(x_s) \cdot v + D_x c(x_s,y_0)\cdot v.
$$
Since $\nabla u(x_s)=-D_x c(x_s,y_s)$, we get
\begin{align*}
\phi'(s)s&=[D_x c(x_s,y_0)-D_x c(x_s,y_s)]\cdot (sv)\\
&\geq c({\RJM x_0},y_s) - c(x_s,y_s) -c({\RJM x_0},y_0)+ c(x_s,y_0) - \|D_{xx}^2 c\|_{L^\infty(U\times V)} s^2|{\RJM v}|^2.
\end{align*}
So, by Lemma \ref{lemma:eng} we get
$$
\frac{1+K}{K} \phi(s) \leq \phi'(s)s + \|D_{xx}^2 c\|_{L^\infty(U\times V)}s^2 {\RJM |v|}^2,
$$
that is
$$
\frac{d}{dt}\left(\frac{\phi(s)}{s^{1+1/K}}\right) \geq - \frac{\|D_{xx}^2 c\|_{L^\infty(U\times V)}{\RJM |v|^2}}{s^{1/K}}.
$$
Hence
$\phi(s)/s^{1+1/K} \leq \phi(1)+ \|D_{xx}^2 c\|_{L^\infty(U\times V)} {\RJM |v|^2}\int_s^1 {\blue \tau^{-1/K}\,d\tau}\leq \phi(1)+ C_1$ (since $1-\frac{1}{K} >0$).
 So,
\begin{align*}
\phi(s)
& {\RJM =} u(x_s) - u({\RJM x_0}) - c({\RJM x_0},y_0)+c(x_s,y_0)\\
&\leq s^{1+1/K} [u({\RJM x_1}) - u({\RJM x_0}) - c({\RJM x_1})+c({\RJM x_0}+v,y_0) +C_1]\\
& {\blue \leq 2C_1 s^{1+1/K}  \hbox{  (choosing $|v|$ small enough and using the continuity of $u$ and $c$)} .}
\end{align*}
By the arbitrariness of $x_0,v$ and $s$ we easily deduce that, for all $x_0,x_1 \in U^\l$ sufficiently close,
$$
u(x_1)- u(x_0) - c(x_0,y_0)+c(x_1,y_0) \leq 2C_1 |x_0-x_1|^{1+1/K}.
$$
Since $c$ is smooth, $\nabla u(x_1)= -D_x c(x_1,y_1)$, and $u(x_0)- u(x_1) - c(x_1,y_0)+c(x_0,y_0)\geq 0$, the last inequality implies
\begin{align*}
&|u(x_0)-u(x_1) - \nabla u(x_1)\cdot (x_0-x_1)| \\
& \le  | u (x_0) - u(x_1) -c(x_1, y_0) + c(x_0, y_0) | + \|D_{xx}^2 c\|_{L^\infty(U\times V)} |x_0-x_1|^2\\
& \leq C_2 (|x_0-x_1|^{1+1/K} + |x_0-x_1|^2) \\
& \le 2 C_2 |x_0-x_1|^{1+1/K}
\end{align*}
for all $x_0,x_1 \in U^\l$ sufficiently close.
This proves the desired estimate, and concludes the proof of the
 $C^{1,1/K}_{loc}$ regularity of $\uu$ inside $U^\l$.
\end{proof}

{\blue
In fact, the H\"older exponent in the previous theorem does not depend on the particular cost function:
{\RJM
\begin{corollary}[Universal H\"older exponent]\label{C:universal Holder exponent}
\blue With the same notation and assumptions as in Theorem~\ref{T:C1alpha}, $\uu \in C^{1, \alpha}_{loc} (U^\l)$, where
the H\"older exponent $\alpha>0$
depends only on $n$ and $\l>0$.
\end{corollary}
}
%
%\footnote{\RJM RM: Replacing $2$ by $1+1/k$, Does it actually yield an exponent arbitrarily close to the one \cite{ForzaniMaldonado04}
%obtained for the bilinear cost?  YHK: I am not sure. It is not clear......
%AF: It doesn't look to me, as there are many places where we lose, for instance in the estimates
%where we use the dual norm (see for instance (9.4)). I think we should keep 2.}
\begin{proof}
 Recalling
$\gamma^\pm_c = \gamma^\pm_c(\U^\l \times \V) := \| (\det D^2_{xy} c)^{\pm 1}\|_{L^\infty(\U^\l \times \V)}$
from \eqref{Jacobian bound}, we see $\gamma^+_c \to 1/\gamma_c^-$
as the set $\U^\l \times \V$ shrinks to a point. Since $u$ is
$C^1$ by Theorem \ref{T:C1alpha}, $\partial^c u$ gives a single-valued continuous map.
Compactness of $\cl \U^\l$ combined with \Bzero-\Bone\ allows the
set $\partial^c u \cap (\cl U^\l \times \cl \V) $ to be covered
with finitely many neighborhoods $\U_k \times \V_k$,  such that
$\gamma^+_c(\U_k \times \V_k) \gamma^-_c(\U_k \times \V_k) \leq 2$
 for all $k$.
%{\red Moreover, by choosing both $\U_k$ and $\V_k$
% to be sufficiently small balls, condition \Btwos\ holds (see for instance \cite{MaTrudingerWang05}
% \cite{TrudingerWang08p}).}
Hence, {\green thanks to Remark~ \ref{R:no-need-B2-B2u}, we can apply Theorem \ref{T:C1alpha}  on each such neighborhood $(U^\l
\cap U_k)$, which in turn} yields a H\"older exponent $0<\alpha <1$
depending only on $n$ and $\l>0$.
\end{proof}
}

\end{document}